\documentclass[11pt]{article}

\newcommand{\guanyang}[1]{\textcolor{brown}{[GW: \emph{#1}]}}

\usepackage[a4paper, margin=1in]{geometry}

\usepackage{algorithm}
\usepackage{amssymb, amsmath, amsthm}
\usepackage{breakcites}
\usepackage[normalem]{ulem}
\usepackage{bigstrut} %needed to keep blockarrays from being crunched
\usepackage{blkarray}
\usepackage{booktabs}
\usepackage{enumitem} % detailed list control
\usepackage{float} % for 'H' forced placement of floats
\usepackage{graphicx}
\usepackage[colorlinks=true, allcolors=LinkBlue, backref=page]{hyperref}
\usepackage{lmodern}
\usepackage{mathrsfs} % for mathscr
\usepackage{multirow}
\usepackage{natbib} % for \citet and \citep
\usepackage{scalerel}
\usepackage{setspace}
\usepackage{soul} % for \st = strikethrough
\usepackage{subcaption}
\usepackage[T1]{fontenc}
\usepackage{titling}
\usepackage{verbatim}
\usepackage[dvipsnames]{xcolor}
\usepackage{xfrac}
\usepackage{tikz}
\usetikzlibrary{graphs,graphs.standard}
\usepackage{scalerel}

\usetikzlibrary{arrows.meta}

\usepackage{outlines}
\usepackage{titlesec}
\setlist{noitemsep, topsep=0pt} % for enumitem
\definecolor{LinkBlue}{rgb}{.15, .25, .85} %for hyperref

\binoppenalty=\maxdimen % prevent linebreaks after binary operators, usually 700
\relpenalty=\maxdimen % prevent linebreaks after relations, usually 500

\makeatletter
\renewcommand*{\NAT@spacechar}{~}
\makeatother

\newcommand{\be}{\begin{equation}}
\newcommand{\ee}{\end{equation}}
\newcommand{\ba}{\begin{array}}
\newcommand{\ea}{\end{array}}
\newcommand{\bea}{\begin{eqnarray}}
\newcommand{\eea}{\end{eqnarray}}

\newcommand{\calO}{{\cal O }}

\newcommand{\EE}{\mathbb{E}}
\newcommand{\ZZ}{\mathbb{Z}}

\newcommand{\bR}{\mathbb R}

\newcommand{\bP}{\mathbb P}

\newcommand{\bE}{\mathbb E}

\newtheorem{dfn}{Definition}
\newtheorem{prop}{Proposition}
\newtheorem{eg}{Example}

\newtheorem{lem}{Lemma}
\newtheorem{cor}{Corollary}
\newtheorem{rem}{Remark}

\newtheorem{conj}{Conjecture}
\newtheorem{thm}{Theorem}
\newtheorem*{thm*}{Theorem}

\newcommand{\tcov}{t_{\text{cov}}}

\DeclareMathOperator*{\argmax}{arg\,max}
\DeclareMathOperator*{\argmin}{arg\,min}

\begin{document}

\title{Repeated Averages on Graphs}

\author{Ramis Movassagh\\ ramis@us.ibm.com  \thanks{IBM Quantum, MIT-IBM Watson AI Lab, Cambridge, U.S.A.} 
	\and Mario Szegedy \\szegedy@cs.rutgers.edu\thanks{Department of Computer Science, Rutgers University, New Brunswick,  USA.
	}
	\and Guanyang Wang \\guanyang.wang@rutgers.edu\thanks{Department of Statistics, Rutgers University, New Brunswick,  USA.}
}
\date{}
\maketitle
\thispagestyle{empty}
\begin{abstract} 
Sourav Chatterjee, Persi Diaconis, Allan Sly and Lingfu Zhang (\citep{chatterjee2019phase}), prompted by a question of Ramis Movassagh, renewed the study of a
process proposed in the early 1980s
by Jean Bourgain. A state vector $v \in \mathbb R^n$,
labeled with the vertices of a 
connected graph, $G$,
changes in discrete time steps following the simple rule that
at each step a random edge $(i,j)$ is picked and  $v_i$ and $v_j$ are both replaced by their average $(v_i+v_j)/2$. It is easy to see that the value associated with each vertex converges to $1/n$. The question was how quickly will $v$ be
$\epsilon$-close to uniform in the 
$L^{1}$ norm in the case of 
the complete graph, $K_{n}$,
when $v$ is initialized as a standard basis vector
that takes the value 1 on one coordinate, and zeros everywhere else.
They have established a 
sharp cutoff of $\frac{1}{2\log 2}n\log n + O(n\sqrt{\log n})$.
Our main result is to prove, that 
% $\frac{((1-\epsilon)\log n -\frac{1}{e\log2})n}{2\log2}$ 
$\frac{(1-\epsilon)}{2\log2}n\log n-O(n)$ is a general
lower bound for all connected graphs on $n$ nodes.
We also get sharp magnitude of $t_{\epsilon,1}$ for several important families of graphs, including star, expander, dumbbell, and  cycle. In order to establish our results we make several observations 
about the process,
such as the worst case initialization 
is always a standard basis vector. 
Our results add to the body of work of
\cite{aldous1989lower,aldous2012lecture,quattropani2021mixing,cao2021explicit,olshevsky2009convergence}, and others. The renewed interest is 
due to an analogy to a question
related to the Google's supremacy circuit.
For the proof of our main theorem we employ a concept that we call {\em augmented entropy function} which may find independent interest in the computer science and probability theory communities.
\end{abstract} 

\newpage
 \thispagestyle{empty}
\hypersetup{linkcolor=black}
 \tableofcontents

\newpage
%\ramis{ Ramis's comments}
%\mario{ Mario's comments}
%\guanyang{Guanyang comments}
%\setcounter{page}{1}
\section{Introduction}\label{Sec:Intro}
\subsection{The Averaging Process}

Let $G=\left(V(G),E(G)\right)$ be an undirected, 
connected finite graph. The averaging process 
on $G$ proceeds in discrete time steps,
$t=0,1,\ldots$.
At $t=0$ there is a real number
associated with each node of the graph,
defining the initial vector, $v(0)$.
In each step, $t= 1,2,\ldots$,
we pick an edge {\it randomly and uniformly} from $E$ and replace the values associated with both 
of its end points of this edge by their average. 
Let $(v_1,\dots,v_n)^T$ denote the 
state vector at a given time. 
If in the next step an edge $(i,j)\in E$ is picked, the new state vector becomes
\[
\left(v_1,\; \dots,\; v_{i-1},\; \frac{v_i+v_j}{2},v_{i+1},\; \dots,v_{j-1},\; \frac{v_i+v_j}{2},\; v_{j+1},\;\dots,\; v_n\right)^T.
\]

Let $v(t)=(v_{1}(t),\dots,v_{n}(t))^T$
be the value of $v$ in the $t^{\rm th}$ step
($t=0,1,2,\ldots$). The expression
$\frac{1}{n}\sum_{i=1}^n v_{i}(t)$
remains invariant in every branch of the process. Let
\[
{\bar{v}}\equiv(a,\dots,a)^T, \;\;\;\;\; {\rm where}
\;\;\;\;\;
a\equiv\frac{1}{n} \sum_{i=1}^n v_{i}(0)
\]
An easy argument shows that
for large enough $t$ the state vector $v(t)$ will be 
very likely in the proximity of
${\bar{v}}$.
Most results, including ours, investigate the 
speed of convergence to ${\bar{v}}$. 
This speed depends on {\bf ($i$)}  the graph $G$, {\bf ($ii$)}  the initial vector $v(0)$, and {\bf ($iii$)} the notion of convergence. The convergence in the $L^{2}$ 
norm is easy to analyse, but it is less 
interesting than the convergence in the $L^{1}$ norm,
which is the main subject of our study.
We will also study $2\rightarrow 1$
convergence defined below. To capture
the quantities we are interested in, we have developed
the following general definition:

\begin{dfn}[$\epsilon$-mixing time for $L^q$ metric under the unit $L^p$ sphere initialization]\label{def:tpq}
$$t_{\epsilon,\,p\rightarrow q}(G) \equiv\min
\left\{t\in \mathbb{N}\; \middle\vert \;
\scaleto{\forall\mathstrut}{15pt} \;  v(0) \;{\rm with} \;
\lVert v(0)\rVert_p = 1 :\; 
\sqrt[\scaleto{q\mathstrut}{8pt}]
{\bE\left[\lVert v(t) - \bar v\rVert_q^q\right]} \leq \epsilon\right\},$$
where $\lVert w \rVert_q$ is 
$\sqrt[\scaleto{q\mathstrut}{8pt}]
{\sum w_{i}^{q}}$, the $L^{q}$
norm.
\end{dfn}

We may omit $G$ in the argument when clear from the context. Further, when $p=q$, we simply say ``$L^p$ convergence,'' ``$L^p$ mixing time'' and we write ``$t_{\epsilon,p}(G)$.'' 
Notice that $\lVert v(t) - \bar v\rVert_q^q$ is a random variable, and Definition 
\ref{def:tpq} focuses on its expectation.

Let $A(G)$ be the adjacency matrix of $G$ and $D(G)$  
be the diagonal matrix formed from the 
vertex degrees of $G$. 
The spectrum of the {\em graph 
Laplacian}, $D(G) - A(G)$, is a well-known aid in
analysing
random walks on $G$. 
All eigenvalues of $L(G)$ are non-negative,
the smallest of which is zero, while the 
second smallest, $\lambda_2(G)$
is a controlling parameter that roughly tells
how quickly the random walk on $G$ mixes.
In our case the following parameter
will be central:
\begin{dfn}\label{def:gamma} For a connected, undirected graph $G = (V,E)$, we define
\[
\gamma(G) = \frac{|E(G)|}{\lambda_2(G)}.
\]
\end{dfn}

Since $\lambda_2(G) > 0$ for every connected graph
\cite{fiedler1973algebraic}, we have 
$\gamma(G)<\infty$. 

\subsection{Main results}

Throughout we will use
the notation $ f(x,\epsilon) = 
\Theta_\epsilon (g(x,\epsilon))$ 
when constants $c_\epsilon, C_\epsilon$ exist 
that depend only on $\epsilon$ exist
such that
$ c_\epsilon g(x,\epsilon)
\leq f(x,\epsilon) \leq C_\epsilon g(x,\epsilon)$.
We may omit the $\epsilon$ subscipt
when $c_\epsilon$ and $C_\epsilon$
need not depend on $\epsilon$.

We compute or estimate the mixing times $t_{\epsilon,1}$ and $t_{\epsilon,2\rightarrow1}$  (see Definition \ref{def:tpq}). The $L^1$ convergence, $t_{\epsilon,1}$, is the focus of 
most other works.

For all norms the quantity $\gamma(G)$ defined in Definition \ref{def:gamma} determines convergence 
up-to a logarithmic factor, and the question is when this factor or some part of it is present. 
It has been 
observed earlier by several authors, that 
$t_{\epsilon,2}(G) = \Theta(\gamma(G))$
for every fixed $0< \epsilon< 1$, and the 
constant factor is not large (about $2$). 
We show that the $t_{\epsilon,2\rightarrow 1}$ mixing time is also $\Theta(\gamma(G))$
when the graph Laplacian $L(G)$ admits a delocalized Fiedler vector -- the eigenvector 
the corresponds to the second smallest eigenvalue
(Theorem \ref{thm: t_12,delocalize}).
For the $L^1$ mixing time we prove:

\begin{thm}\label{thm:genlb}
For any connected graph $G$ with $n$ nodes, the $\epsilon$-mixing time satisfies 

$$t_{\epsilon,1}(G) \geq\frac{(1-\epsilon)}{2\log2}n\log n-O(n)$$ 
\end{thm}

Our bound does not only beat 
the standard argument's 
general $\Omega(n)$ lower bound,
but matches up-to a factor of 
$1-\epsilon$ with the sharp threshold of 
Sourav Chatterjee, Persi Diaconis, Allan Sly and Lingfu Zhang (\citep{chatterjee2019phase})
for the complete graph. Thus our result essentially shows that the complete graph asymptotically mixes  fastest in the $L^1$ metric.

We conjecture, that
$t_{\epsilon,1}(G) = \Theta_\epsilon(\max\{\gamma(G), n\log n\})$
(Conjecture \ref{conj:L1 mixing}).
Towards the conjecture we derive several useful results on general graphs. In particular, the following result gives the slowest initialization under the $L^1$-metric, which can be of use in future research.

\begin{figure}
\begin{centering}
\begin{tabular}{ll}
\includegraphics[scale=0.5]{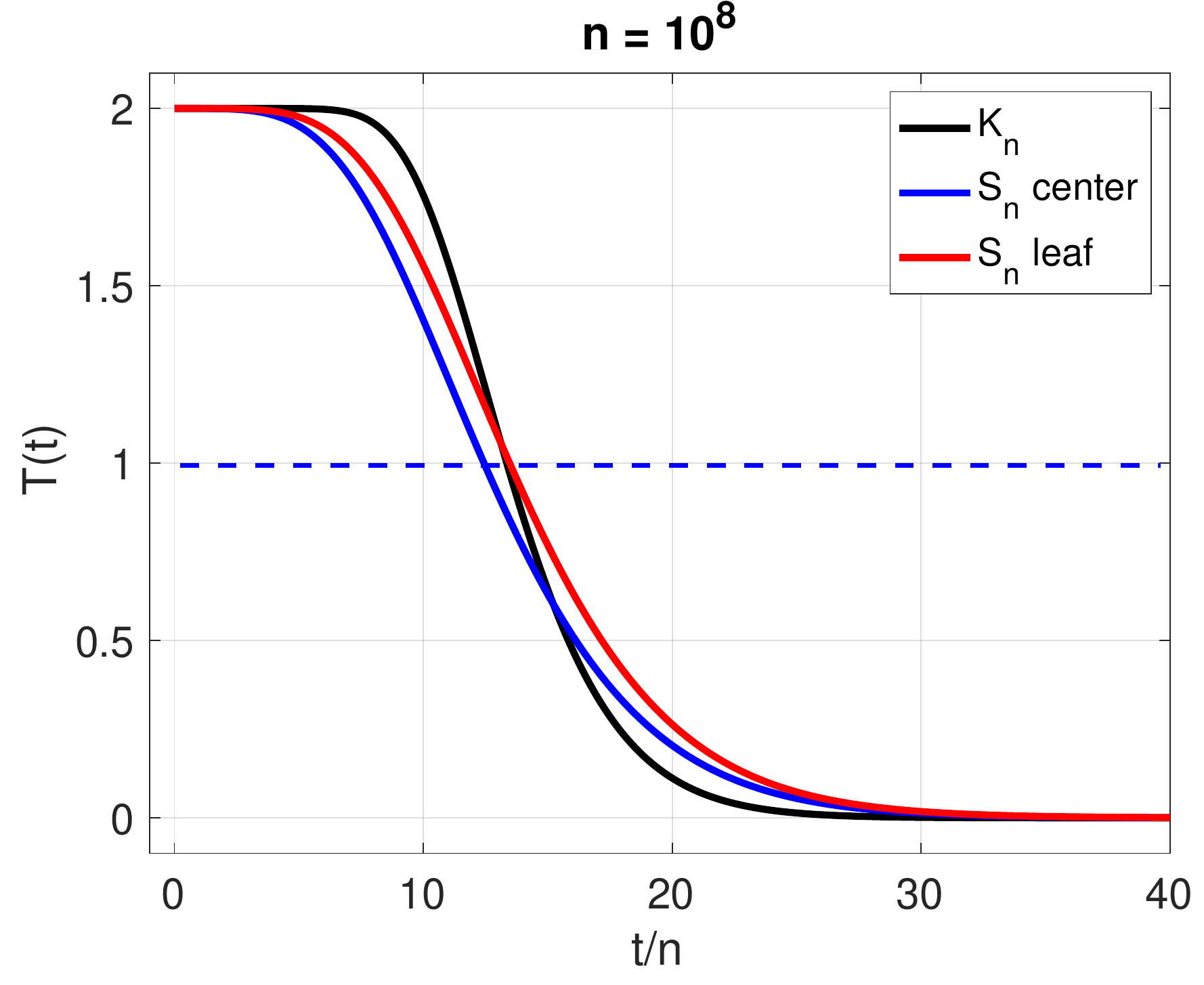}
& 
\includegraphics[scale=0.4]{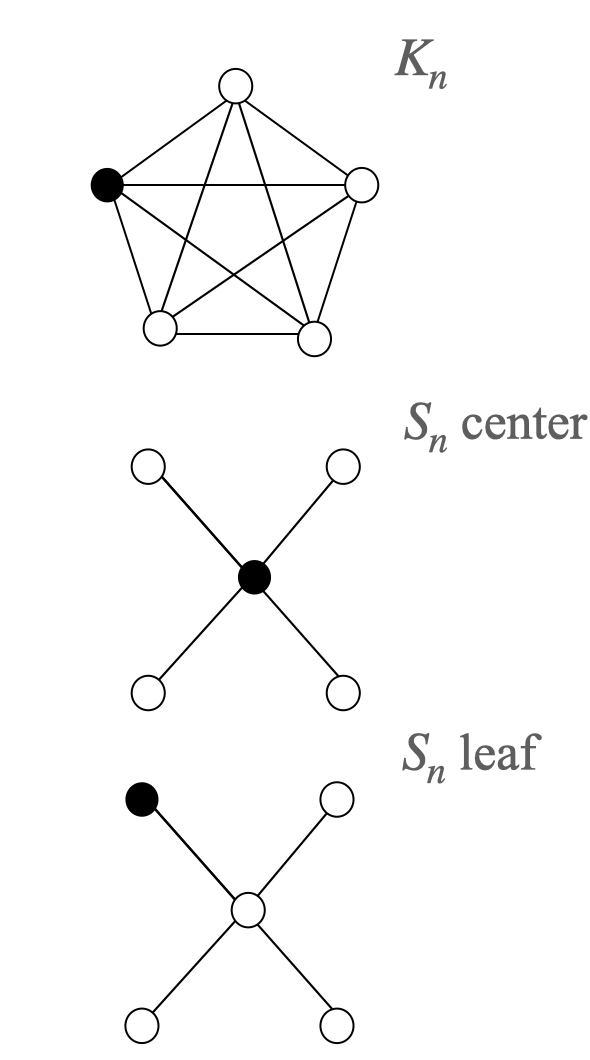}
\end{tabular}
\par\end{centering}
\caption{\label{fig:ComStars_exp}Expected $L^1$ convergence 
of the Averaging Process for the 100,000,000 node
$K_{n}$ and for the star graph, $S_{n}$,
of the same size
with two different starting states
(averaged over $20$ runs in each 
of the three cases; $T(t)$ estimates the 
expectation $\bE[||v(t)-\frac{1}{n}{\bf 1}||_{1}]$
for the respective graphs).
Interestingly, the star graph may slightly 
beat $K_{n}$, when the initial vector
is $1$ at the center, although
our theorem implies, that even when this is the case, it must be a lower order advantage.}
\end{figure}

\begin{thm}[Slowest initialization is at the corner]\label{thm: L1, slowest initialization}
Let $G$ be a graph with $n$ nodes. For every fixed $t$, the slowest initialization:
\[
\argmax_{v(0): \lVert  v(0) \rVert_1 = 1} \bE \lVert v(t) - \bar v \rVert_1 = e_i
\]
for some $i$, where 
\[
e_i = (0, \cdots,0, 1, 0,\cdots, 0)
\]
only has its $i$-th coordinate equals $1$. 
\end{thm}

 Consequently, we get sharp magnitude of $t_{\epsilon,1}$ for several important families of graphs, with results summarized below. All our existing results support Conjecture \ref{conj:L1 mixing}. Our result on cycles resolve a conjecture in \cite{spiro2020averaging}, independently from \cite{quattropani2021mixing}, and with  different techniques.

 \begin{table}[htbp]
\begin{center}
\begin{tabular}{c|c|c|c|c}
\toprule
    {\bf Graph} & Bounded deg. expander   & Star  &  Dumbbell   & Cycle\\[4pt]
\midrule
$\boldsymbol{t_{\epsilon, 1}}$      &     $\Theta_\epsilon(n\log n)$ (Prop \ref{prop:expander})&    $\Theta_\epsilon(n\log n)$ (Cor \ref{cor:star})           &      $\Theta_\epsilon(n^3)$ (Thm \ref{thm: dumbbell})           &         $\Theta_\epsilon(n^3)$ (Thm \ref{thm:cycle}) \\[4pt]
\bottomrule
\end{tabular}
\caption{The $L^1$ mixing time for selected graphs (each with $n$ nodes).}
\label{tab:L1}
\end{center}
\end{table}

This paper provides several new techniques for the study of the averaging process. For example, the universal $\Omega(n\log n)$ lower bound relies on 
a novel {\it augmented entropy function} (see Definition \ref{def:AugmentedEntropy}). The  mixing time  for several families of graphs is based on `flow', `comparison' and `splitting' techniques that are introduced and developed in this paper.  These techniques may be of independent interest, and 
a utility for further investigations.

\subsection{Open problems}

%In particular, combining our results with the result on the complete graph in \cite{chatterjee2019phase}, we show that the complete graph asymptotically mixes fastest over all graphs under the $L^1$ metric. 
Regarding the $L^1$ mixing time, evidences 
point towards
\begin{conj}\label{conj:L1 mixing}
Let $G$ be any connected graph with $n$ nodes, 
\[\qquad t_{\epsilon,1}(G) = \Theta_\epsilon(\max\{\gamma(G), n\log n\})
\] 
\end{conj}
We also think that in
the case of $t_{\epsilon,2\rightarrow1}$,
which is  immediate between
$t_{\epsilon,2}(G)$ and
$t_{\epsilon,1}(G)$,
the $\log n$
factor is always present:
\begin{conj}\label{conj:L12 mixing}
Let $G$ be any connected graph with $n$ nodes,
\[
\qquad t_{\epsilon,2\rightarrow1} =  \Theta_\epsilon(\gamma(G)\log n )
\]
\end{conj}

It is natural to inquire the cut-off phenomenon for different families of graphs. In \citep{chatterjee2019phase}, the authors  prove a cut-off phenomenon for the complete graph. It is also proved in \cite{quattropani2021mixing} that certain families of graphs do not exhibit cut-offs.   Are there other graphs for which one can prove or disprove a cut-off phenomenon, for instance for star graphs?

\subsection{Organization}
The result of this paper is organized as follows. 
After motivation and a discussion of previous work in Section \ref{sec:motivation},
we state and prove  the convergence speed of the averaging process under  $L^2$, $L^{1}$, and $L^{2\rightarrow 1}$ metrics in Section \ref{subsec:L2},  \ref{subsec:L1} and \ref{subsec:L2->1}, respectively. Some numerical explorations are provided in Section \ref{sec: NumericalExp}. We introduce a related random process and prove  the extremity of the complete graph in Section \ref{sec:MC2}. Useful properties of the averaging process, and some technical proofs are provided in Section \ref{sec:proof}.

\section{Motivation and previous work}\label{sec:motivation}

Repeated averages (an alternative term for the averaging process in some literature) is a simple model for reaching equilibrium in statistical physics. Suppose each vertex value $v_i$ is a local temperature. We can envision an interaction that brings this system to thermal equilibrium. In much of statistical physics the equilibrium is reached via intermediate partial equilibria, where local patches reach an equilibrium and then each patch equilibrates with neighboring patches till full equilibrium is reached \cite[Vol. 5, Section 4]{landau2013course}. Therefore, it is more realistic to consider the averaging process with respect to a locality constraint of the underlying system such as a graph structure or a lattice.

Another motivation comes from a toy model of quantum supremacy circuits, in which each gate enacts a unitary matrix drawn independently from the Haar measure. This random circuit family itself is an abstraction of Google's Sycamore circuit~\cite{boixo2018characterizing,arute2019quantum}. Despite intensive research it remains open to prove that upon measurement and for a randomly selected member of the above family (with appropriate depth and connectivity parameters) the output distribution of the $2^{n}$ possible output strings approaches what is known as the Porter-Thomas distribution.

The effect of a Haar-local unitary is that regardless of the input to each gate, the output probability is the same for all possible outcomes. A key question is: how many such local gates should be applied for the overall output to reach the Porter-Thomas distribution? This crucially depends on the so-called architecture of the quantum circuit, which is the underlying graph of connectivity that dictates the position of the gates in the course of the computation. Therefore, both the power of a quantum circuit and its convergence to the desired output are dependent on the graph that defines the circuit architecture. The averaging process is a classical abstraction of these random quantum circuits.
\begin{figure}
\begin{centering}
\includegraphics[scale=0.4]{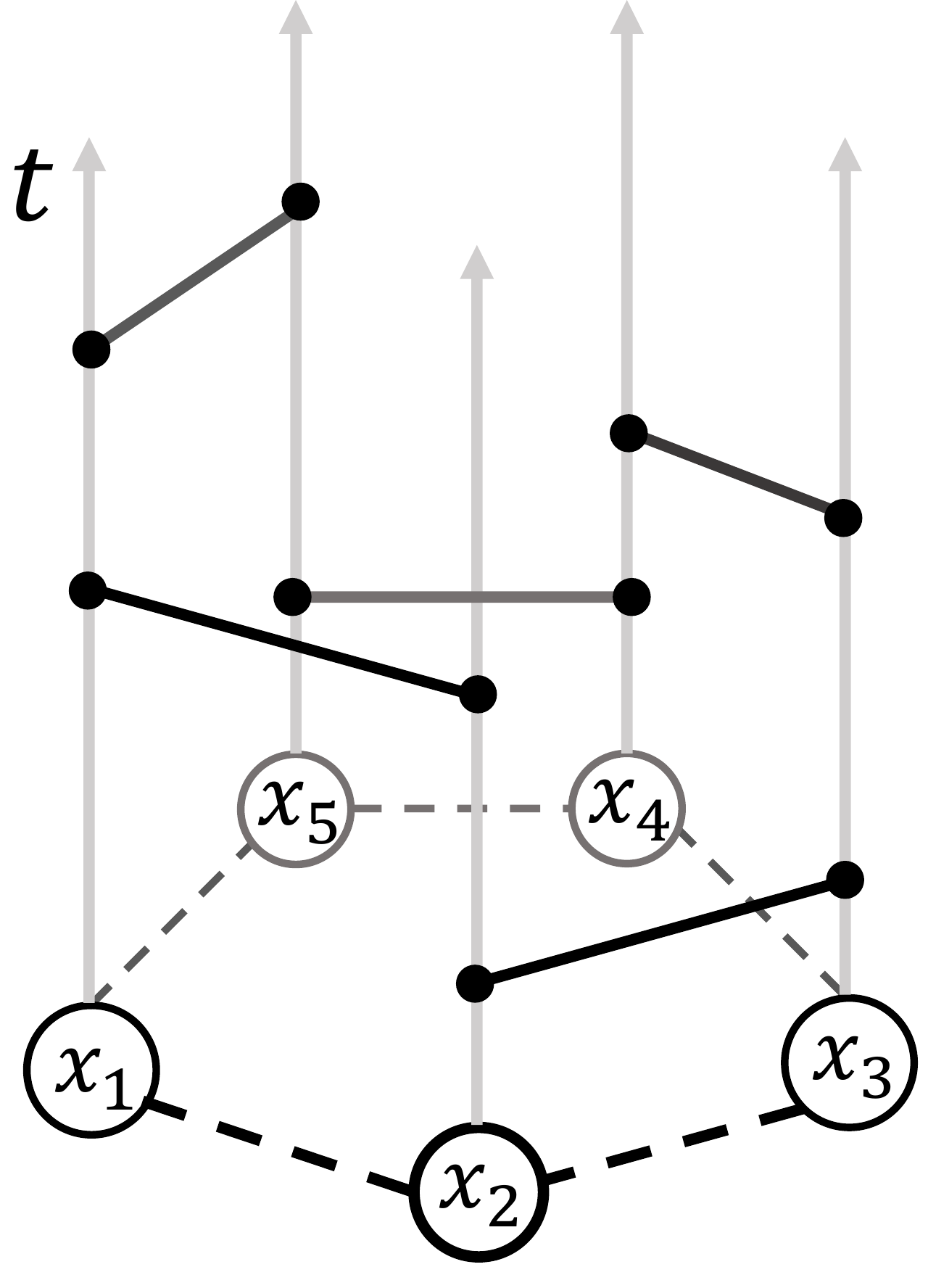}\qquad\qquad\includegraphics[scale=0.38]{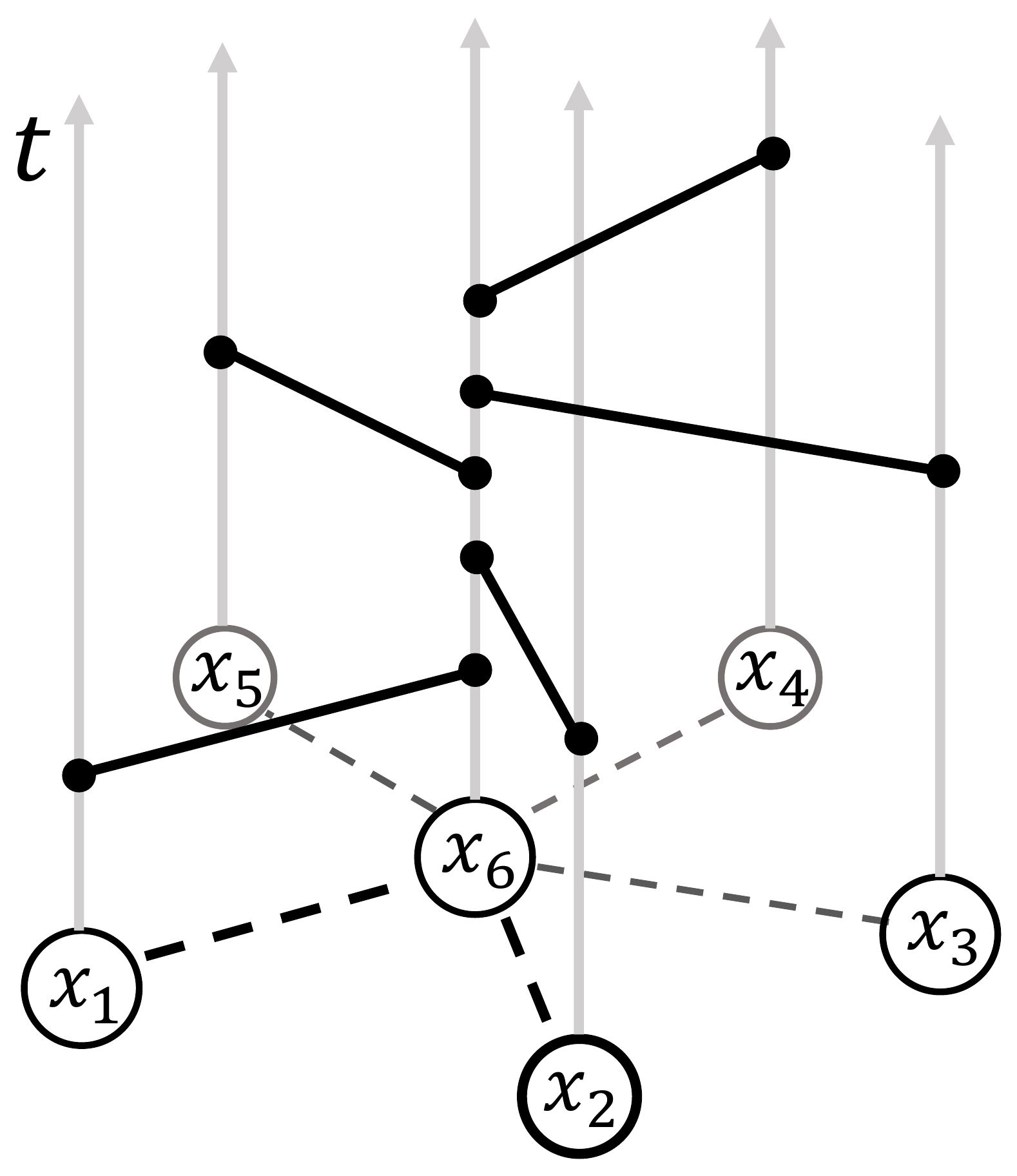}\\
\par\end{centering}
\caption{\label{fig:ClassSwap}Schematics of the classical randomized swap gate circuit. Left: Circuit has a cycle architecture. Right: Circuit has a star architecture.}
\end{figure}

Our toy model is even simpler. It is a family of {\em classical} randomized circuits with 2-bit randomized swap gates. Each gate independently flips an unbiased coin. If Heads, the gate does not change the two bits involved. If Tails, the two bits are swapped. Given a graph, $G$, and a depth parameter $m$, we can define a randomized circuit family ${\cal F}(G,m)$ as follows. The nodes of $G$ correspond to input bits. Edges of $G$ represent bit-pairs where randomized swap gates can be placed. Example circuits for the star and circle graphs are shown in Figure \ref{fig:ClassSwap}. In the family ${\cal F}(G,m)$ each member of the family performs a sequence of $m$ gates $g_{1},\ldots,g_{m}$, where $g_{i}$ is randomly selected from $E(G)$. Assume now that we start with a distribution $p_{1}, \ldots, p_{n}$ on inputs $e_{1}, \ldots, e_{n}$, where $e_{i} =( 
0\cdots 0, \underbrace{1}_{i}, 0 \cdots 0)^T$.
The output is also a distribution on strings of Hamming weight one, since the application of each gate leaves the Hamming weight unchanged. For a random member of the family ${\cal F}(G,m)$ the output-distribution vector, $q$, distributes in the same way as $v(m)$ of the averaging process (in both cases we talk about distribution on vectors), where $v(0) = (p_{1}, \ldots, p_{n})^T$, and the underlying graph is $G$.
Further motivation for the averaging process can be found in Section 1.1 of \cite{chatterjee2019phase}.\\

The averaging process we study 
has been the subject of various works. Most notably, \cite{aldous2012lecture} 
give a survey entitled ``Lectures on the averaging process,'' 
where the problem is 
expressed in the language of continuous-time Markov chains.
Convergence in $L^2$ and entropy
are investigated in connection with the spectral 
gap.  \cite{spiro2020averaging} derives upper bounds for the $L^2$ convergence of the averaging process on hyper-graphs.

\citep{chatterjee2019phase} focused on the $L^{1}$ convergence in the special case of the complete graph. 
They succeeded in showing a sharp cutoff phenomenon answering  a conjecture of Jean Bourgain and a question of Ramis Movassagh. They proved that the $L^1$ convergence happens in time ${1\over 2\log 2}n\log n$ and the cutoff window is of magnitude $n\sqrt{\log n}$. 

 \cite{quattropani2021mixing} considers a dual problem to averaging and characterizes the $L^1$ mixing time. They analyze the cutoff phenomena with respect to the number of random walkers on the graph.

Lastly, \cite{cao2021explicit} derives the convergence to the uniform state of the averaging process on the complete graph from the exponential decay of the ``Gini index'' to zero in the continuum hydrodynamic limit.

There are a number of other works that investigate processes closely related to the averaging process. For example, \cite{olshevsky2009convergence} consider the averaging and consensus finding problems, where they mainly focus on the expected state vector under the averaging process.
Shah's 127 pages manuscript \cite{shah2009gossip} on Gossip Algorithms derives mixing times and conductance for various graphs.
 \cite{haggstrom2012pairwise} considers the Deffuant model in the context of social interactions, where the graph has vertices on the integers with nearest-neighbor interactions. The underlying process averages the randomly chosen adjacent nodes with the possibility of ``censorship''.  
They characterize the compatibility of social interactions as a function of the parameters of the underlying process.  \citep{jain2020fast} consider the Kac walk with the aim of finding efficient algorithms for dimension reduction. For a set of special parameters their construction resembles the averaging process.

\section{Convergence for the averaging process}\label{sec: convergence}

This section includes our theoretical results. Some detailed proofs are deferred to the Appendix.  Preliminary useful properties that we will directly use here are also proved in the Appendix (Section \ref{subsec:general}). 
\subsection{$L^2$ convergence}\label{subsec:L2}

The $L^{2}$ convergence of the process is standard to analyze. We have that for every connected 
 graph $G$ with $n$ nodes, 
 \[
  (2\gamma(G)-1)\log(\epsilon^{-1})\le t_{\epsilon,2}(G) \; \le\; 4\gamma(G)\log(\epsilon^{-1})
 \]
 whereas,
  \begin{equation}\label{eq:l1l21}
 (2\gamma(G) - 1) \log  (\epsilon^{-1})\; \le\; \{t_{\epsilon, 1}(G), \; t_{\epsilon,2\rightarrow 1}(G)\} 
 \; \le\;  4\gamma(G)\log (\sqrt n \epsilon^{-1})
 \end{equation}

We notice that replacing $(v_{i},v_{j})$ by $\left(\frac{v_{i}+v_{j}}{2},\frac{v_{i}+v_{j}}{2}\right)$
decreases $\mathbb{E}\left[\left\Vert v(t)-\bar{v}\right\Vert _{2}^2\right]$ exactly by $(v_{i}-v_{j})^{2}/2$. Based on this observation we have the following theorem. The proof details can be found in Section \ref{subsubsec: proof L2}. The upper bound for Theorem \ref{Thm:L2_convg} appears in \cite{aldous2012lecture}.
\begin{thm}
\label{Thm:L2_convg}Let $v(0)\in\mathbb{R}^{n}$ be an arbitrary
initial vector. Then
the expected $L^{2}$ distance satisfies
\begin{equation}
\sqrt{\mathbb{E}\left[\left\Vert v(t)-\bar{v}\right\Vert _{2}^2\right]}\le\left(1-\frac{1}{2\gamma(G)}\right)^{t/2}\left\Vert \bar{v}-v(0)\right\Vert _{2}\quad.\label{eq:thmL2_Upp}
\end{equation}
Furthermore, the slowest convergence speed is at least $\left(1-1/(2\gamma(G))\right)^{t}:$
\begin{equation}
\sup_{\left\Vert v(0)\right\Vert _{2}=1}\sqrt{\mathbb{E}\left[\left\Vert v(t)-\bar{v}\right\Vert _{2}^2\right]}\ge\left(1-\frac{1}{2\gamma(G)}\right)^{t}\quad.\label{eq:thmL2_Low}
\end{equation}
\end{thm}

We are in the position to bound the mixing time, which is the time
it takes for the distribution of the state vector to be $\epsilon$-close
to the uniform vector with respect to the norm of interest. The next corollary shows the $\epsilon$-mixing time under the $L^2$ metric is  $\Theta(\gamma(G)\log(\epsilon^{-1}))$. The proof details can be found in Section \ref{subsubsec: corL2}.
\begin{cor}
\label{Cor:L2mixing}($L^{2}$ mixing time)
\begin{equation}
(2\gamma(G) - 1)\log\epsilon^{-1}\leq t_{\epsilon,2}\leq 4\gamma(G)\log \epsilon^{-1}\quad\label{eq:L2_Mixing}
\end{equation}
\end{cor}

With Corollary \ref{Cor:L2mixing} in hand, we are able to analyze the $L^2$ mixing time for several natural families of graphs. Our results are summarized in Table \ref{tab:L2} below, see Section \ref{subsubsec: proof examples L2} for proof details. 
\begin{table}[htbp]
\begin{center}
\begin{tabular}{c|c|c|c|c}
\toprule
{\bf Graph} & Complete graph $K_n$ & Binary tree $B_n$ & Star $S_{n-1}$ & Cycle $C_n$ \\[4pt]
\midrule
$\boldsymbol{t_{\epsilon, 2}}$      &     $\Theta(n)\log\epsilon^{-1}$           &    $\Theta(n^2)\log\epsilon^{-1}$            &      $\Theta(n)\log\epsilon^{-1}$           &         $\Theta(n^3) \log\epsilon^{-1} $ \\[4pt]
\bottomrule
\end{tabular}
\caption{The $L^2$ mixing time for selected graphs}
\label{tab:L2}
\end{center}
\end{table}

\subsection{$L^{1}$ convergence}\label{subsec:L1}

Although the $L^2$ distance is often regarded the natural metric for its mathematical convenience, two special reasons motivate the study of $L^1$ distance. First,  the $L^1$ metric has a natural interpretation in probability theory. Any discrete probability measure is a unit vector under the $L^1$ metric, with non-negative coordinates.
If the starting vector, $v(0)$, is a probability distribution, it remains so during the entire process.
We then look at how quickly this distribution tends to uniform measured in total 
variation distance.
Second, it is proved in \citep{chatterjee2019phase} that 
for the averaging process the $L^1$ convergence behaves differently from the $L^2$ convergence for the complete graph. They show that it takes $\Theta(n)$ steps  to converge in $L^2$ metric, but $\frac12 (n\log_2n + c\, n \sqrt{\log_2n)}$ steps  to converge in $L^1$ metric. 

In this section we turn to study of the $\epsilon$-mixing time, $t_{\epsilon,1}$.

From the standard $L^{2}$ techniques
we have the following upper and lower bounds, see Section \ref{subsubsec: proof L1} for proof details. 
\begin{thm}[General bounds for $t_{\epsilon,1}$]\label{thm: t_1}
	\begin{align} \label{eqn: t_1, first formula}
		\left(2\gamma(G) - 1\right)  \log  \epsilon^{-1} \;
		\leq  \;  t_{\epsilon,1}  \; \leq \;  4\gamma(G)
		(\log \sqrt n + \log \epsilon^{-1})
	\end{align}
\end{thm}

Equation \eqref{eqn: t_1, first formula} %and \eqref{eqn: t_1, second formula} 
gives the $L^{1}$ mixing time up to a $\log n$ factor. 

In this section we develop techniques to get
the sharp bounds in many cases. 

\subsubsection{A universal $\Omega(n\log n)$ lower bound via the augmented entropy approach}\label{subsubsec: universallower}

% THE nlog SECTION

Since $\gamma(G) = \Omega(\lvert V(G)\rvert)$, Equation  \eqref{eqn: t_1, first formula} immediately gives a $\Omega(n)$ lower bound for every graph with $n$ nodes. Due to e.g. \citep{chatterjee2019phase} we know that the complete graph mixes in $\Theta(n\log n)$ steps, and since intuitively the complete graph seems to mix fastest, it is natural to assume that this is a universal lower bound for general connected graphs. In this section, we show that this is the case
in a precise sense: $t_1(\epsilon) \geq \frac{(1-\epsilon)n\log n}{2\log2}-O(n)$ for all graphs on $n$ nodes. 

To motivate our proof, we review the entropy lower bound idea (e.g. of \citep{chatterjee2019phase}) and explain why it fails for graphs like the star. Throughout our analysis, we assume without loss of generality that the  initial vector $v = v(0)$ has 
non-negative coordinates and  $\sum_{i=1}^{n} v_{i}(0) = 1$.
The idea goes as follows. Let 
\[
S(v) :=  \sum_{i=1}^n v_i \log \frac{1}{v_i}
\]
be the entropy of $v$, where $0\times \log 0$
is by definition zero. When the averaging process for a certain graph $G$ starts at $v(0) = (0, \cdots,0, 1, 0,\cdots, 0)$, the entropy $S(v(t))$ increases from $0$ to $\log n$ as $t$ approaches the mixing time. If the expected entropy increase is bounded by $\calO(1/n)$ per step, then we need at least $\Omega(n\log n)$ steps for the mixing. This idea works well for regular
graphs (in general, for graphs with ``balanced'' degree sequences), but entirely fails for graphs with $d_{\rm max}/d_{\rm min}\rightarrow \infty$. Let us take for instance the star graph, $S_n$, which has one node (center) with degree $n-1$, and $n-1$ leafs with degree $1$. Suppose that $v(0)$ puts weight 1 on the center and zero elsewhere. We look at the progress of the expected entropy increase. Already in the first step the entropy
increases exactly by one,
since no matter which edge of the star is chosen by the averaging process, the sorted list of
the entries of $v(1)$ will be 
$\frac{1}{2}, \frac{1}{2}, 0, \ldots$.
We are able to fix this problem if we penalize for the center 
having a large weight. If we add twice the 
weight of the center to the entropy function,
the decrease of this penalty term after the first averaging step will cancel out the increase of the entropy, and in general the increase of entropy + penalty will not exceed $\calO(1/n)$ on expectation per step. This is our plan for
providing a balanced progress measure.

To carry out this plan we introduce what we call  \textit{augmented entropy function}.

\begin{dfn}[Augmented Entropy Function]
	\begin{equation}\label{eqn: augmented entropy}
		F(v) := \sum_{i=1}^{n}v_{i}\log \frac{1}{v_{i}} +
		\sum_{i=1}^{n}\beta_{i}v_{i} = S(v) +      \sum_{i=1}^{n}\beta_{i}v_{i}.
	\end{equation}\label{def:AugmentedEntropy}
\end{dfn}

This will be our progress measure instead of 
$S(v)$.
The sequence $\beta_{1},\ldots,\beta_{n}$ 
are parameters of $F$ that 
will be cleverly chosen depending on the graph. The first term $S(v)$ in \eqref{eqn: augmented entropy} is  the entropy of $v$ when viewing $v$ as a probability distribution. When $v$ is close to the uniform vector $(1/n, 1/n, \ldots, 1/n)^T$, the entropy term will be close to $\sum_{i=1}^{n} \frac{1}{n}\log n= \log n$. 
The precise relation between $L^{1}$-closeness to uniform and end entropy is spelled out 
in Corollary \ref{cor:logn}.
It will be one of the challenges in the proof to find
suitable parameters $\beta_{i}\ge 0$
so that the expected change of $F(v)$ 
in every step remains upper bounded by 
$\calO(1/n)$. When this is accomplished and the starting vector, $v(0)$, is well-chosen, we have shown 
that the convergence takes at least $\Omega(\log n / (1/n)) =\Omega(n\log n)$ steps, and 
we can even pin down the precise constant in 
the $\Omega$.

\begin{thm}\label{thm:nlogn-gen}
	Let $G$ be a graph on $n$ nodes with degree vector $\mathbf{d}=(d_{1},\ldots,d_{n})^{T}$ 
	and Laplacian $L$. Let $\bar d  = \sum_{i=1}^n d_i /n$ the average degree of $G$. Suppose there exists a vector $\beta \equiv (\beta_{1},\ldots,\beta_{n})^T$ such that for some constant $C$ it component-wise holds that
	\begin{equation}\label{eq:thmeq}
		\mathbf{d} \leq \frac{1}{2\log 2} L\beta + \frac{C\bar{d}}{\log 2} ~ \mathbf{1}\;.
	\end{equation}

	%%%%%%%%%%%%%%%%%%%%%%%%%%%%%%%%%%%
	%%%%%%%%%  newly inserted %%%%%%%%%

	Then there exists a $1\le k\le n$ such that the 
	starting vector $e_{k}$ (the vector whose $k$-th coordinate is $1$ and the others are $0$) has the
	following property: If we set $v(0) =  e_{k}$, 
	then for any $\epsilon < 1$ it holds that
	\begin{equation}\label{eq:thmeq2}
		\bE[\lVert v(t) - \bar v\rVert_1] \geq \epsilon,\;\;
		\text{as long as}~~ t \leq \frac{(1-\epsilon)n\log n}{2C} -O(n).
	\end{equation}
\end{thm}

Putting the above conclusion in the counter-positive, we cannot 
reach $\epsilon$-convergence in the $L^{1}$ sense earlier than
after $\frac{(1-\epsilon)n\log n}{2C} -O(n)$
steps.

In the preparation for proving the theorem we 
state some useful facts.
Our first lemma estimates the 
change of
$S(v) = \sum_{i=1}^{n}v_{i}\log \frac{1}{v_{i}}$ 
in a single move:

\begin{lem}\label{lem:increase}
	When averaging $v_{i}$ and $v_{j}$ the entropy  $S(v)$ increases 
	by at most $(\log 2) (v_{i}+ v_{j})$.
\end{lem}

\begin{proof}
	The increase of $S(v)$ is expressed as
	\begin{eqnarray*}
		v_{i}  \log v_{i} + v_{j}  \log v_{j} - 2
		\frac{v_{i} + v_{j}}{2} \log \frac{v_{i} + v_{j}}{2} & = & \\
		v_{i}  \log \frac{2v_{i}}{v_{i} + v_{j}} +
		v_{j}  \log \frac{2v_{j}}{v_{i} + v_{j}}.
	\end{eqnarray*}
	We have $\log \frac{2a}{a + b} \le \log 2$ for every 
	$a,b\ge 0$, which gives the lemma.
\end{proof}

We will also need a lemma, which gives an upper bound on the entropy difference by the $L^1$ distance:

\begin{lem}[Fannes inequality]
	\label{lemma:fa}
	Let $S=\sum_{i=1}^{n}v_{i}\log\frac{1}{v_{i}}$
	be the entropy function (with $e$-based log)
	and $v$ and $w$ be two probability distributions on 
	$n$ points. Then
	\[
	|S(v)-S(w)| \le ||v-w||_{1}\log n + \frac{1}{e\log 2}
	\]
\end{lem}

\begin{cor}\label{cor:logn}
	Let $v$ be a probability distribution 
	and $0\le \delta\le 1$ be arbitrary. Then
	\begin{align*}
		S(v) \; \geq \;
		\log n - 
		\left|\left|v-\frac{1}{n}\mathbf{1}
		\right|\right|_{1}\log n - \frac{1}{e\log 2}
	\end{align*}
\end{cor}
The corollary is immediate. Finally
we describe the expected effect 
of a single step of
the process on a given $v(t)$:
\begin{lem}\label{lem:singlestep}
	$\bE[v(t+1)\mid v(t)]
	= \left( I_n - \frac{L}{2|E|}\right) v(t)$
\end{lem}
The lemma comes 
from the linearity of expectation. Now we are ready to prove Theorem \ref{thm:nlogn-gen}.
\begin{proof}[Proof of Theorem \ref{thm:nlogn-gen}]
	Due to the identity
	$L\mathbf{1} =
	\mathbf{0}$ (vector $\mathbf{1}$ is always an eigenvector
	of the Laplacian with eigenvalue zero)
	we can replace
	$\beta$ with $\beta + \mu \mathbf{1}$
	without changing $L\beta$. 
	We call this ``shifting,'' since we add
	$\mu$ to each coordinate of $\beta$.
	Let
	\[
	k = \argmin \{1\le i \le n \; \mid \; \beta_{i}\}.
	\]
	Shift $\beta$ 
	in (\ref{eq:thmeq}) by $\beta_{k}$, 
	i.e. redefine $\beta$ as 
	$\beta - \beta_{k}\mathbf{1}$.
	Then $\beta_{k} = 0$, and we have not changed 
	(\ref{eq:thmeq}).
	Thus without loss of generality we can assume that:
	
	\begin{center}
		\begin{enumerate}
			\item $\beta$ has non-negative entries,
			\item there 
			is a $k$ such that $\beta_{k}=0$.
		\end{enumerate}
	\end{center}
	
	Define $v(0) = e_{k} = (0,\ldots,0,\underbrace{1}_{k}
	0,\ldots,0)$ {\em with the above $k$}.
	
	To estimate $\bE[F(v(t))]$
	in the final $t$ we have the lemma below
	(immediate from Corollary \ref{cor:logn}).
	This is the only place where
	we need Corollary \ref{cor:logn}
	and the non-negativity of $\beta$:
	
	\begin{lem}\label{lem:elem}
		For any non-negative $\beta$ and any 
		probability 
		distribution 
		$v = (v_{1},\ldots,v_{n})$ we have:
		\begin{align*}
			F(v) \; \geq \;
			\log n - 
			\left|\left|v-\frac{1}{n}\mathbf{1}
			\right|\right|_{1}\log n 
			- \frac{1}{e\log 2},
		\end{align*}
	\end{lem}
	
	\begin{cor}\label{cor:exp}
		$\bE[F(v(t_{\epsilon,1}))] \ge
		(1-\epsilon)\log n
		- \frac{1}{e\log 2}$
	\end{cor}
	\begin{proof} Let $t=t_{\epsilon,1}$. From
		Lemma \ref{lem:elem} and from the definition
		of $t_{\epsilon,1}$: 
		\begin{eqnarray}\label{eqn:entropy-L1}\nonumber
			\bE[F(v(t))] 
			\ge
			\log n - \bE\left[
			\left|\left|v(t)-\frac{1}{n}\mathbf{1}
			\right|\right|_{1}
			\right]\log n - \frac{1}{e\log 2}
			\ge (1-\epsilon)\log n
			- \frac{1}{e\log 2}
		\end{eqnarray}
	\end{proof}
	
	Corollary \ref{cor:exp} 
	gives a lower bounds 
	on the {\em expectation} of $F$ in the final
	step, when we reach $\epsilon$-closeness 
	to uniform in the $L_{1}$ norm, {\em on expectation}. To press ahead with our proof 
	plan we also need the value
	of $\bE[F(v(0))]$ which is just
	$F(v(0))$. This is easy to compute:
	Due to $v(0) = e_{k} = (0,\ldots,0,\underbrace{1}_{k}
	0,\ldots,0)$ and $\beta_{k}=0$
	we have $F(v(0))=0$.
	
	Next, we upper bound the change, 
	$\bE[F(v(t+1))]-\bE[F(v(t))]$,
	at any step $t$. The ratio of 
	\begin{equation}\label{eq:diff}
		\bE[F(v(t_{\epsilon,1}))] - 
		\bE[F(v(0))] \;\ge\;
		(1-\epsilon)\log n
		- \frac{1}{e\log 2}
	\end{equation}
	and this upper bound
	on the step-wise change of $\bE[F(v(t))]$
	will give Equation (\ref{eq:thmeq2}) 
	of the theorem.
	Recall $C$ from Equation (\ref{eq:thmeq}). 
	We prove 
	\begin{lem}\label{lem:jump} For any $t\ge 0$:
		\[
		\bE[F(v(t+1))]-\bE[F(v(t))]\;\leq\; 2C/n
		\]
	\end{lem}
	
	\begin{proof}
		We prove the stronger
		\begin{equation}\label{eq:ee}
			\bE[F(v(t+1))\mid v(t)] - F(v(t))\;\leq\; 2C/n
		\end{equation}
		In other words, no matter how
		we fix $v(t)$, the expected growth
		of $F$ in the next step is at most $2C/n$. 
		Equation (\ref{eq:ee}) implies Lemma
		\ref{lem:jump}, simply by taking the expectation 
		over all 
		$v(t)$ (according to the distribution in which $v(t)$ arises 
		in the process from $v(0)$).
		
		The function $F$ consists of the entropy term
		and the term $\beta^Tv(t)$. 
		By Lemma \ref{lem:increase}, the expected change of the entropy term conditioned on a fixed $v(t)$, 
		can be bounded by:
		\begin{equation}\label{eq:eterm}
			\bE[S(v(t+1))\mid v(t)] - S(v(t)) \leq \frac{(\log 2)\sum_{(i,j)\in E} (v_i(t) + v_j(t))}{\lvert E\rvert} = \frac{(\log 2)\mathbf{d}^T v(t)}{\lvert E\rvert}.
		\end{equation}
		(Note, that this alone can be much larger than $2C/n$.)
		Further, the expected change of the term $\beta^T v(t)$
		conditioned on $v(t)$ is exactly
		
		\begin{equation}\label{eq:lterm}
			\bE[\beta^T v(t+1)\mid v(t)] - \beta^T v(t) = \beta^T\left( I_n - \frac{L}{2|E|}\right) 
			v(t) - \beta^T v(t) = -\beta^T \frac{L}{2|E|} v(t).
		\end{equation}
		where the first equality follows from 
		Lemma \ref{lem:singlestep}. Adding (\ref{eq:eterm}) and (\ref{eq:lterm}) 
		a ``cancellation'' occurs:
		\begin{equation}\label{eqn:estimate}
			\bE[F(v(t+1))\mid v(t)] - F(v(t)) \;\leq\; \frac{2(\log2)\mathbf{d}^T - \beta^T L}{2\lvert E\rvert} v(t)\;\leq\; \frac{2C\bar d}{2|E|}{\mathbf{1}^T v(t)
			}\;=\; \frac{2C}{n}
		\end{equation}
		
		The first inequality follows from the definition of $F$ and the estimates above. The second inequality uses the assumption (\ref{eq:thmeq}) of the theorem: $\mathbf{d} \leq \frac{1}{2\log 2} L\beta + \frac{C\bar{d}}{\log 2}  \mathbf{1}$, and the fact that $v(t)$ is a non-negative vector. 
	\end{proof}
	Comparing Expression (\ref{eq:diff}) 
	(total change of $\bE[F(v(0))]$
	from $t=0$ to $t=t_{\epsilon,1}$)
	and
	Lemma 
	\ref{lem:jump} (for the one step increase 
	of $\bE[f(t)]$)
	we get:
	\[
	t_{\epsilon,1} \; \ge \; \frac{(1-\epsilon)\log n
		- \frac{1}{e\log 2}}{2C/n}
	=  \frac{1-\epsilon}{2C}n\log n- O(n)
	\]
\end{proof}

%%%%%%%%%%%%% end of newly inserted %%%%%
%%%%%%%%%%%%%%%%%%%%%%%%%%%%%%%%%%%%%%%%%

We are left with having to find the $\beta$
that gives the optimum $C$.
\begin{lem}
	For a connected $G$ the equation
	\begin{equation}\label{eq:laplaciannew1}
		L\,\mathbf{x} = c\, \mathbf{d} - c\,\overline{d}\, \mathbf{1}
	\end{equation}
	always has a solution $\mathbf{x}$ for any 
	$c\in \mathbb{R}$
	in the subspace orthogonal to
	$\mathbf{1}$.
\end{lem}
\begin{proof}
	When $G$ is connected,
	the rank of $L$ is $n-1$, with $\langle {\bf 1}\rangle$ as its zero subspace. 
	(It is known that the multiplicity of the $0$ eigenvalue of $L$ equals the number of connected components of the graph \cite{fiedler1973algebraic}.)
	Since $L$ is symmetric (thus self-adjoint), any equation $L\,\mathbf{x} = \mathbf{y}$,
	where $\mathbf{y} \in \langle {\bf 1}\rangle^{\perp}$
	is solvable.
	The lemma is now implied by the fact that 
	$c\, \mathbf{d} - c\,\overline{d}{\bf 1}$
	is orthogonal to ${\bf 1}$.
\end{proof}

With the 
choice of $c=2\log2$ we can turn Equation 
(\ref{eq:laplaciannew1}) into
Equation \ref{eq:thmeq} in Theorem \ref{thm:nlogn-gen} (satisfied with equality!) 
if we also set 
$C=\log 2$. This in turn proves our main 
theorem:

\begin{thm*}[Theorem \ref{thm:genlb} in Section \ref{Sec:Intro}]
	For any connected graph $G$ with $n$ nodes, the $\epsilon$-mixing time satisfies 
	
	$$t_{\epsilon,1} \geq\frac{(1-\epsilon)n\log n}{2\log2}-O(n)$$ 
\end{thm*}

Theorem \ref{thm:genlb} has the immediate corollaries
(see also in Section \ref{subsubsec:univproof}):
\begin{cor}\label{cor:star}
	The star graph $S_{n-1}$ with $n$ nodes has $t_{\epsilon,1} = \Theta_\epsilon(n\log n)$.
\end{cor}

(The upper bound easily follows from the spectral
analysis of the star.)

\begin{cor}[Complete graph mixes fastest]\label{cor: fastest L1}
	
	Let $\{G_n\}_{n}$ be any family of 
	connected graphs ($n$ is the number of nodes), and $K_n$ be the complete graph on $n$ nodes. Then we have:
	\[
	\liminf_{\epsilon\rightarrow 0} \left(\liminf_{n\rightarrow \infty}\frac{t_1(\epsilon, G_n)}{t_1(\epsilon, K_n)}\right) \geq 1.
	\]
\end{cor}

It is natural to ask whether $K_n$ mixes the fastest under the $L^1$ metric for any fixed
$\epsilon$. Perhaps surprisingly, this is not the case. The star graph seems to mix 
faster than the complete graph 
for $\epsilon$ values that are between 2 and a threshold, which is 1 or slightly larger than 1
(the threshold seems to converge to 1 as 
$n$ converges to infinity). 
For small values of $\epsilon$ however
$K_{n}$ seems to be the best
(see Figures \ref{fig:ComStars_exp} and \ref{fig:EDelta1}). Alternatively, we
can fix $t$ and compute
$\bE[||v(t))-\frac{1}{n}{\bf 1}||_{1}]$
as an indication for the mixing speed.
For $t=1$ and $t=2$ the clique gives a 
lower expectation value than the star graph 
of the same size, for $t=3$ the clique still has a tiny advantage, which reverses by $t=4$.
For very small $t$s these numbers can be explicitly computed. For larger $t$s we refer the reader to Figure \ref{fig:ComStars_exp}.

\begin{figure}[H]
	\begin{centering}
		\includegraphics[scale=0.28]{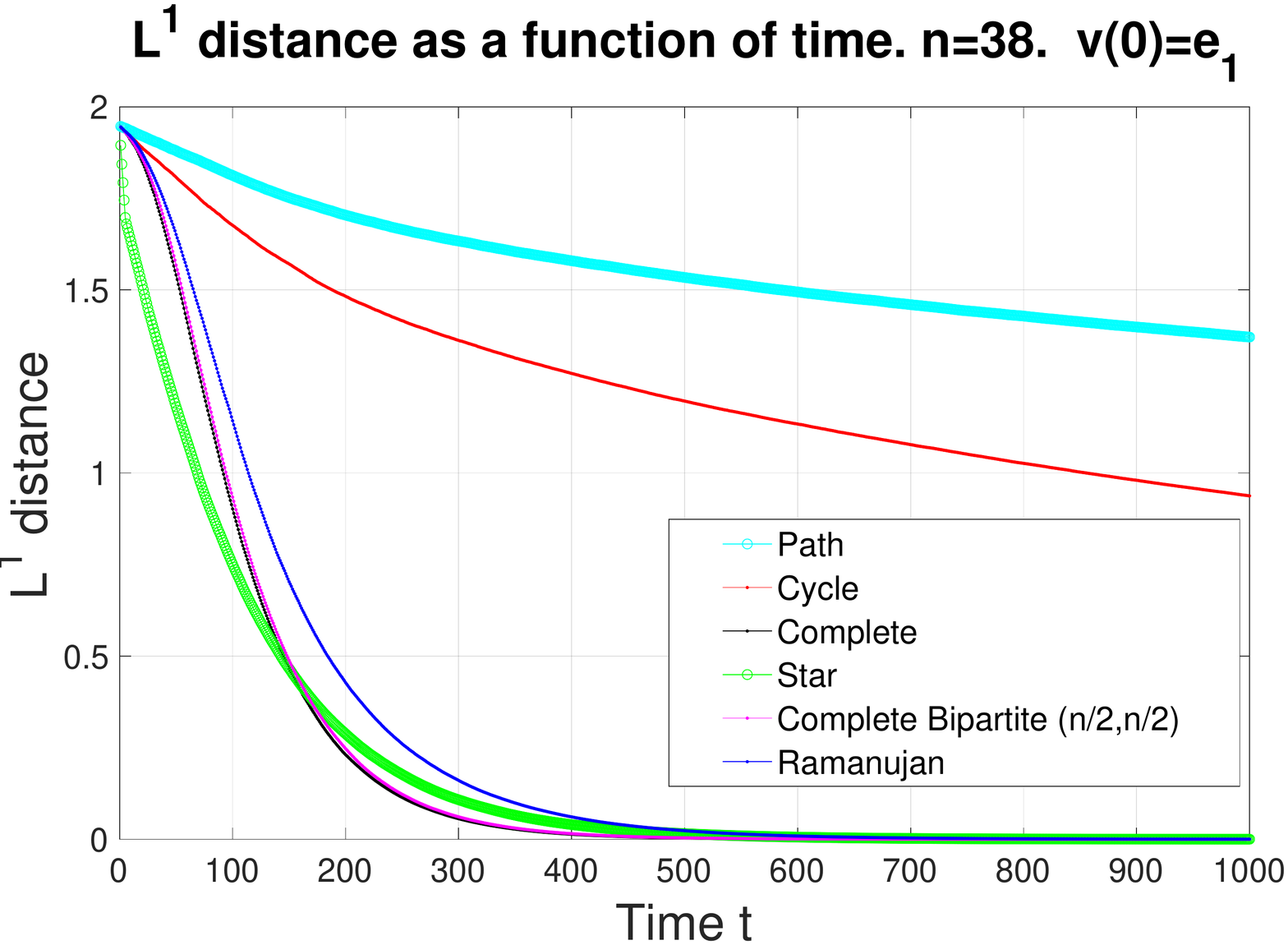}
		\includegraphics[scale=0.28]{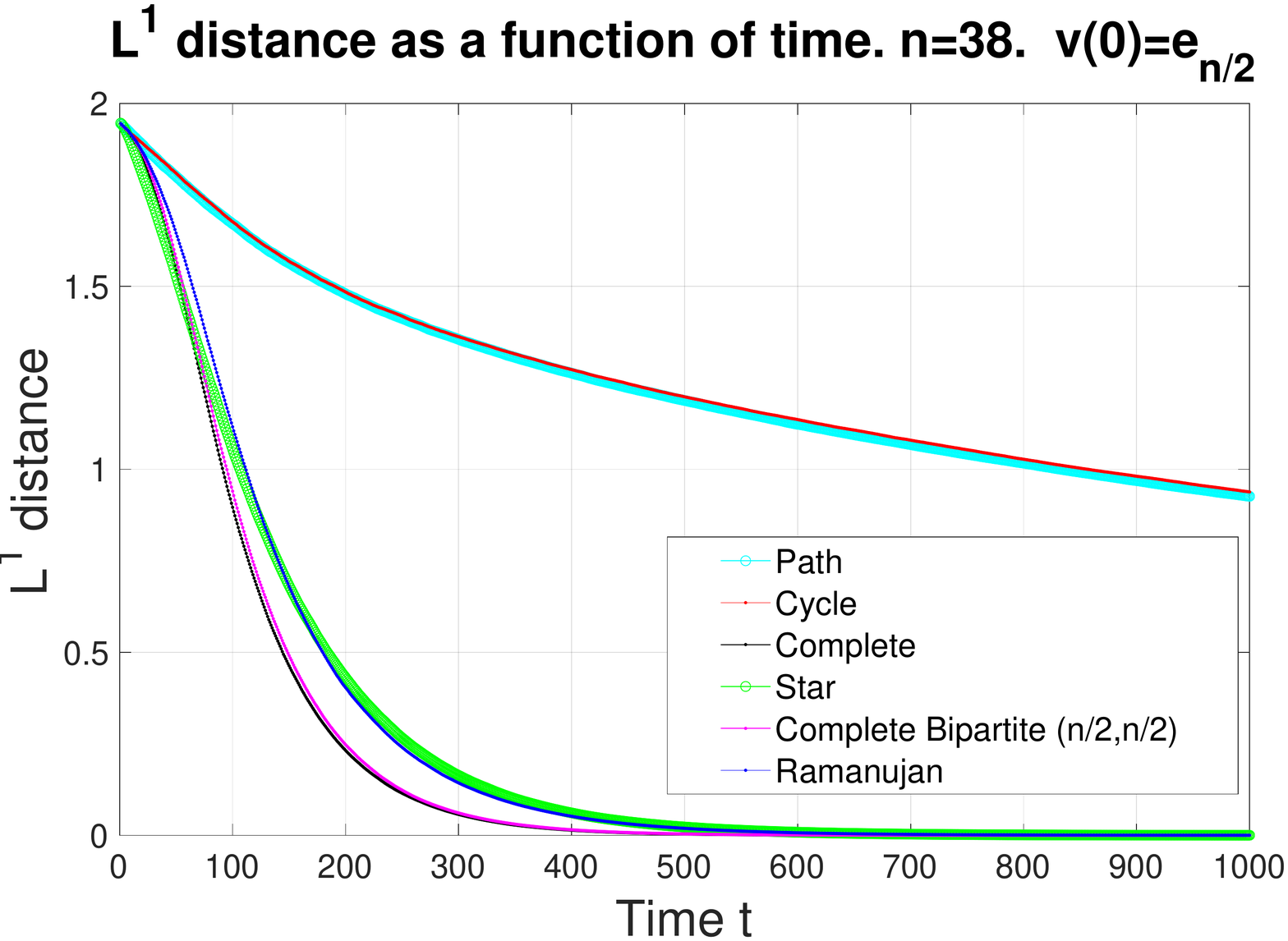}
		\par\end{centering}
	\caption{\label{fig:EDelta1}
		Comparison of the 
		$L^{1}$ convergence for the Path, Cycle, Complete, Star, Complete Bipartite graphs and 
		for a relatively sparse expander graph, each on 38 nodes.
		The left and right 
		figures differ only where we place 
		the starting weight 1 (all other starting weights are 0: On the left figure we place the 1 in the middle of the star and the side of the path, on the right 
		in the middle of the path and on a terminal node of the star; the other graphs are symmetric.}
\end{figure}

\subsubsection{Improved bounds: $\alpha$-covering, flow and comparison techniques.}\label{subsubsec:improved bound}

In this section we introduce new techniques to bound $t_{\epsilon,1}$ for several instances.
These bounds will add information not derivable from
Theorems \ref{thm:genlb} or \ref{thm: t_1}.
Our first very simple but also very useful observation
will be that the worst initialization always happens at a corner of the $L^{1}$-unit ball. 
Our proof  uses the convexity of 
the $L^1$ norm together with a coupling argument. 

\begin{dfn}[$L^{1}$-unit ball $\Diamond$]
	The $L_{1}$-unit ball is the set
	\[
	\Diamond_{n} = 
	\{v\in \mathbb{R}^{n}\mid ||v||_{1} \le 1\}
	\]
	When $n$ is clear from the context, 
	we just write $\Diamond$ for $\Diamond_{n}$.
\end{dfn}

We have:

\begin{prop}\label{prop:diam}
	$\Diamond_{n}$ is the convex hall
	of $\{e_{1},\ldots,e_{n},-e_{1},\ldots,-e_{n}\}$.
\end{prop}

\begin{lem}\label{lem:neg}
	The mixing properties of the 
	process do not change if we replace
	$v(0)$ with $-v(0)$.
\end{lem}

\begin{thm*}[Theorem \ref{thm: L1, slowest initialization} in Section \ref{Sec:Intro}]
	Let $G$ be a graph with $n$ nodes. For every fixed $t$, define $v_{{\rm slowest},t}$
	such that 
	$\bE \lVert v(t) - \bar v \rVert_1$
	is maximized (if there are 
	more than one such starting vectors we can non-deterministically pick one). Then we can define
	\[
	v_{{\rm slowest},t} = e_i = (0, \cdots,0, \underbrace{1}_{i}, 0,\cdots, 0)\;\;\;\;
	{\rm for} \; {\rm some} \; 1\le i\le n
	\]
\end{thm*}

The theorem follows
from Proposition \ref{prop:diam},
Lemma \ref{lem:neg} and from the following fact:
\begin{prop}
	\label{prop:convex_combination}
	Assume that $v$ is a convex combination of 
	$v_{1},\ldots,v_{\ell}$, i.e. 
	$v = \sum_{i=1}^{\ell}\lambda_{i}v_{i}$, where
	$\sum_{i=1}^{\ell}\lambda_{i}\; =\; 1$ and
	$\lambda_{i}\ge 0$ for $1\le i\le \ell$.
	Let $v(t)$, respectively $v_{i}(t)$ 
	denote state vector 
	when we start from $v$,
	respectively from $v_{i}$
	($1\le i\le \ell$). Then
	\[
	\bE (\lVert v(t) - \bar v \rVert_1)
	\; \le \; \sum_{i=1}^{\ell}\lambda_{i}\;
	\bE (\lVert v_{i}(t) - \bar v \rVert_1)
	\]
\end{prop}

\begin{proof}
	We couple two processes. The first process is simply the averaging process with initial vector $v$, while the second process  has a random initialization, which equals $v_i$ with probability $\lambda_i$ ($1\le i\le \ell$). After specifying the initial vector, we always choose the same edge for both processes at every step. After $t$ steps, we observe that
	\[
	\lVert v(t) - \bar v \rVert_1 
	\; = \; \left\lVert \sum_{i=1}^n \lambda_{i} (v_i(t) - \bar v_i )\right\rVert_1 \; 
	\; \le \;
	\sum_{i=1}^n \lambda_{i} 
	\left\lVert 
	v_i(t) - \bar v_i
	\right\rVert_1
	\]
	taking expectation on both sides yields
	\[
	\bE \lVert v(t) - \bar v \rVert_1
	\leq \sum_{i=1}^n 
	\lambda_{i} \bE \; \lVert 
	v_i(t) - \bar v_i
	\rVert_1\; ,
	\]
	as desired.
\end{proof}

\begin{proof}
	(of theorem
	\ref{thm: L1, slowest initialization})
	From the above theorem we
	get that when
	$v$ is a convex combination of 
	$v_{i}$s, then
	$\bE \lVert v(t) - \bar v \rVert_1
	\; \leq \; \max_{i=1}^\ell 
	\bE \; \lVert 
	v_i(t) - \bar v_i
	\rVert_1$. Assume now that
	$\lVert v \rVert = 1$.
	Since $v$
	is a convex combination
	of the $e_{i}$s and 
	$-e_{i}$s, there must 
	be an $i$ such that either
	$\bE \lVert v(t) - \bar v \rVert_1
	\; \leq \;
	\bE \lVert e_{i}(t) - 
	\frac{1}{n}{\bf 1} \rVert_1$
	or $\bE \lVert v(t) - \bar v \rVert_1
	\; \leq \;
	\bE \lVert -e_{i}(t) 
	-\frac{1}{n}{\bf 1} \rVert_1$.
	But since both right hand sides 
	are the same, for 
	$v_{{\rm slowest},t}$
	we can pick the $e_{i}$ with 
	the positive sign, and
	Theorem 
	\ref{thm: L1, slowest initialization} 
	follows.
\end{proof}

Finding the slowest initialization requires knowing the structure of the underlying graph. For example, by vertex transitivity, every $v(0)=e_{i}$ $(1\le i\le n)$
converges at the same speed for complete graphs and cycles. On the other hand, it is not hard to show that for star graphs starting from a leaf is strictly slower than starting from the root. 

Now we can assume without loss of generality that the process starts at a corner of the simplex. We give some techniques to bound the $L^1$ mixing, which improves the general bound in Theorem \ref{thm: t_1}. 

\paragraph{$\alpha$-covering} 
The $\alpha$-cover time captures the first time that at least $\alpha$-proportion of the entries of $v(t)$ are non-zero. 

\begin{dfn}[$\alpha$-covering time]\label{def:covering}
	For any vector $v\in \bR^n$, let $n(v)\equiv \#\{i: v_i \neq 0\}$ denote the number of non-zero components of $v$. Suppose the averaging process starts at $v(0) = e_i$. We denote by $
	\tcov(\epsilon, i) := \bE[\inf\{t: n(v(t)) = \epsilon n, v(0) = e_i\}]$ the expected time that $v(t)$  has $\alpha n$ non-zero coordinates when
	starting from $e_{i}$.
	We also define the $\alpha$-covering time as $
	\tcov(\alpha) := \max_i \tcov(\alpha, i).$
\end{dfn}
The notion comes from
\citep{chatterjee2019phase},
where it was observed that 
the $\alpha$-covering time
gives a ``makeshift''
lower bound for the $L^1$-mixing time 
of $K_{n}$. 
In general:
\begin{prop}[Alternative lower bound on $t_{\epsilon,1}$]\label{prop: t_1, alternative bound}
	\[
	t_{\epsilon,1} \geq \tcov(1-\epsilon).
	\]
\end{prop}
(See the easy proof in Section \ref{para: proof covering time}.)
The quantity $\tcov(1-\epsilon)$ can often be calculated or bounded by coupon-collector type arguments. For example, $\tcov(1-\epsilon)$ is at least $\Theta_\epsilon(n\log n)$ if the graph is nearly regular:
\begin{prop}\label{prop: nearly regular graph, covering time}
	Let $G_n$ be a graph with $n$ nodes which satisfies:
	$\frac{\max_i d_i}{\min_i d_i} \leq C$
	for some universal constant $C$, where $d_i$ is the degree of node $i$. Then 
	$\tcov(1-\epsilon) \geq \frac{n}{2C}\log((1-\epsilon)n).
	$
\end{prop}
\begin{prop}[Expander graph]\label{prop:expander}
	Let $G_{n}$ be a 
	be a family of bounded degree 
	regular expander graphs with $n$ nodes.
	Then $\tcov(1-\epsilon,\; G_{n}) =
	\Theta_\epsilon(n\log n)$.
\end{prop}
The proofs of the 
above propositions are
in Sections \ref{para: proof near regular} and \ref{para: proof of expander example}.

Although $\tcov(1-\epsilon)$ 
never beats our best
bounds on $t_{\epsilon, 1}$,
and it fails to give good 
bounds for graphs like e.g. the star graph,
it often gives an easy-to-prove
approximation,
and it is applicable
to various averaging 
processes with different weight 
updates, not just ours.

\paragraph{Flow} The  averaging process is a transportation
process, and observing the mass flow can be a way to estimate the mixing time. A typical example is the path graph $P_n$ with $n$ points. Assuming we label the nodes by $1,2,\dots, n$ from left to right, and initialize the process from the leftmost entry. It is clear that the direction of the flow is always from left to right at every step, see Figure \ref{fig:massflow} for illustrations. Lower bounds on the flow at each step can be translated to new upper bounds for the $L^1$ convergence. The 
flow technique is based on the
next general proposition:

\begin{figure}[H]
	% \begin{centering}
	\centering
	\includegraphics[scale = 0.5]{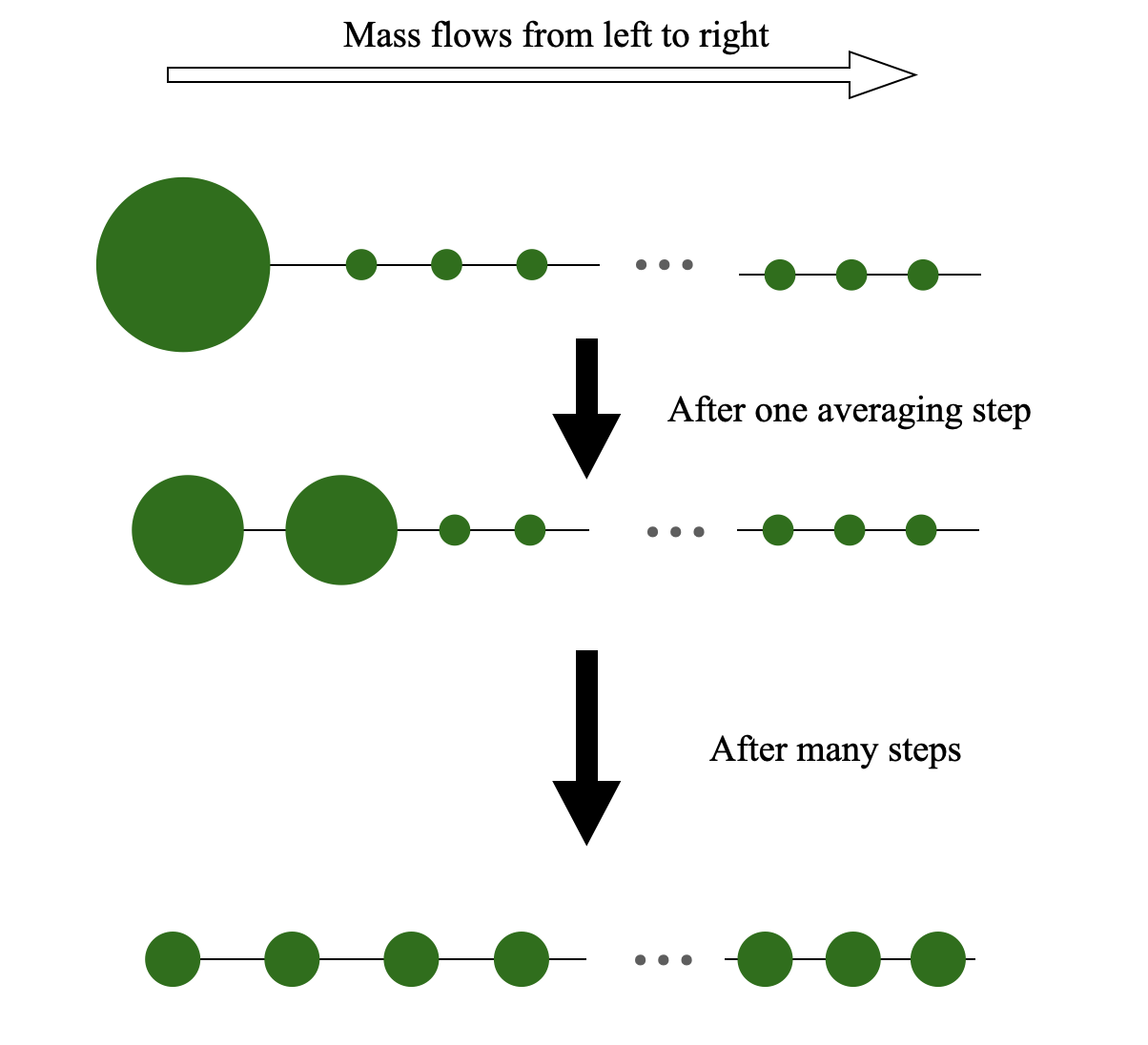}
	% \end{centering}
	\caption{\label{fig:massflow} An illustrative figure for the mass flow on the path graph $P_n$.}
\end{figure}

\begin{prop}[Upper bounding
	$t_{\epsilon,1}$ via the flow technique]\label{prop: flow}
	Let $G$ be a connected undirected graph. 
	For a starting vector 
	$v(0)$ of the 
	averaging process on $G$ 
	define $\bar v := \bar v(0)$.
	Suppose that for any
	$v(0)$ with 
	$\lVert v(0)\rVert \leq 1$ and
	each $t\ge 1$
	there exists a non-negative 
	real-valued random variable 
	$F(t,v(0))$ (called flow) 
	over the random
	edge sequences of length $t$
	such that
	\begin{itemize}
		\item 
		There exists a $C>0$ such that
		for every $v(0)$ 
		with $\lVert v(0)\rVert \leq 1$ and every $t\ge 1$ 
		and every $w\in 
		\mathbb{R}^{V(G)}$
		we have:
		$$\bE[F(t,v(0))\mid
		v(t-1) = w] \geq C\, 
		\lVert w - \bar v \rVert_1$$
		\item The cumulative flow, $\sum_{t} F(t,v(0))$, on 
		the space of infinite
		random edge sequences,
		converges to a 
		random variable $F(v(0))$ in $L^1$, i.e., $\bE[\lVert \sum_{t=1}^k F(t,v(0)) - F(v(0))\rVert]\rightarrow 0$ as $k\rightarrow \infty$. 
	\end{itemize}
	Let $\bE_{\rm max} = 
	\max_{\lVert v(0)\rVert \leq 1} \bE[F(v(0))]$.
	Then the mixing time can be upper-bounded by $$
	t_{\epsilon,1}(G,v(0)) \leq \frac{8\, \bE[F(v(0))]}{\epsilon^2 \,C};
	\;\;\;\;
	t_{\epsilon,1}(G) \leq \frac{8\, \bE_{\rm max}}{\epsilon^2 \,C}$$
	where
	$
	t_{\epsilon,1}(G,v(0)) \equiv\min
	\left\{t\in \mathbb{N}\; \middle\vert \;
	\bE\left[\lVert v(t) - \bar v(0)\rVert_1\right] \leq \epsilon\right\}
	$.
\end{prop}

\begin{proof}
	Fix an initialization 
	with $\lVert v(0)\rVert \leq 1$
	and for simplicity let us 
	omit $v(0)$
	from the arguments of $F(t)$
	and $F$.
	Let $E(t)$ be the event $\{\lVert v(t) -\bar v \rVert_1\geq \frac{\epsilon}{2}\}$. Then we have
	\begin{equation}\label{eq:flow1}
		\frac{\epsilon}{2}\bP[E(t)]\;\le\;
		\bE[\lVert v(t) - \bar v\rVert_{1}] \leq \frac{\epsilon}{2} + 2\bP[E(t)]
	\end{equation}
	The last inequality is
	due to $\lVert v(t) -\bar v \rVert_1 \leq \lVert v(t)\rVert_1 + \lVert \bar v \rVert_1\leq 2$. To bound $\bP[E(t)]$, we observe:
	\begin{align*}
		\bE[F] \; =\;
		\bE\left[\sum_{l=1}^\infty F(l)\right] 
		\; \geq \;  
		\bE\left[\sum_{l=1}^t F(l)\right] 
		\; \geq \;
		\frac{\epsilon\, C}{2} \sum_{l=1}^t  \bP\left[E(l-1)\right] 
		\; \geq \;
		t\frac{\epsilon\, C}{2} 
		\bP\left[E(t-1)\right],
	\end{align*}
	where the last inequality uses the fact $E(t)\subset E(t-1) \subset E(t-2) \cdots$, and the second to last inequality 
	uses the left side 
	of Equation (\ref{eq:flow1}). 
	Our result follows when
	taking $t \geq \frac{8 \bE[F]}{\epsilon^2 C}$,
	plugging in the result to
	the right-side inequality of (\ref{eq:flow1}), and 
	maximizing over all $\lVert v(0)\rVert \leq 1$.
\end{proof}

\begin{rem}
	The conditional 
	random variable
	$F(t,v(0),w) := [F(t,v(0))\mid
	v(t-1) = w]$
	will only depend on $w$
	and the last edge picked,
	but is independent on $t$
	and $v(0)$.
	Thus we can uniquely 
	define $F'(w) := F(t,v(0),w)$.
	Of course, $F(t,v(0))$ must also
	dependent on the probabilities
	$\bP(v(t-1)=w)$.
\end{rem}

Now we use the flow technique to analyze the path graph. Notice that  Theorem \ref{thm: t_1} shows  $t_{\epsilon,1}(P_n)$ is between $\Theta_\epsilon(n^3)$ and $\Theta_\epsilon(n^3\log n)$.
\begin{prop}[Path graph]\label{prop: L1, path}
	The $L^1$ mixing time for the path graph equals $\Theta_\epsilon(n^3)$ when the initialization vector $v(0) = e_1$.
\end{prop}

\begin{proof}
	It follows from induction that the random vector $v(t)$ is monotone
	in the sense that 
	$v_i(t) \geq v_j(t)$ when $i\le j$. We define $F(t)$ as $\frac{v_I(t) - v_{I+1}(t)}{2}$ where $(I,I+1)$ is the random edge chosen at time $t$
	(since $v(0)$ is fixed, we omit it from the argument list of $F(t,v(0))$). Here $F(t)$ is the mass transported from left to right at time $t$. 
	Given vector $v(t-1)$, the conditional expectation of $F(t)$ equals
	\begin{align*}
		\bE[F(t)\mid v(t-1)] = \frac{v_1(t-1) - v_n (t-1)}{2(n-1)} \geq \frac{\lVert v(t-1) - \bar v\rVert_1}{2n(n-1)},
	\end{align*}
	therefore the quantity $C(n)$ in Proposition \ref{prop: flow} can be chosen as $1/(2n^2-2n)$. On the other hand, due to the convergence theorem, the long-time total cumulative flow across all the edges is $\sum_{k=1}^{n-1}1-\frac{k}{n}=(n-1)/2$. Therefore $\bE[F_n] = \frac{n-1}{2}$. Applying Proposition \ref{prop: flow} gives us the mixing time upper bound $\frac{8 n (n-1)^2}{\epsilon^2}$. 
\end{proof}

\paragraph{Comparison and splitting:}\label{para:comparison splitting}  The comparison idea has been widely used in analyzing Markov chains \cite{diaconis1993comparison,diaconis1993bcomparison} to get sharp results on the mixing time. Here we develop comparison techniques to prove the mixing time for the averaging process. Our results can be applied to show that the process mixes in $\Theta(n^3)$ steps on both Dumbbell graphs $D_n$ and cycles $C_n$. When studying the mixing on Dumbbell graphs, we compare the averaging process with the same process on the complete graph $K_n$. When studying the mixing on cycles, we compare the process with a new splitting process and use the results in Proposition \ref{prop: L1, path}.

\subparagraph{The Dumbbell graph.}
The Dumbbell graph $D_{n}$ on $2n$ nodes is a disjoint union of two 
complete graphs of size $n$, connected by an edge.

\begin{thm}\label{thm: dumbbell}
	The $\epsilon$-mixing time of the Dumbbell graph is $\Theta_\epsilon(n^{3})$.
\end{thm}
\begin{proof}[Proof Sketch of Theorem \ref{thm: dumbbell}]
	(See the detailed proof in Section \ref{para:dumbbell}.)
	Let the complete graph on
	$\{1,\ldots, n\}$ be denoted by $L$, and the complete graph
	on $\{n + 1,\ldots, 2n\}$ be denoted by $R$. We also call edge $e=(n,2n)$ a {\em bridge}.  We shall decompose
	the edge sequence of the process to subsequences,
	separated by those events when $e$ is picked.  We call the run of each subsequence a {\em phase}, while averaging over the edge $e$ an {\em equalization step}.
	
	We may assume without loss of generality that the initialization vector is $e_1$. Easily deducible from the random edge sequence, each phase takes a random time which is geometrically distributed with expectation $\Theta(n^2)$. It is known from both Theorem \ref{thm: t_1} and \citep{chatterjee2019phase} that  the mixing time of the complete graph is $\Theta(n\log n)$ which is   smaller than $\Theta(n^2)$. Therefore we expect each phase very close to equalizing half of the Dumbbell graph, and each equalization step transfers $\Theta(1/n)$ mass from one part to the other. Therefore, it takes in total $\Theta(n)$ phases, and in turn $\Theta(n^3)$ steps for mixing. 
\end{proof}

\subparagraph{The Cycle graph.} Let $C_{n}$ be the cycle graph on $n$ nodes.
The usual spectral argument gives  $\Theta(n^3)\le t_{\epsilon,1}\le \Theta(n^3\log n)$, and
our goal is to sharpen this to $t_{\epsilon,1}(C_{n}) = \Theta(n^3)$.
The proof sketch that we present here is discussed in more details 
in Section \ref{para:cycle}.
It has two main ideas: the flow idea 
for the path graph $P_n$, and an interesting application 
of the comparison method. 

Let $PC_1$ denote the original averaging process on 
$C_n$. We also define a new splitting process, $PC_2$. 
The process starts at a vector $v(0) = (v_1, v_2, \cdots, v_n)^T$ with  $v_1 \geq v_2\ge \cdots \geq v_n$. Then, in each step $PC_2$ 
proceeds as follows:

\noindent
\rule{\textwidth}{0.8pt}
% \begin{center}\noindent\rule{8cm}{0.4pt}\end{center}
\begin{itemize}
	\item[1] Chooses an edge $(i, i+1 \mod n)$ uniformly at random. 
	\item[2] When $i < n$, updates both $v_i$ and $v_{i+1}$ to $\frac{v_i+ v_{i+1}}{2}$.
	\item[3] When $i = n$, if $\frac{v_1+v_n}{2}\geq v_2$, updates the sequence to $(\frac{v_1+v_n}{2}, \frac{v_1+v_n}{2}, v_2, v_3,\cdots v_{n-1})$. Otherwise, let $k\geq 2$ be the largest index with $v_{k} > \frac{v_1+v_n}{2}$, splits the original sequence into two sequences 
	\begin{align*}
		\left(\underbrace{\frac{v_1+v_n}{2}, \frac{v_1+v_n}{2},\dots,  
			\frac{v_1+v_n}{2}}_{k+1}, v_{k+1},\dots, v_{n-1}\right) 
	\end{align*}
	and
	\begin{align*}
		\left(v_2 - \frac{v_1+v_n}{2},
		v_3 - \frac{v_1+v_n}{2},
		\dots,  v_k - \frac{v_1+v_n}{2},
		\underbrace{0, 0, \dots, 0}_{n-k+1}\right) 
	\end{align*}
	and proceeds the splitting process $PC_2$ on the two sequences independently.
\end{itemize}
\noindent
\rule{\textwidth}{0.8pt}
% \begin{center}\noindent\rule{8cm}{0.4pt}\end{center}

Unlike the averaging process, the splitting process may split one sequence into many sequences after $t$ steps. On the other hand, as long as the process starts at a non-increasing sequence,
in the splitting process
monotonicity will be preserved for every 
split sequence, and for every time step $t$. 

First we define a few notations. Consider the splitting process with a non-increasing initialization $v = v(0)$. For time $t$, let $N_t$ be the number of sequences from the splitting process ($N_{t}$ is a random variable), and denote by $v^1(t),\cdots, v^{N_t}(t)$ those sequences. We also let $\tilde v(t) := \sum_{i=1}^{N_t} v^i(t)$ to be the sum of the split sequences, and
define
\[
T_v(t) := \lVert v(t) - \overline{v} \rVert_1 \;\;
({\rm for}\; PC_{1});\;\;\;
\tilde T_v(t) := \sum_{i=1}^{N_t} \lVert v^i(t) - \overline{v^i}(t) \rVert_1 \;\; ({\rm for}\; PC_{2})
\]
In Section \ref{para:cycle} we show:

\begin{lem}\label{lem:comparison cycle}
	Let $PC_1$ and $PC_2$ 
	processes as above
	with the same initialization $v =
	(v_1, v_2, \cdots, v_n)$
	satisfying $v_1 \geq 
	v_2 \geq \cdots 
	\geq v_n$. We have:
	\[
	\bE[T_v(t)] \leq \bE[\tilde T_v(t)] 
	\]
	and 
	\[
	\tilde T_v(t) \leq n (\tilde v_1(t) - \tilde v_n(t))
	\;\;\;\;\;
	{\rm where}\;
	\tilde v_j (t) := \sum_{i=1}^{N_t} v^i_j (t)
	\]
\end{lem}
Due to vertex transitivity,
without loss of generality we assume that $v = e_1 = (1,0, \cdots,0)^T$ is the initial vector. 
We define the linear form 
$Q(\vec{x}) = \sum_{i=1}^{n} (n+1-i) x_i$ and 
introduce $\tilde Q(t) = Q(\tilde v(t))$. Then we follow how $\tilde Q(t)$ changes 
throughout the splitting process $PC_2$. We show that $\tilde Q(t)$ is monotonically non-increasing in $t$
and give a lower bound on its 
{\it conditional} expected change in 
one step:
\begin{lem}\label{lem: monotonicity Q(t)}
	Let $v = e_1$ be 
	the initial vector.
	With all the notations defined above, we have:
	\begin{align*}
		n =  \tilde Q(1) \geq \tilde Q(2) \geq \cdots \geq \tilde Q(t) \geq \cdots \geq 0\;\;\;\;\mbox{(over every fixed 
			infinite edge sequence)},
	\end{align*}
	and 
	\begin{align*}
		\bE[\tilde Q(t-1) - \tilde Q(t)\mid \tilde v(t-1)] \geq \frac{\tilde v_1(t-1) - \tilde v_n(t-1)}{2n}\;.
	\end{align*}
\end{lem}

Lemma \ref{lem:comparison cycle} asserts that we can 
use the mixing time of the splitting process $PC2$ to 
upper bound the original process $PC1$. 
Notice that $\tilde Q(1) = n$, and that $\tilde Q(t)$ 
always remains non-negative. The total decrease
of $\tilde Q$ throughout 
the entire process is therefore upper bounded by $n$
(in every branch of the process). 
Let $E(t)$ be the event that $\tilde v_1(t) - \tilde v_n(t)\geq \epsilon/2n$. It suffices to show, that $\bP[E(t)]\leq \epsilon/4$ for $t = \Theta_\epsilon(n^3)$ (here we used the second part of Lemma \ref{lem:comparison cycle}). The second equation in  Lemma \ref{lem: monotonicity Q(t)} shows $$\bE[\tilde Q(t-1) - \tilde Q(t)]\geq \frac{\epsilon}{4n^2}\bP[E(t-1)]$$ by only counting the contribution made by the vectors in $E(t-1)$. Summing the above inequality from $t = 1$ to $t = 16n^3/\epsilon^2$, and observing the fact that $E(t)\subseteq E(t-1)\subseteq E(t-2)\subseteq \cdots$, we have $$n \geq \frac{16n^3}{\epsilon^2} \frac{\epsilon}{4n^2} \bP[E(16n^3/\epsilon^2)],$$
and therefore $\bP[E(16n^3/\epsilon^2)] \leq \frac{\epsilon}{4}$, which shows our main theorem  (details 
are in Section \ref{para:cycle}):  

\begin{thm}[Mixing time on a cycle]\label{thm:cycle}
	$t_{\epsilon,1}(C_n) = \Theta_\epsilon(n^3)$.
\end{thm}

% \section{Discussions}

\subsection{{$L^{2}\rightarrow L^{1}$} convergence}\label{subsec:L2->1}

In this section we study $L^1$ convergence with respect to  the initialization  $\lVert v(0)\rVert_2 = 1$. Since 
the $L^{2}$ unit sphere, where the initial vector ranges,
contains the $L^{1}$ unit sphere, but we measure the convergence in the $L^{1}$ distance, 
we must have slower (or at most equal) convergence than in the $L^{1}\rightarrow L^{1}$ case:

\begin{prop}\label{prop: comparison t_1 and t_12}
	Let $G$ be any connected graph with $n$ nodes, we have $
	t_{\epsilon,1}\leq t_{\epsilon,2\rightarrow 1}$
\end{prop}

Theorem \ref{Thm:L2_convg} implies 
the proposition below (the
simple details of the implication are in \ref{subsubsec: proof of L21 speed}):
\begin{prop}\label{prop: L21 convergence speed}
	Let $v(0)$ be an arbitrary vector on $\bR^n$ and $\bar v$ as before.
	Then the $L^1$ distance between $v(t)$ and $\bar v$ satisfies:
	
	\begin{equation}\label{eqn: L1 convergence speed, upper}
		\bE[\lVert v(t) - \bar v\rVert_1]\leq \bigg(1- {1 \over 2  \gamma(G)}\bigg)^{t/2}  \sqrt{n}\; \lVert v(0) - \bar v\rVert_2 
	\end{equation}
	Furthermore, the slowest convergence speed is at least $\left(1-\frac{1}{2\gamma(G)}\right)^{t}$:
	\[
	\sup_{\left\Vert v(0)\right\Vert _{2}=1}\mathbb{E}[\left\Vert v(t)-\bar{v}\right\Vert _{1}]\geq\left(1-\frac{1}{2\gamma(G)}\right)^{t}.
	\]
\end{prop}

Proposition \ref{prop: L21 convergence speed} gives the following general upper and lower bounds for $t_{\epsilon,2\rightarrow 1}$
that are very similar to formulas
in Theorem \ref{thm: t_1}
(proof details are 
in \ref{subsubsec: proof of t_12}):

\begin{cor}[General bounds for $t_{\epsilon,2\rightarrow 1}$]\label{cor: t_12}
	\begin{align} \label{eqn: t_12, first formula}
		(2\gamma(G) - 1) \log  (\epsilon^{-1} )
		\leq   t_{\epsilon,2\rightarrow 1}  \leq 4\gamma(G)\log (\sqrt n \epsilon^{-1})
	\end{align}
	In particular, for every undirected connected graph  with $n\ge 3$ nodes, we have
	\begin{align} \label{eqn: t_12, second formula}
		\gamma(G) \log  (\epsilon^{-1})
		\leq   t_{\epsilon,2\rightarrow 1}  \leq 4\gamma(G)\log (\sqrt n \epsilon^{-1}) 
	\end{align}
\end{cor}

The $\log n$ gap between the lower and upper bounds in Corollary \ref{cor: t_12} essentially comes from the Cauchy-Schwarz inequality used to prove Proposition \ref{prop: L21 convergence speed}. The gap however can be removed when $L(G)$
has a delocalized Fiedler vector.

\begin{dfn}\label{def:delocalization}
	A vector $v\in \bR^n$ is called $\delta$-delocalized if 
	\begin{align}\label{eqn: delocalization}
		\lVert v \rVert_1 \geq \delta \sqrt {n}\:\lVert v \rVert_2 \;.
	\end{align}
\end{dfn}

Heuristically, a vector is called 
delocalized if the top constant fraction of its coordinates are of roughly
the same magnitude (there are other ways to quantify delocalization, but they are of similar spirit). Delocalization of eigenvectors of graph 
Laplacians is an active area 
in random matrix theory and 
combinatorics. We refer the readers to \cite{brooks2013non} and \cite{rudelson2015delocalization} for recent progress. We show, that $\Theta(\log (n)  \gamma(G))$ is the correct magnitude for $t_{\epsilon,2\rightarrow 1}$ given the second eigenvector of the graph Laplacian is $\Omega(1)-$delocalized. Proof details can be found in Section \ref{subsubsec:t12,delocalize}.
\begin{thm}[Improved bounds for $t_{\epsilon,2\rightarrow1}$ with delocalized eigenvector]\label{thm: t_12,delocalize}
	Let $G$ be an undirected connected graph with $n$ nodes. Let $u_2$ be a unit eigenvector of  $L(G)$ with respect to the second smallest eigenvalue. If $u_2$ is $\delta$-delocalized, we have:
	\begin{align}\label{eqn: t_12, delocalize}
		\left (2\gamma(G) - 1\right)  \log  (\delta\sqrt n \epsilon^{-1})
		\leq   t_{\epsilon,2\rightarrow 1}  \leq  4\gamma(G) \log (\sqrt n \epsilon^{-1}) 
	\end{align}
	In particular, for every undirected connected graph  with $n\ge 3$ nodes, we have 
	\begin{align} \label{eqn: t_12, delocalize, second formula}
		\gamma(G) \log  (\delta\sqrt n \epsilon^{-1}) 
		\leq   t_{\epsilon,2\rightarrow 1}  \leq 4  \gamma(G)\log (\sqrt n \epsilon^{-1}) 
	\end{align}
\end{thm}

Theorem \ref{thm: t_12,delocalize} can be used to analyze the $L^{2\rightarrow1}$ mixing time for several natural families of graphs. Our results are summarized in Table \ref{tab:L21} below (see Section \ref{subsubsec: example L21} for proof details).

\begin{table}[htbp]
	\begin{center}
		\begin{tabular}{c|c|c|c|c}
			\toprule
			{\bf Graph} & Complete graph $K_n$ & Binary tree $B_n$ & Star $S_{n-1}$ & Cycle $C_n$ \\[4pt]
			\midrule
			$\boldsymbol{t_{\epsilon, 2\rightarrow 1}}$      &     $\Theta_\epsilon(n\log n)$           &    $\Theta_\epsilon(n^2\log n)$            &      $\Theta_\epsilon(n\log n)$           &         $\Theta_\epsilon(n^3\log n) $ \\[4pt]
			\bottomrule
		\end{tabular}
		\caption{The $L^{ 2\rightarrow 1}$ mixing time for selected graphs}
		\label{tab:L21}
	\end{center}
\end{table}

\section{Numerical Explorations}\label{sec: NumericalExp}

The goal of this section is to numerically corroborate some of the theory and inspire future work. We explore the expected $L^1$ behavior for three distinct scenarios: (1) The complete graph (2) Star graph with the initialization on the center (i.e., high degree) node (3) Star graph with the initialization on a leaf (i.e., high degree) node.

Fig.~\ref{fig:ComStars_std} shows the expected $L^1$ convergence over $20$ runs enveloped by the standard of deviation. The enveloping curves can be seen as horizontal error bars. Fig.~\ref{fig:ComStar_size} illustrates the convergence in each case as a function of the size of the graph.

\begin{figure}[H]
	\begin{centering}
		\includegraphics[scale=0.35]{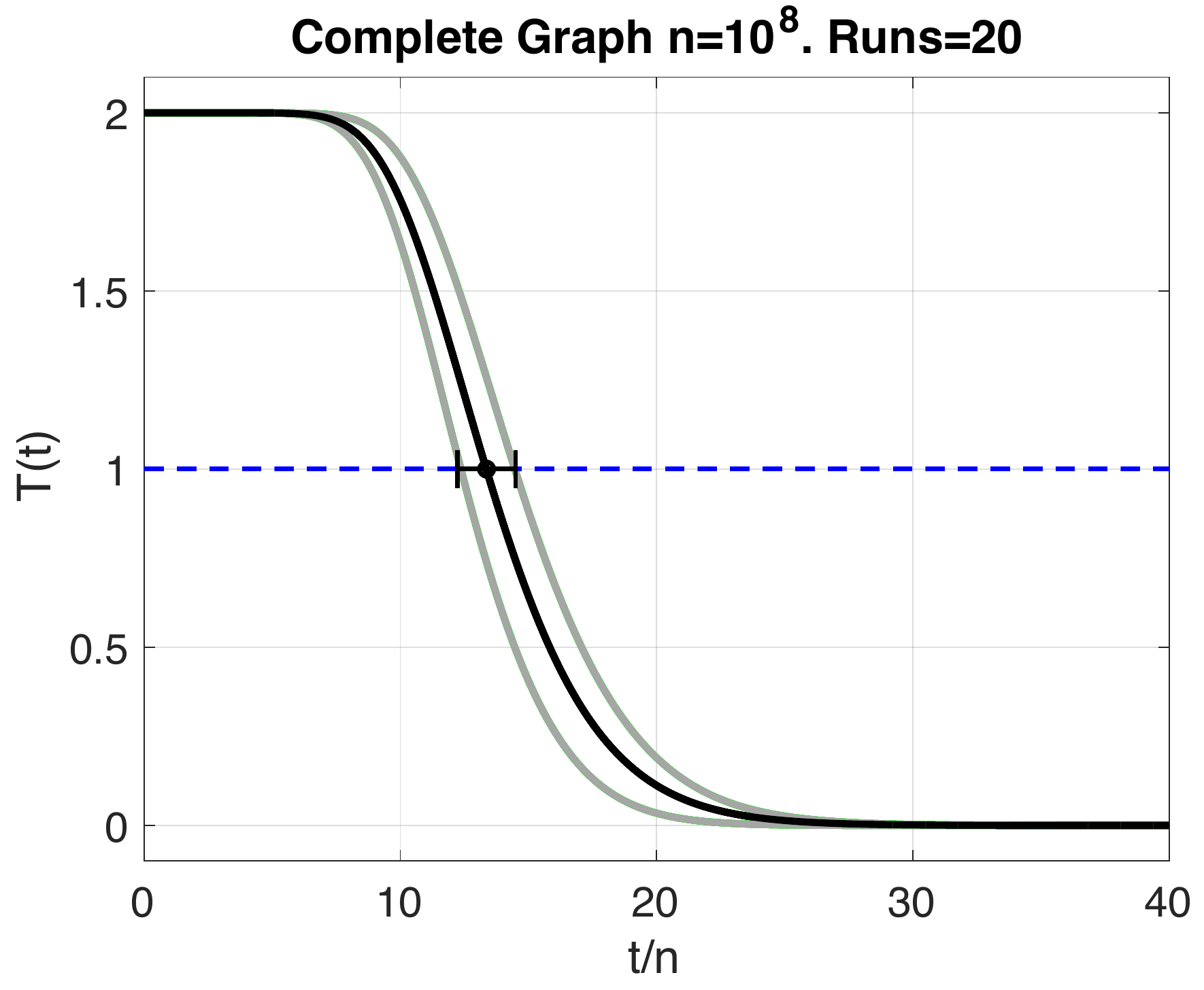}\includegraphics[scale=0.35]{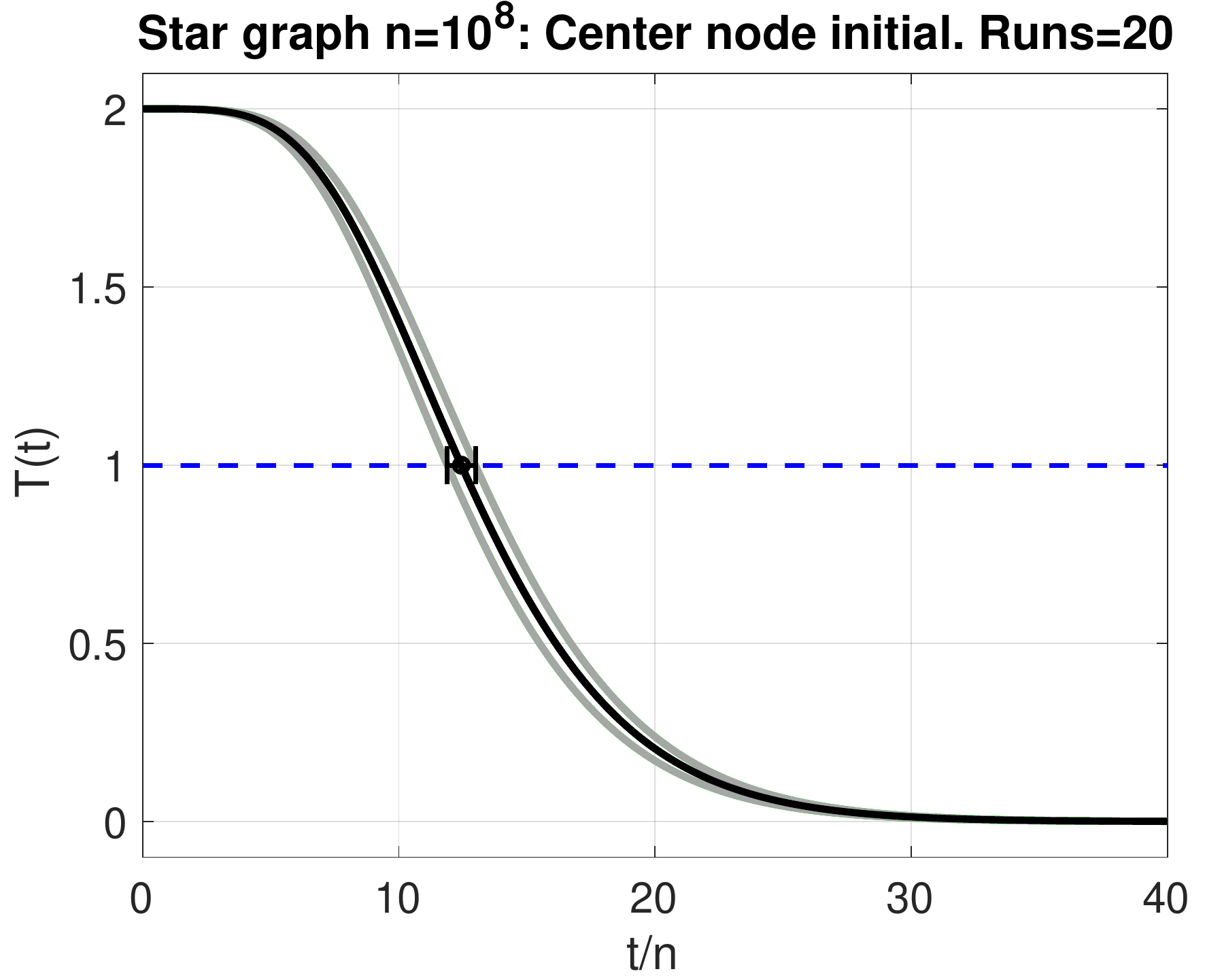}\\ \includegraphics[scale=0.35]{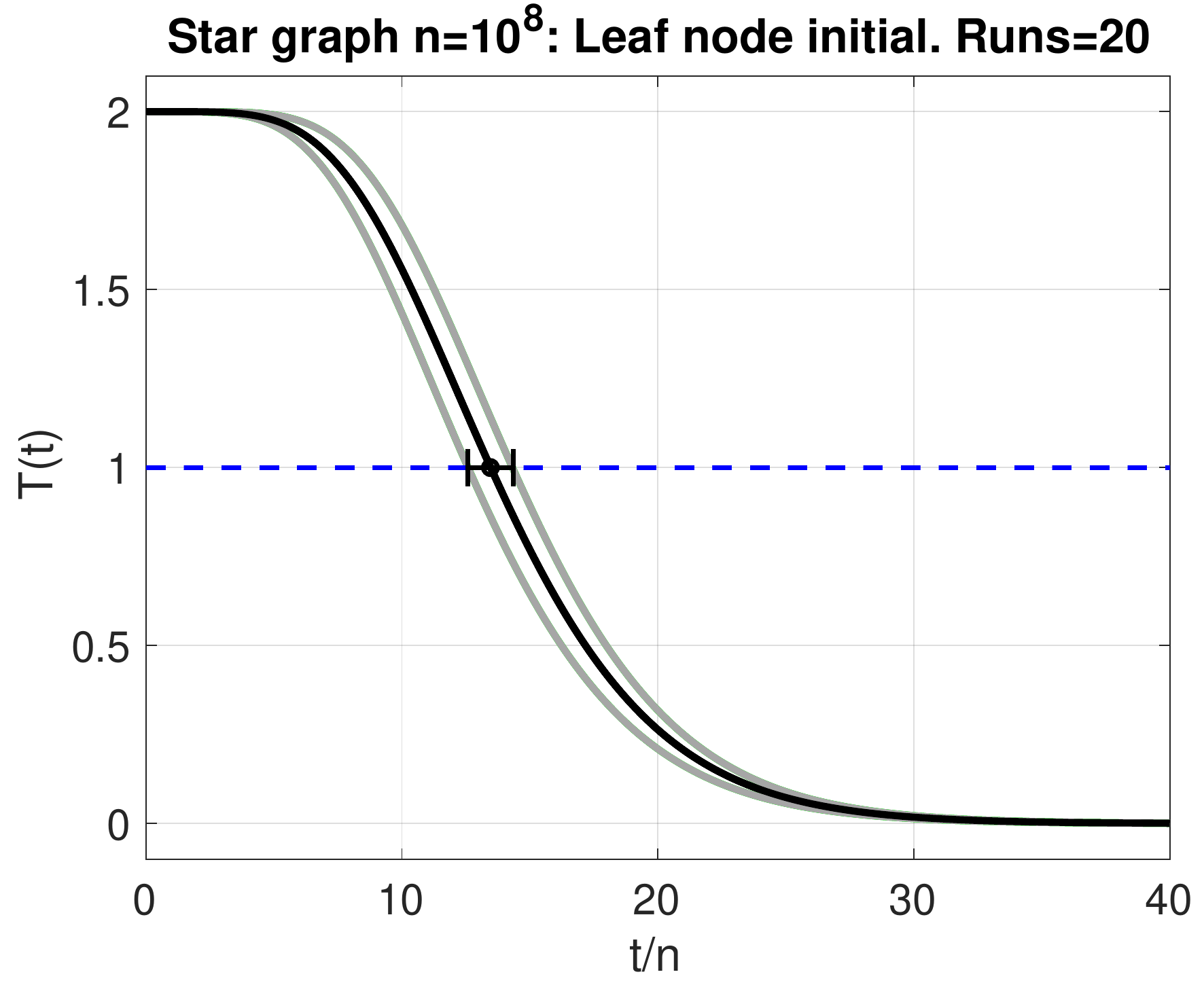}
		\par\end{centering}
	\caption{\label{fig:ComStars_std}Expected $L^1$ enveloped by the standard of deviation over $20$ runs.}
\end{figure}

\begin{figure}[H]
	\begin{centering}
		\includegraphics[scale=0.35]{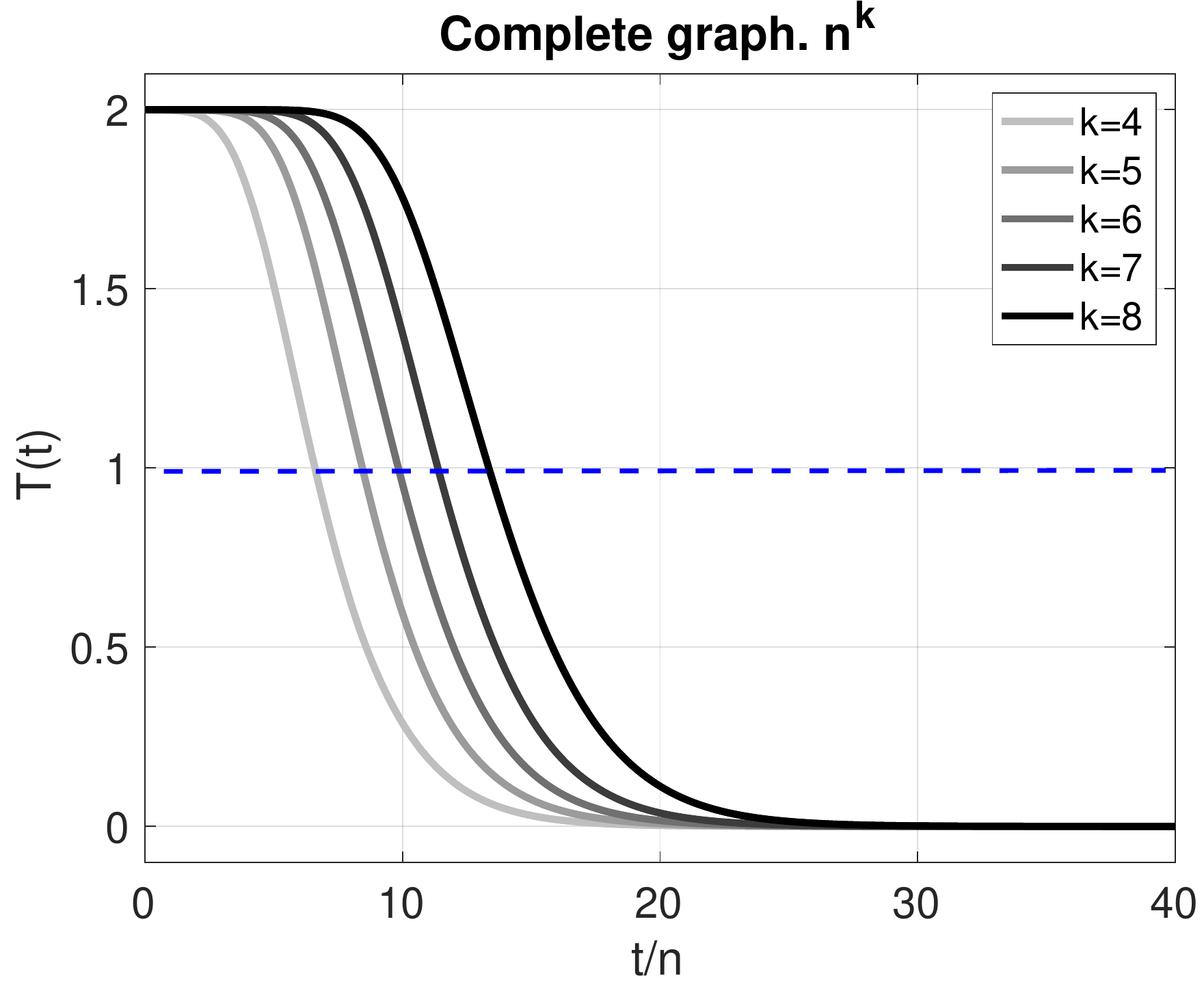}\includegraphics[scale=0.35]{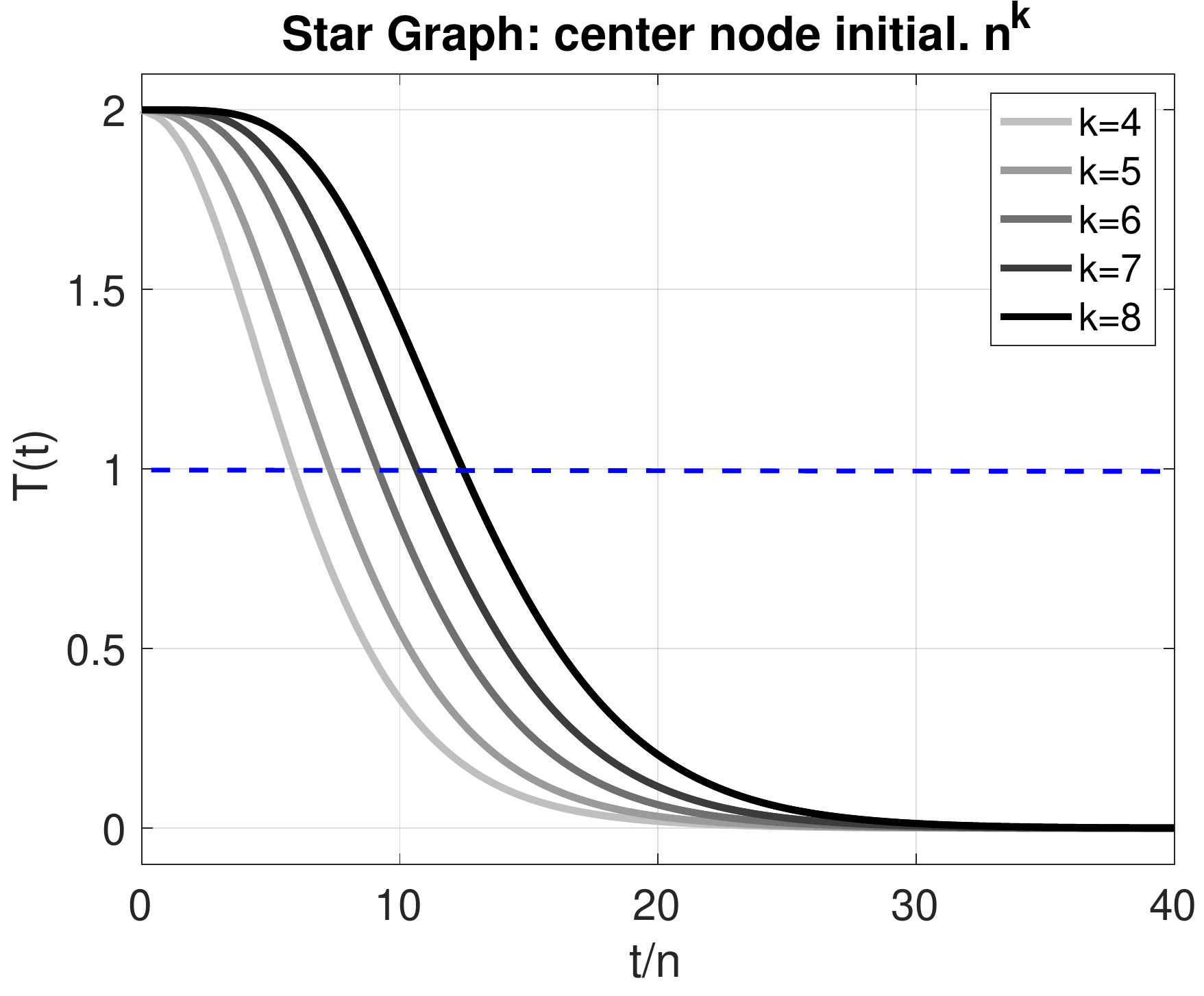}\\ \includegraphics[scale=0.35]{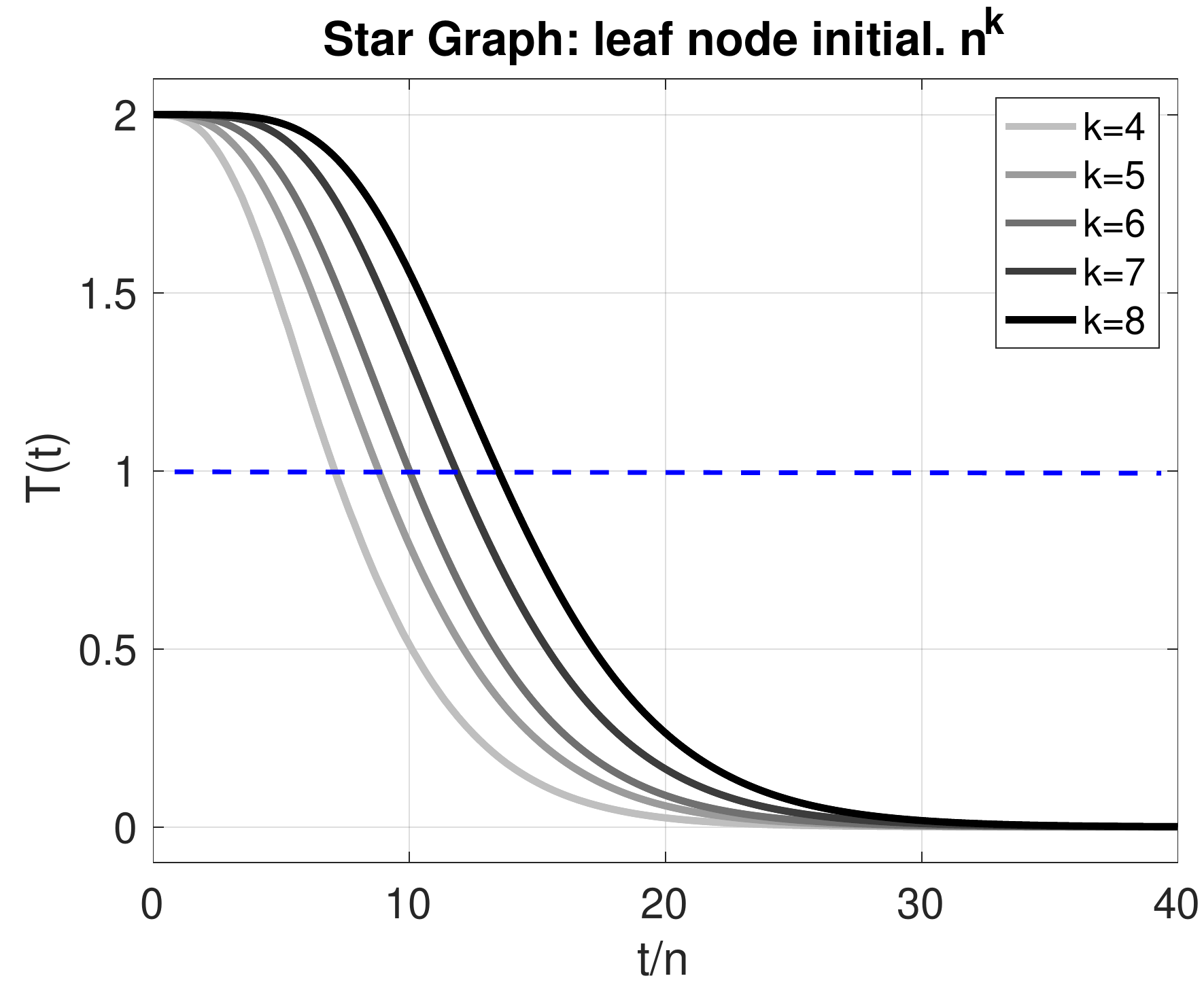}
		\par\end{centering}
	\caption{\label{fig:ComStar_size}Convergence in time for various graph sizes.}
\end{figure}

\begin{figure}[H]
	\begin{centering}
		\includegraphics[scale=0.35]{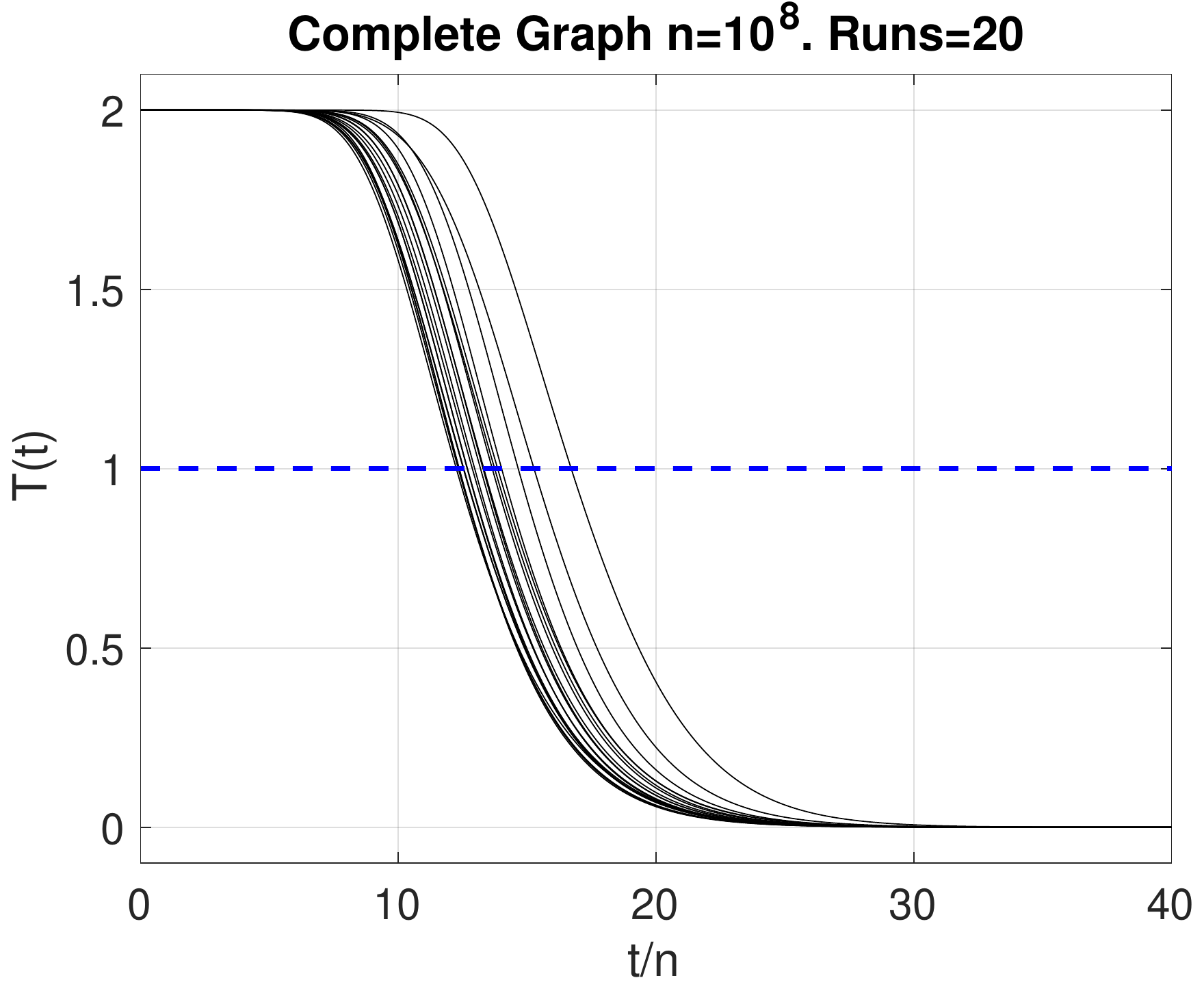}\includegraphics[scale=0.35]{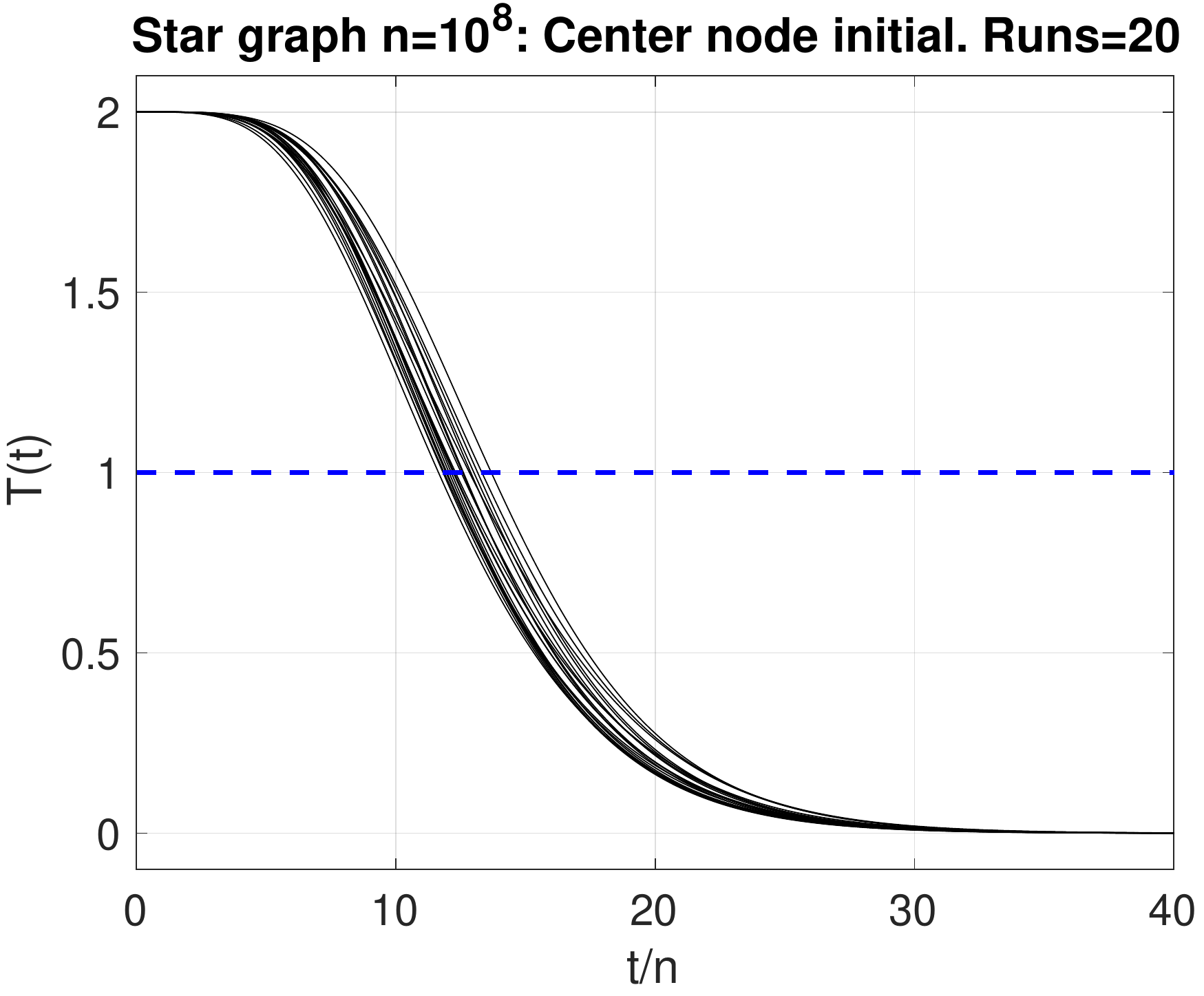}\\ \includegraphics[scale=0.35]{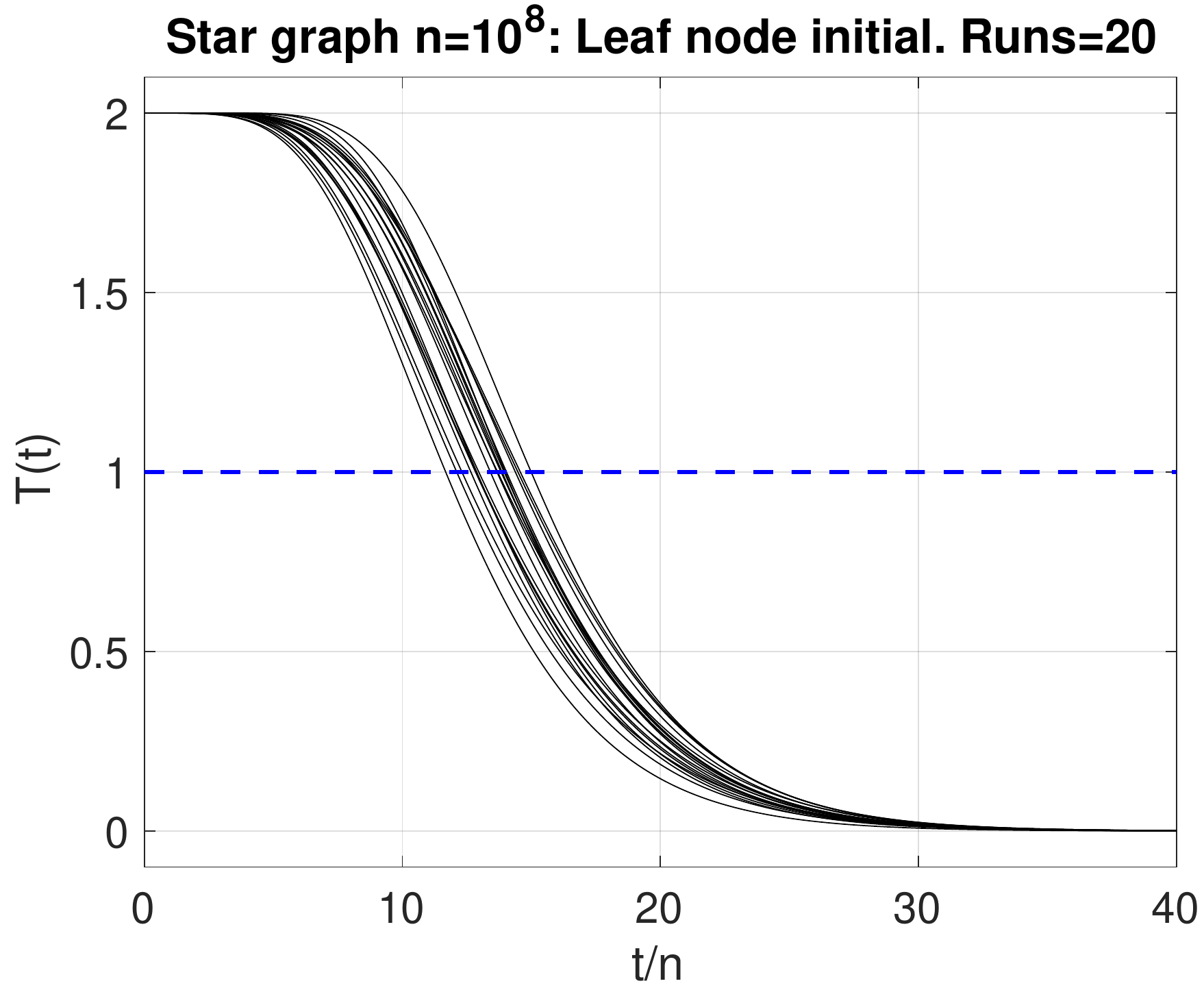}
		\par\end{centering}
	\caption{\label{fig:ComStar_size}Convergence in time for various graphs and $20$ distinct runs per graph.}
\end{figure}

\section{A related random process and the extremal nature of the complete graph
	with respect to it} \label{sec:MC2}
It is shown in Corollary \ref{cor: fastest L1} that the complete graph has asymptotically the smallest $L^1$-mixing time among all undirected, connected graphs when the number of nodes goes to infinity. However, for a fixed number of nodes $n$ and a fixed number of iterations $t$, there exists other graphs that mix faster than the complete graph. For example, Figure \ref{fig:EDelta1} shows the star graph mixes faster than the complete graph when the number of iterations is small. In this section, we introduce a related process for which we will
prove that the complete graph is extremal. The new process will be
identical to the original, slowed down
by a factor of $|E(G)|/{n\choose 2}$.

Given an undirected and connected graph $G$ with $n$ vertices, the new process is defined as follows: The process starts from an initial vector $v_{0}\in \bR^n$. At each discrete time-step the process picks
a pair of nodes $(i,j)$ uniformly at random from all ${n\choose 2}$ pairs 
of nodes
with $i\neq j$, and performs the averaging step if $(i,j)\in E(G)$. If $(i,j)\notin E(G)$, the process simply keeps 
the current $v$ (i.e. it ``does nothing'').  

The new process clearly coincides with the original one 
for the complete
graph, $K_{n}$. It however 
runs slower on 
general graphs $G$, 
due to the steps when it idles
(that constitute
$1 - |E(G)|/{n\choose 2}$ fraction of steps, on expectation).

\subsection{$L^{1}$-norm is monotonically non-increasing}
Recall that if $(i,j)\in E(G)$, 
we update the vector $(v_{1},\dots,v_{i},\dots,v_{j},\dots,v_{n})$
by replacing $v_i$ and $v_j$ by $\frac{v_{i}+v_{j}}{2}$. The only stationary state of the process is the vector $\bar v =(a,a,\dots,a)^T$. We define
\[
T_{G}(t)=\sum_{j=1}^{n}|v_{j}(t)-a|
\]
where we have put the graph $G$ as subscript on purpose.
By applying the averaging process to the $i^{\rm th}$ and $j^{\rm th}$ entry
when $(i,j)\in E(G)$, we obtain
\begin{eqnarray*}
	T_{G}(t+1) & = & T_{G}(t)-|v_{i}(t)-a|-|v_{j}(t)-a|+2\;
	\left|\frac{v_{i}(t)+v_{j}(t)}{2}-a
	\right| \; \le  \; T_{G}(t)
\end{eqnarray*}
(The above argument is same for 
the original process 
and the new one.) 
Further, $T_{G}(t+1)$
is clearly the same as
$T_{G}(t)$, when the $(i,j)$ 
pair picked is not in $E(G)$.
Thus monotonicity holds in the
case of the new process too.

\subsection{The complete graph is extremal for the new process}

We now prove that the complete graph converges to the averaging vector fastest in $L^1$ distance under the new dynamics, i.e. for any connected graph $G$ with $n$ nodes and for $K=K_{n}$
we have $\mathbb{E}_{v}\left[T_{K}(t)\right]\le\mathbb{E}_v\left[T_{G}(t)\right]$ for any initialization $v$. This will in turn imply that for all $t$.
$$\sup_{\lVert v\rVert_1 = 1}\mathbb{E}_{v}\left[T_{K}(t)\right]\le \sup_{\lVert v\rVert_1 = 1}\mathbb{E}_v \left[T_{G}(t)\right]$$

This implies that if we start the new process with the same initial vector, $v$, on $K_{n}$, respectively
on some other graph on $n$ nodes,
then after step $t$ the state vector for the complete
graph will be closer to $\bar{v}$, than the corresponding state vector for the other 
graph (in terms of the expected distance).
That is, the complete graph mixes the fastest in a strong sense.

\begin{thm}\label{thm: complete fastest, new process}
	(The complete graph is extremal)\label{Thm:Expected_1-norm} Let $G_n$ be any connected graph with $n$ nodes and $K=K_n$ be the $n$-node complete graph. Let $v$ be  an arbitrary  initial vector. We have  
	$\mathbb{E}_v [T_{K}(t)]\le\mathbb{E}_v [T_{G}(t)]$
	for all $t\ge0$ and any connected graph $G$, where expectation is
	with respect to the sequence of random choices of pair of nodes $1\le i<j\le n$. 
\end{thm}
\begin{proof}
	Let $v(0):= v = (v_{1},\dots,v_{n})$  be an arbitrary initial vector
	for both processes, and let $v_{K}(t)$ and $v_{G}(t)$ be the state
	vectors for $K=K_{n}$ 
	and $G$, respectively, at time $t$ of the process. We prove the theorem
	by induction on $t$. 
	
	\medskip
	
	\noindent{\it Case $t=0$:} The statement is trivial, since $T_{K}(0)=T_{G}(0)=||v(0)-\bar v||_{1}$.
	
	\medskip
	
	\noindent{\it Inductive step:}
	We prove $\mathbb{E}[T_{K}(t)]\le\mathbb{E}[T_{G}(t)]$
	by conditioning on the first step, and relying 
	on $\mathbb{E}_v[T_{K}(t-1)]\le\mathbb{E}_v[T_{G}(t-1)]$
	for all $v\in \bR^n$, as given by the induction hypothesis. If in the first step a pair $(i,j)\in E(G)$
	is picked, then $v(1)$ is the same for both graphs and we have reduced
	the number of steps by one. The proof, conditioned on this event, becomes a trivial consequence of the induction hypothesis.
	
	Next suppose, that in the first step an
	$(i,j)\not\in E(G)$ is picked. Without loss of generality let us
	assume that $(i,j) =(1,2)$. We have:
	\[
	v_{K}(1)\;=\; v_{1,2} \; = \; \left(\frac{v_{1}+v_{2}}{2},\frac{v_{1}+v_{2}}{2},v_{3},\dots,v_{n}\right),\]
	whereas $v_{G}(1)=v(0)=v$. Observe that 
	\[
	v_{1,2}=\frac{1}{2}(v_{1},v_{2},v_{3,}\dots,v_{n})+\frac{1}{2}(v_{2},v_{1},v_{3},\dots,v_{n}).
	\]
	Now by the convexity lemma proved in Section \ref{subsec:general}  we have that 
	\[
	\mathbb{E}_{v_{1,2}}\left[T_{K}(t-1)\right]\le\frac{1}{2}\mathbb{E}_{(v_{1},v_{2},v_{3,}\dots,v_{n})}\left[T_{K}(t-1)\right]+\frac{1}{2}\mathbb{E}_{(v_{2},v_{1},v_{3,}\dots,v_{n})}\left[T_{K}(t-1)\right].
	\]
	By the symmetry of
	the complete graph $\mathbb{E}_{(v_{1},v_{2},v_{3,}\dots,v_{n})}\left[T_{K}(t-1)\right] = \mathbb{E}_{(v_{2},v_{1},v_{3,}\dots,v_{n})}\left[T_{K}(t-1)\right]$,
	so we get
	\[
	\mathbb{E}_{v_{1,2}}\left[ T_{K}(t-1)\right]\le\mathbb{E}_{v}\left[T_{K}(t-1)\right]
	\; \le \; \mathbb{E}_{v}\left[T_{G}(t-1)\right],
	\]
	where the last inequality is by induction.
	Since we have considered both cases where the first randomly chosen edge is in $E(G)$ and when it is not, and in both cases the conditional expectation of $T_K(t)$ is less than that of $T_G(t)$, the proof is complete.
\end{proof}

\begin{rem}
	Although at first it might seem so,
	$T_{K}(t)\le T_{G}(t)$ does not hold for 
	all points of the event space, only on expectation
	(here we couple the processes for $K=K_{n}$ and for $G$, starting from the same vector $v$,
	in the obvious way, since they both run on infinite edge sequences, where the edges are
	taken from the complete graph). An example 
	is furnished by the following. Let $G$
	be the almost complete graph with missing edge $(1,2)$. Suppose the initial
	vector is $(1,0,\dots,0)$ and at the first step of the Markov chain
	the pair $(i,j)=(1,2)$ is picked. The state vector over the complete
	graph becomes $(1/2,1/2,0,\dots,0)$ and the state vector over $G$
	remains to be $(1,0,\dots,0)$. Suppose that all the subsequent
	moves for a long time pick pairs $(i,j)$ where $i,j\ne2$ (i.e.
	all moves avoid edges incident to $2$). Then after a long time the
	state vector over the complete graph is close to $(\text{\ensuremath{\frac{1}{2(n-1)}}},\frac{1}{2},\text{\ensuremath{\frac{1}{2(n-1)}}},\dots,\text{\ensuremath{\frac{1}{2(n-1)}}})$,
	whereas the state vector over $G$ is close to $(\frac{1}{n-1},0,\frac{1}{n-1},\dots,\frac{1}{n-1})$.
	Therefore, $T_{K}(t\gg1)\approx1-2/n$, and $T_{G}(t\gg1)\approx 2/n$. 
\end{rem}

\begin{rem}
	When we compare Theorem \ref{thm: complete fastest, new process}
	with Corollary \ref{cor: fastest L1}, we will notice that the latter only asserts 
	that (under the original process) the complete graph mixes fastest {\it asymptotically}, i.e. when 
	the number of nodes goes to infinity. When $n$ is fixed and $t$ is small, 
	the star often mixes slightly faster than the complete graph. In contrast, 
	in the case of the new process the complete graph mixes fastest for 
	every $t$ and $n$.
\end{rem}

\begin{figure}
	\begin{centering}
		\includegraphics[scale=0.4]{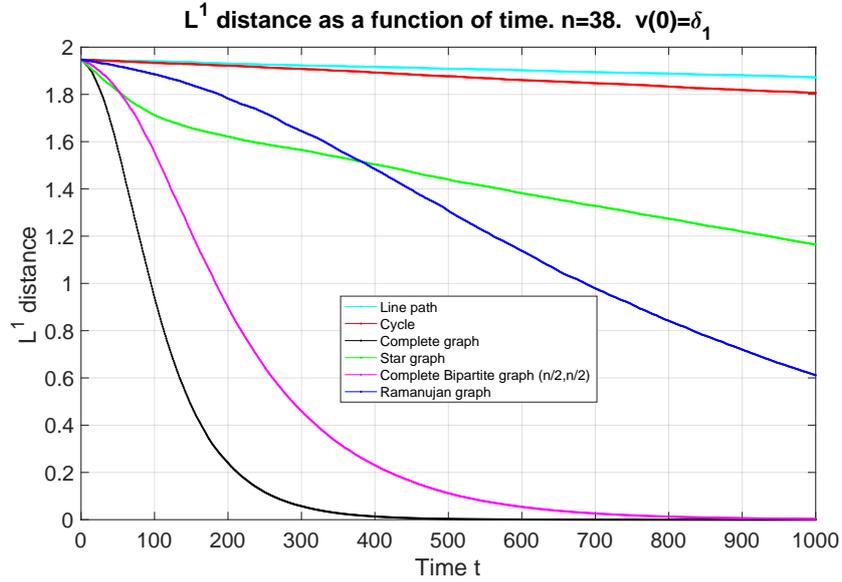}\\
		\includegraphics[scale=0.4]{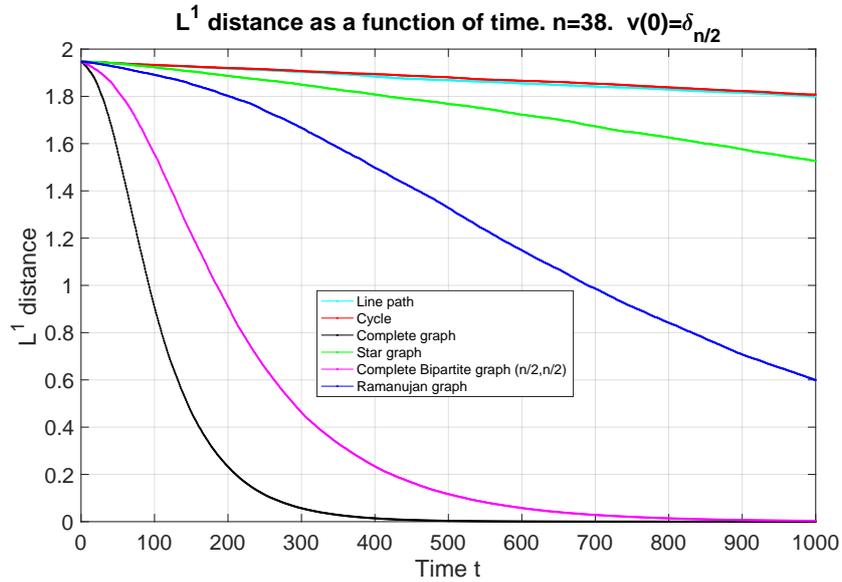}
		\par\end{centering}
	\caption{\label{fig:EDelta}Numerical illustration of $L^1$ convergence (see Theorem  \ref{Thm:Expected_1-norm}). The two plots are different in the initialization of the state, which is mainly relevant in comparing the path and cycle graphs, and the star graph as discussed in the text. The size $n=38$ was dictated by the Ramanujan graph we picked.}
\end{figure}

\subsection{Numerical illustrations for the new process}\label{sec:Numerics2}
In Figure \ref{fig:EDelta} we show $\mathbb{E}[T_{G}(t)]$ vs.
$t=\{0,1,\dots,\tau\}$ 
for various graphs $G$. 

\section*{Acknowledgments}
RM acknowledges the support of the Frontiers institute and the support of MIT-IBM AI lab through the grant ``Machine Learning in Hilbert Spaces''. Guanyang Wang would like to thank Jun Yan, Yuchen Liao, Yanjun Han, and Fan Wei for  helpful discussions.

\section{Proofs}\label{sec:proof}
\subsection{Preparation}\label{subsec:general}
In this subsection, we collect several useful results for the averaging process.
\subsubsection{Averaging matrix}

For a graph $G = (V,E)$ with $n$ nodes, recall that the Laplacian 
is defined as $L = D-A$
where $A$ and $D$ are the adjacency matrix and degree matrix of $G$ respectively. 
Recall that $v(t)$ is a vector-valued random variable determined by the $t^{\rm th}$ step of the averaging process. We prove a formula for the expectation of $v(t)$.

\begin{prop}[Averaging matrix]\label{prop:averaging matrix}
	\begin{eqnarray}
		\mathbb{E}[v(t)] & = & M^{t}v(0)\; ,\label{eq:E_v(t)}
	\end{eqnarray}
	where
	\begin{equation}
		M=I-\frac{1}{2|E|}L\;.\label{eq:M_laplac}
	\end{equation}
\end{prop}

\begin{proof}
	We first prove \eqref{eq:E_v(t)} for $t = 1$. Fix any index $i$, the expected value of $v_i(1)$ equals 
	\[
	\bE[v_i(1)] = (1 - \frac{d_i}{\lvert E \rvert}) v_i(0) + \frac{1}{\lvert E\rvert} \sum_{(i,j)\in E}\frac{v_i + v_j}{2}\; .
	\]
	Writing the above formula in the matrix form yields $\bE[v(1)] = Mv(0)$, which concludes the case $t = 1$. The general case follows from writing $\bE[v(t)]$ as $\bE\left[\bE[v(t)\mid v(t-1)]\right]$ iteratively and using the linearity of expectation. 
\end{proof}

Observe that $M$ is a doubly stochastic matrix. Assuming that the eigenvalues of $L$ are \\ $0=\lambda_{1}\le\lambda_{2}\le\cdots\le\lambda_{n}$, the eigenvalues of $M$ are $\mu_{i}=1-\frac{\lambda_{i}}{2|E|}$
with $1\ge\mu_{1}\ge\mu_{2}\ge\cdots\ge\mu_{n}$. The eigenvalue gap
of the Laplacian, $\lambda_{2}$,  controls the convergence of $\mathbb{E}[v(t)]$
to the uniform distribution with respect to $t$.

\subsubsection{Monotinicity and other useful properties}\label{subsec:monoton}
In this section, we will state several useful properties of the  averaging process. First, the averaging process preserves linear combination, with consequences: 
\begin{prop}\label{prop:lincomb}
	Let $u_{0}$, $v(0)$ be two starting vectors, 
	$\lambda, \mu \in \mathbb{R}$ be two real numbers,
	$u(t), v(t)$ and $w(t)$ are results of the $t$ step averaging process with starting vectors $u_{0}$, $v(0)$ and
	$\lambda u_{0} + \mu v(0)$, respectively.
	Let $r$ be a fixed edge-sequence, i.e. a fix
	random branch for the process. Then 
	\[
	w(t)(r) = \lambda u(t)(r) + \mu v(t)(r) 
	\]
\end{prop}

As a consequence we get:

\begin{prop}[Translation and scaling]\label{prop:translation and scaling}
	Let $v(0)\in \bR^n$ be an arbitrary initial vector, $c, \lambda\in\mathbb{R}$. 
	Let $w_{0} = \lambda v(0) - c$ be another initial vector. Then:
	\[
	w(t) \stackrel{\mathcal L}{=} \lambda v(t) - c
	\]
	where $\stackrel{\mathcal L}{=}$ denotes equality in the distribution sense. 
\end{prop}
The proofs of Propositions
\ref{prop:lincomb} and
\ref{prop:translation and scaling}
are straightforward and we omit them here. Proposition \ref{prop:translation and scaling}
allows us to normalize the initial vector 
in convenient ways, for instance when we study the $L^{2}$ norm, we can assume that 
the sum of the coordinates of $v(0)$ is zero and $||v(0)||_{2}=1$.

\begin{prop}[Monotonicity]\label{prop: monotonically}
	Let $f: \bR \rightarrow\bR$ be a convex real function. Define 
	\[
	S_f(v)\equiv \sum_{i=1}^n f(v_i)
	\]
	Then the sequence $S_f(v(i))$ $i=0,1,\ldots$
	of random variables
	is monotonically non-increasing: 
	\begin{align}\label{eqn: monotonically}
		S_f(v(0)) \geq S_f(v(1)) \geq \cdots \geq  S_f(v(t)) \geq \cdots
	\end{align}
\end{prop}
\begin{proof}
	Fix any time step $k$, and assume at time $k$ an random edge $(I,J)$ is chosen, then we have
	$$S_f(v(k+1) - S_f(v(k)) = 2 f\left(\frac{v_I(k) + v_J(k)}{2}\right) - \left(f(v_I(k)) + f(v_I(k))\right) \leq 0$$
	where the last inequality follows from the convexity of $f$. 
\end{proof}

\begin{cor}\label{cor:monotonically, special function}
	Let $p\geq 1$. With the choice of $f(x)=\lvert x \rvert^p$,
	Proposition \ref{prop: monotonically} implies 
	that the $p$-th power of the $L^p$ norm of $v(t)$ is monotonically non-increasing.
\end{cor}

\subsection{Proofs for Section \ref{subsec:L2}}
\subsubsection{Proof of Theorem \ref{Thm:L2_convg}}\label{subsubsec: proof L2}

We start with the following Proposition. First notice that replacing $(v_{i},v_{j})$ by $\left(\frac{v_{i}+v_{j}}{2},\frac{v_{i}+v_{j}}{2}\right)$
decreases $S(t)=\sum_{i=1}^{n}v_{i}^{2}(t)$, the square of the $L^{2}$ norm, exactly by $(v_{i}-v_{j})^{2}/2$. Using this we show that on expectation 
$S(t)$ changes as:
\begin{prop}
	\label{Prop:L2_monot}$\mathbb{E}\left[S(t)|v(t-1)\right]=v^{T}(t-1)\,M\,v(t-1)$
\end{prop}
\begin{proof}
	Since each edge is picked uniformly at random we have 
	\[
	S(t-1)-\mathbb{E}\left[S(t)|v(t-1)\right]=\frac{1}{|E|}\sum_{(i,j)\in E}\frac{\left(v_{i}(t)-v_{j}(t)\right)^{2}}{2}=\frac{1}{2|E|}v^{T}(t-1)\,L\,v(t-1)\;,
	\]
	where the latter follows from the quadratic form expression of the
	Laplacian. Using Eq.~\eqref{eq:M_laplac} we conclude that 
	\begin{equation}
		\mathbb{E}\left[S(t)|v(t-1)\right]=S(t-1)-\frac{1}{2|E|}v^{T}(t-1)\,L\,v(t-1)=v^{T}(t-1)\,M\,v(t-1)\; .\label{eq:Prop_L2_key}
	\end{equation}
\end{proof}
This leads us to the upper and lower bounds about the $L^{2}$ convergence.

\begin{proof}[Proof of Theorem \ref{Thm:L2_convg}]
	With no loss of generality, we take $\sum_{i=1}^{n}v_i(0)=0$.
	Since $M$ is a real symmetric matrix we write its spectral decomposition $
	M=\sum_{i=1}^{n}\left(1-\frac{\lambda_{i}}{2|E|}\right)u_{i}u_{i}^{T},$
	where $\left(1-\frac{\lambda_{i}}{2|E|}\right)$ are the eigenvalues
	and $u_{i}$ are the corresponding set of orthonormal eigenvectors.
	$M$ is a doubly stochastic matrix that can be viewed as a transition
	matrix of an irreducible Markov chain. The standard Markov chain theory
	guarantees a unique largest eigenvalue $1$ with the corresponding
	eigenvector $u_{1}=\frac{1}{\sqrt{n}}(1,1,\dots,1)$. We now prove
	the upper and lower bound on the $L^{2}$ norm.
	
	For the upper-bound, from Eq.~\eqref{eq:Prop_L2_key} in Prop. \ref{Prop:L2_monot}
	and the spectral decomposition of $M$ we have
	\begin{eqnarray*}\label{eq:ES_t}
		\mathbb{E}\left[S(t)|v(t-1)\right] & = & \sum_{i=2}^{n}\left(1-\frac{\lambda_{i}}{2|E|}\right)|u_{i}^{T}v(t-1)|^{2}\\
		& \le & \left(1-\frac{\lambda_{2}}{2|E|}\right)\sum_{i=2}^{n}|u_{i}^{T}v(t-1)|^{2}=\left(1-\frac{\lambda_{2}}{2|E|}\right)S(t-1)\;,
	\end{eqnarray*}
	where the first equality uses the zero-mean choice $\sum_{i=1}^{n}v_{i}(t-1)=0$. Taking an expectation of both sides with respect to $v(t-1)$ yields $
	\mathbb{E}[S(t)]\le\left(1-\frac{\lambda_{2}}{2|E|}\right)\mathbb{E}[S(t-1)].$ Solving the foregoing recursion we find $\mathbb{E}[S(t)]\le\left(1-\frac{\lambda_{2}}{2|E|}\right)^{t}\mathbb{E}[S(0)]$.
	Since $\mathbb{E}[S(t)]=\left\Vert v(t)\right\Vert _{2}^{2}$ we have
	\[
	\sqrt{\mathbb{E}[\left\Vert v(t)\right\Vert _{2}^{2}]}\le\left(1-\frac{\lambda_{2}}{2|E|}\right)^{t/2}\left\Vert v(0)\right\Vert _{2}\;.
	\]
	For the lower-bound, we prove it by taking the initial condition to be the
	second eigenvector $v(0)=u_{2}$ (if there is a multiplicity, then
	pick any in the eigenspace). Recall that by Eq.~\eqref{eq:E_v(t)} we
	have 
	\[
	\mathbb{E}[v(t)]=\left(1-\frac{\lambda_{2}}{2|E|}\right)^{t}u_{2}\;.
	\]
	This along with the fact that $\left\Vert u_{2}\right\Vert _{2}^{2}=1$
	give
	\begin{eqnarray*}
		\mathbb{E}[\left\Vert v(t)\right\Vert _{2}^{2}] & = & \sum_{i=1}^{n}\mathbb{E}\left[v_{i}^{2}(t)\right]\ge\sum_{i=1}^{n}\left(\mathbb{E}[v_{i}(t)]\right)^{2}\\
		& = & \left(1-\frac{\lambda_{2}}{2|E|}\right)^{2t} \sum_{i=1}^{n}u_{2,i}^{2}=\left(1-\frac{\lambda_{2}}{2|E|}\right)^{2t}.
	\end{eqnarray*}
	Since $u_{2}^{T}u_{1}=0$, we have that $\sum_{i=1}^{n}u_{2,i}=0$
	and hence $\bar{v}=0$. The desired lower-bound then follows
	\[
	\sup_{\left\Vert v(0)\right\Vert _{2}=1}\sqrt{\mathbb{E}[\left\Vert v(t)-\bar{v}\right\Vert _{2}^{2}]}\ge\underbrace{\sqrt{\mathbb{E}[\left\Vert v(t)\right\Vert _{2}^{2}]}}_{v(0)=u_{2}}=\left(1-\frac{\lambda_{2}}{2|E|}\right)^{t}.
	\]
\end{proof}
\subsubsection{Proof of Corollary \ref{Cor:L2mixing}}\label{subsubsec: corL2}
\begin{proof}[Proof of Corollary \ref{Cor:L2mixing}]
	Notice that $\left\Vert v(0)\right\Vert _{2}^{2} = \left\Vert v(0)-\bar{v}\right\Vert _{2}^{2}+\left\Vert \bar{v}\right\Vert _{2}^{2}=1$
	and we have $\left\Vert v(0)-\bar{v}\right\Vert _{2}^{2}\le1$. Eq.
	\eqref{eq:thmL2_Upp} in the first part of Theorem \ref{Thm:L2_convg}
	easily gives
	\[
	\sqrt{\mathbb{E}[\left\Vert v(t)-\bar{v}\right\Vert _{2}^{2}]}\le\left(1-\frac{\lambda_{2}}{2|E|}\right)^{t/2}\left\Vert v(0)-\bar{v}\right\Vert _{2}\le e^{-\frac{t}{4|E|}\lambda_{2}}\quad,
	\]
	where we used $\left\Vert v(0)-\bar{v}\right\Vert _{2}=1$ and $1-x\le e^{-x}$.
	Solving for $t$ in $\epsilon=e^{-\frac{t}{4|E|}\lambda_{2}}$ gives
	the upper bound $4\gamma(G)\log(\epsilon^{-1})$ in
	Eq.~\eqref{eq:L2_Mixing}.
	
	From Eq.~\eqref{eq:thmL2_Low} in the second part of Theorem \ref{Thm:L2_convg}
	we have 
	\begin{eqnarray*}
		\sup_{\left\Vert v(0)\right\Vert _{2}=1}\sqrt{\mathbb{E}\left[\left\Vert v(t)-\bar{v}\right\Vert _{2}^{2}\right]} & \ge & \left(1-\frac{\lambda_{2}}{2|E|}\right)^{t}
		=  \exp\left[t\log\left(1-\frac{\lambda_{2}}{2|E|}\right)\right]\\
		& \ge & \exp\left[\frac{-t\lambda_{2}}{2|E|\left(1-\frac{\lambda_{2}}{2|E|}\right)}\right]
	\end{eqnarray*}
	where in the last inequality we used $\log(1-x)>-x/(1-x)$. Once more
	solving for $t$ in $\epsilon=\exp\left[\frac{-t\lambda_{2}}{2|E|\left(1-\frac{\lambda_{2}}{2|E|}\right)}\right]$
	yields the lower bound in Eq.~\eqref{eq:L2_Mixing}
	\[
	t=\left(1-\frac{\lambda_{2}}{2|E|}\right)\frac{2|E|}{\lambda_{2}}\log(\epsilon^{-1}) = (2\gamma(G) - 1) \log  (\epsilon^{-1} ).
	\]
\end{proof}

\subsubsection{Proof of the examples in Table \ref{tab:L2}}\label{subsubsec: proof examples L2}
With Theorem \ref{Thm:L2_convg} and Corollary \ref{Cor:L2mixing} in hand, now we can estimate the $L^2$ convergence speed for several different graphs.

\begin{eg}[Complete graph]\label{eg: complete, L2}
	The complete graph $K_n$ has $n$ nodes and  $|E(K_n)| = \frac{n(n-1)}{2}$ edges. It is well-known that the eigenvalues of the graph Laplacian $L(K_n)$ are: 
	$$ \lambda_1 = 0\;; \qquad \lambda_2 = \cdots = \lambda_n=n\;,$$
	and Theorem \ref{Thm:L2_convg} and Corollary \ref{Cor:L2mixing} respectively imply:
	\begin{align*}
		\bigg(1- \frac 1{n-1}\bigg)^t \leq \sup_{\{v(0): \lVert v(0)\rVert_2 = 1\}} \sqrt{\bE[\lVert v(t) - \bar v\rVert_2^2]}  \leq \bigg(1- \frac 1{n-1}\bigg)^{t/2},
	\end{align*}
	and
	\begin{align*}
		(n-2)\log(\epsilon^{-1}) \leq &t_{\epsilon, 2}(K_n) \leq (2n-2) \log(\epsilon^{-1}).
	\end{align*}
	Our bounds match the magnitude of the result in \citep[Proposition 2.1]{chatterjee2019phase}. 
\end{eg}

\begin{eg}[Cycle with $n$ nodes]\label{fig:C_5}
	% \label{eg: cycle, L2}
	A cycle graph $C_n$ is a graph with $n$ nodes that consists of a single cycle. We know $|E(C_n)| = n$ and the Laplacian $L(C_n)$ can also be diagonalized explicitly, with  eigenvalues $2 - 2\cos(kw)$ where $w \equiv \frac{2\pi}{n}$ and $0\leq k\leq \frac{n}{2}$.  Since $\cos(w) = \cos(\frac{2\pi}{n})\sim 1 - \frac{2\pi^2}{n^2}$ when $n$ is large, from Corollary \ref{Cor:L2mixing} we have that $t_{\epsilon,2}(C_n) = \Theta(n^3)\log (\epsilon^{-1})$. 
	
\end{eg}

\begin{eg}[Star graph]\label{eg: star, L2}
	A star graph $S_{n-1}$ is a tree with $n$ nodes. It contains  one root and $n-1$ leaves.  We know $|E(S_{n-1})| = n-1$ and the second smallest eigenvalue of $S_{n-1}$ equals $1$. Theorem \ref{Thm:L2_convg} and Corollary \ref{Cor:L2mixing} imply:
	\begin{align*}
		\bigg(1- \frac 1{2 (n-1)}\bigg)^t \leq \sup_{\{v(0): \lVert v(0)\rVert_2 = 1\}} \sqrt{\bE[\lVert v(t) - \bar v\rVert_2^2]}  \leq \bigg(1- \frac 1{2 (n-1)}\bigg)^{t/2}
	\end{align*}
	and 
	\begin{align*}
		(2n-3)\log (\epsilon^{-1}) \; \leq \; t_{\epsilon,2} \; \leq \;  (4n-4)\log (\epsilon^{-1}).
	\end{align*}
	Thus $t_{\epsilon,2}(S_{n-1}) = \Theta(n)\log(\epsilon^{-1})$.

\end{eg}

\begin{eg}[Binary tree]\label{eg: binary tree, L2}
	A balanced full binary tree $B_n$ has $n-1$ edges and  $\log_2(n+1)$ levels. It is difficult to explicitly diagonalize the graph Laplacian $L(B_n)$, but it is well known (from the Cheeger's inequality) that the second smallest eigenvalue $\lambda_2(B_n)= \Theta(\frac 1n)$. More precisely, from  \cite[Lemma 3.8]{guattery1994performance} we have
	\[
	\frac 1n < \lambda_2(B_n) < \frac 2n\;.
	\]
	And  Theorem \ref{Thm:L2_convg} and Corollary \ref{Cor:L2mixing} imply:
	\begin{align*}
		\bigg(1- \frac 1{n (n-1)}\bigg)^t < \sup_{\{v(0): \lVert v(0)\rVert_2 = 1\}} \sqrt{\bE[\lVert v(t) - \bar v\rVert_2^2]}  < \bigg(1- \frac 1{2n (n-1)}\bigg)^{t/2}
	\end{align*}
	and 
	\begin{align*}
		(n^2-n-1)\log (\epsilon^{-1}) \; < \; t_{\epsilon,2} \; < \;  4(n^2-n)\log (\epsilon^{-1}).
	\end{align*}
	Thus $t_{\epsilon,2}(B_{n}) = \Theta(n^2)\log (\epsilon^{-1})$.
\end{eg}

\subsection{Proofs for Section \ref{subsec:L1}}
\subsubsection{Proof of Theorem \ref{thm: t_1}}\label{subsubsec: proof L1}

\begin{proof}[Proof of Theorem \ref{thm: t_1}]
	For the upper bound, Theorem \ref{Thm:L2_convg} and Cauchy-Schwarz give:
	% \begin{align*}
	%  \bE[\lVert v(t) - \bar v\rVert_1] &=   \bE \left[ \sum_{i=1}^n\lvert v_i(t) - \bar v_i\rvert\right] \;  \leq \; \sqrt n \; \bE \left[\sqrt{\sum_{i=1}^n \lvert v_i(t) - \bar v_i \rvert^2 }\; \right] \\&\leq \sqrt{n} \; \sqrt{\bE[\lVert v(t) - \bar v\rVert_2^2]}  \leq \sqrt n\;  \bigg(1- {\lambda_{2} \over 2 |E|}\bigg)^{\frac t2}  \lVert v(0) - \bar v\rVert_2.
	% \end{align*}
	\begin{align*}
		\bE[\lVert v(t) - \bar v\rVert_1] &=   \bE \left[ \sum_{i=1}^n\lvert v_i(t) - \bar v\rvert\right] \;  \leq \; \sqrt n \; \bE \left[\sqrt{\sum_{i=1}^n \lvert v_i(t) - \bar v \rvert^2 }\; \right] \\&\leq \sqrt{n} \; \sqrt{\bE[\lVert v(t) - \bar v\rVert_2^2]}  \leq \sqrt n\;  \bigg(1- {1 \over 2 \gamma(G)}\bigg)^{t/2}  \lVert v(0) - \bar v\rVert_2.
	\end{align*}
	Since $\lVert v(0)\rVert_2\leq \lVert v(0)\rVert_1= 1$, we have 
	$\lVert v(0) -\bar v\rVert_2\leq 1$. Therefore
	solving for $t$ in $\sqrt n\;  \bigg(1- {1 \over 2 \gamma(G)}\bigg)^{t/2} = \epsilon$ gives us the upper bound in Eq.~\eqref{eqn: t_1, first formula}.

	For the lower bound, we take
	$v(0) = u_{2}/\lVert u_2 \rVert_1 $ (so $\bar v = \vec 0$) where $u_2$ is the  eigenvector of $L(G)$ corresponding to $\lambda_2$. After the normalization it is clear that $\lVert v(0) \rVert_1 = 1$.  Recall that
	$\bE[v(t)] = M^t v(0)$, which gives:
	\begin{align}\nonumber
		\bE[\lVert v(t)\rVert_1]  \; \geq \;&
		\sum_{i=1}^n \bigg|\bE [ v_i(t)] \bigg| 
		= \;  \lVert M^{t} v(0) \rVert_1
		\; = \; 
		\bigg\lVert \bigg(1- {1 \over 2 \gamma(G)}\bigg)^t  v(0)\bigg\rVert_1\\
		\; = & 
		\bigg(1- {1 \over 2 \gamma(G)}\bigg)^t 
		\geq \exp{\left[-\frac {t}{2\gamma(G)}/ \left(1 - \frac {1}{2\gamma(G)}\right)\right]}.
	\end{align}
	Again, solving for $t$ in $\exp{\left[(-\frac {t}{2\gamma(G)})/ (1 - \frac {1}{2\gamma(G)})\right]} = \epsilon$
	gives us the lower bound in Eq.~ \eqref{eqn: t_1, first formula}.
\end{proof}
\subsubsection{Remaining Proofs for Section \ref{subsubsec: universallower} : $\Omega(n\log n)$ lower bound}
\label{subsubsec:univproof}

%%%%%%%%%%%%%%%%%%%%%%%%%%%%%%%%%%%%%%%%%%%

%%%%%%%%%%%%%%%%%%%%%%%%%%%%%%%%%%%%%%%%%%%%%%%

\paragraph{Proof of Corollary \ref{cor:star}}
\begin{proof}
	The spectral properties of star (see example \ref{eg: star, t12}) shows the mixing time is between $\Theta(n)$ and $\Theta(n\log n)$ for every fixed $\epsilon$. Theorem \ref{thm:genlb} gives a universal $\Omega(n\log n)$ lower bound. Therefore we conclude $t_{\epsilon,1}(S_{n-1}) = \Theta_\epsilon(n\log n)$.
\end{proof}

\paragraph{Proof of Corollary \ref{cor: fastest L1}}
\begin{proof}{Proof of Corollary \ref{cor: fastest L1}}
	It is known in Chatterjee et al. \cite{chatterjee2019phase} that for every $2 > \epsilon >0$, 
	$$
	\frac{(2\log2) t_1(\epsilon, K_n)}{n\log(n/2)} \rightarrow 1.
	$$
	
	Meanwhile, we have $t_1(\epsilon, G_n) \geq \frac{(1-\epsilon)}{2\log 2}n\log(n) - O(n)$ from Theorem \ref{thm:genlb}. Combining the two results together and Corollary \ref{cor: fastest L1} immediately follows. 
\end{proof}

\begin{comment}
\paragraph{Proof of Corollary \ref{cor: fastest L1}}
\begin{proof}
It is known in Chatterjee et al. \cite{chatterjee2019phase} that for every $2 > \epsilon >0$, 
$$
\frac{(2\log2) t_1(\epsilon, K_n)}{n\log(n/2)} \rightarrow 1.
$$

Meanwhile, we have $t_1(\epsilon, G_n) \geq \frac{(1-2\epsilon)}{2\log 2}n\log(n)$ from Theorem \ref{thm:genlb}. Combining the two results together and Corollary \ref{cor: fastest L1} immediately follows. 
\end{proof}
\end{comment}

\subsubsection{Remaining Proofs for Section \ref{subsubsec:improved bound} : $\alpha$-covering, flow, comparison, and splitting}
\paragraph{Proof of Proposition \ref{prop: t_1, alternative bound}}\label{para: proof covering time}
\begin{proof}
	Suppose one starts at $v(0) = e_i$, and let $T_\text{cov}(1-\epsilon, i)$ be the first time $v(t)$ has $(1-\epsilon) n$ non-zero elements. Then for any $t \leq T_\text{cov}(1-\epsilon, i)$, we have:
	\[
	\lVert v(t) - \bar v\rVert_1  = \lVert v(t) - \frac 1n \mathbf{1} \rVert_1  \geq \epsilon n \cdot \frac{1}{n} = \epsilon,
	\]
	which implies
	\[
	t_{\epsilon,1} \geq \bE[T_\text{cov}(1-\epsilon, i)] = \tcov(1-\epsilon, i).
	\]
	
	Since $i$ is arbitrary, we conclude $t_{\epsilon,1} \geq \tcov(1-\epsilon)$. 
\end{proof}
\paragraph{Proof of Proposition \ref{prop: nearly regular graph, covering time}}\label{para: proof near regular}
\begin{proof}
	Without loss of generality we assume the process is initialized at $e_1 = (1,0,\cdots, 0)$. Let $n_0 \equiv (1-\epsilon)n$. Let $W_i$ be the `waiting time' of the $(i+1)$-th non-zero coordinate given the vector has $i$ non-zero coordinate. Clearly $\tcov(1-\epsilon)(G_n) \geq \bE[W_1 + W_2 + W_3 + \cdots W_{n_0}]$. 
	
	Given the vector has $k$-non zero coordinates, the waiting time of $W_k$ follows a geometric distribution with success probability no larger than
	\[
	\frac{k M}{|E(G_n)|} \leq 2\,C\, \frac{k}{n} 
	\]
	since there are at most $kM$ edges which connects between the zero and non-zero elements. Therefore, $\bE[W_k]\geq \frac{n}{2Ck}$ for every $k\leq n_0$, and summing up over $\bE[W_k]$ yields
	\[
	\tcov(1-\epsilon)(G_n) \geq \sum_{k = 1}^{n_0} \frac{n}{2Ck} \geq \frac{n}{C}\log(n_0) = \frac{n}{2C}\log\left((1-\epsilon)n\right).
	\]
\end{proof}
\paragraph{Proof of Proposition \ref{prop:expander}}\label{para: proof of expander example}
\begin{proof}
	It follows from the definition of the expander graph that $\gamma(G_n) = \Theta(n)$, and therefore $t_{\epsilon,1}(G_n)$ is between $\Theta(n)$ and $\Theta(n\log n)$. Meanwhile, it follows from Proposition \ref{prop: nearly regular graph, covering time} that $t_{\epsilon,1}(G_n)$ is $\Omega(n\log n)$. Combining the two facts together gives us the desired result. 
\end{proof}

\paragraph{Proof of Theorem \ref{thm: dumbbell}}\label{para:dumbbell}
\begin{proof}
	Let $D_{n}$ be a Dumbbell graph on the node set 
	$\{1,\ldots, n\} \cup \{n + 1,\ldots, 2n\}$.
	The edge set of $D_{n}$ is $\{(i,j), (n+i,n+j)\mid \; 1\le i < j\le n\}$, plus a special edge $e=(n,2n)$. 
	A lower bound of $O(n^{3})$ is immediately  
	implied by the spectrum, but it is not necessary to refer
	to this, because as soon as we understand how the typical
	process takes place, both the lower and upper bounds follow.
	Below we sketch only the proof of the $\Omega(n^3)$
	upper bound.
	
	Let the clique on
	$\{1,\ldots, n\}$ be denoted by $L$, and the clique
	on $\{n + 1,\ldots, 2n\}$ be denoted by $R$ (for ``left'' and ``right''). We also call edge $e=(n,2n)$ a {\em bridge}. 
	We shall decompose
	the edge sequence of the process to sub-sequences,
	separated by those events when $e$ is picked. 
	Let $S_{1},S_{2},$ etc. denote these sub-sequences, so the 
	process looks like $S_{1}eS_{2}e\ldots$.
	We call the run of each $S_{i}$ a {\em phase},
	while averaging over the edge $e$ an {\em equalization step}.
	
	Let $M_{i}$ be the maximum and $m_{i}$ the minimum
	of the state vector right before the $i^{th}$ phase. We work with the assumption that $|v_{0}|_{1} = 1$, and $v_{0}$ is non-negative, so $1\ge M_{1}\ge m_{1}\ge 0$. We shall prove that for any fixed $\epsilon>0$ 
	there is a $K_{\epsilon} = O(n)$, 
	that after any period of $K_{\epsilon}$ 
	subsequent phases ending with the $\ell^{\rm th}$, with very high
	probability $M_{\ell}-m_{\ell}<{\epsilon\over n}$ 
	(independently of the prior $\ell-K_{\epsilon}$ phases). 
	We leave it to the avid reader to recognize that this 
	together with the phase length distribution (given below in item 1.) 
	implies the theorem.
	Let $\Delta_{i} = M_{i}-m_{i}$.
	Since $\Delta_{i}$ never increases, it is sufficient to show, that as long as $\Delta_{i} >{\epsilon\over n}$,
	with very high probability (0.5 is in fact sufficient)
	$\Delta_{i+2}$ reduces to below
	$\min\{2/n,\; \Delta_{i} -
	\Omega(1/n^{2})\}$, independently of 
	the history before the $i^{\rm th}$ phase, i.e. of phases from 1 to $i-1$. Note: All $\Omega$, $O$, $\Theta$ notations
	depend on $\epsilon$. 
	We again leave it to the insightful reader that the above fact proves the fact before. (It is interesting to note, that in the first phase
	the progress will be very quick: $\Delta_{2}$ is likely 
	$\le 10/n$. Then additional $O(n)$ phases 
	are needed to get $\Delta_{i}$
	to below $\epsilon/n$.)
	Observe:
	\begin{enumerate}
		\item The length of each $S_{i}$ follows a geometric distribution:
		\begin{equation}\label{eq:lengthsi}
			\bP(|S_{i}|=k) =\left(1 - {1\over n(n-1)+1}\right)^{-k}
			{1\over n(n-1)+1}.
		\end{equation}
		\item Each $S_{i}$ consists of moves 
		$S_{i}^{L}$ made on $L$ and $S_{i}^{R}$ made on $R$.
		\item Saying it differently, we recover the usual
		distribution on edge sequences for the dumbbell graph if for each $1\le i\le \infty$ we construct $S_{i}$,
		$S_{i}^{L}$, $S_{i}^{R}$ as follows:
		(i.) We first randomly decide at the length ${\cal L}_i$ of $S_{i}$ according to Formula (\ref{eq:lengthsi}).
		(ii.) We create a random $LLRRL\ldots$ sequence 
		of length ${\cal L}_i$, made from symbols `$L$'s and `$R$'s, and (iii.) We replace
		every symbol `$L$' with a random edge from $L$,
		and every symbol `$R$' with a random edge from $R$.
		\item Due to the previous construction,
		$S_{i}^{L}$ and $S_{i}^{R}$ are independent moves 
		on $L$ and $R$, respectively. Therefore
		we can apply the analysis for the clique graph when
		determining their effects.
		
		\item $\EE(|S_{i}|) = n(n-1)$, and for large enough $n$:
		\begin{equation}\label{eq:dum2}
			\bP\left[\underbrace{|S_{i}^{L}|< {n(n-1)\over 1000}\;\; \wedge \;\; |S_{i}^{R}|< {n(n-1)\over 1000}}_{{\rm event}\; A}\right] < 0.01
		\end{equation}
	\end{enumerate}
	Let $M_{i}^{L}$, $m_{i}^{L}$, $M_{i}^{R}$, $m_{i}^{R}$
	be the maximums and minimums on $L$ and $R$ before the $i^{\rm th}$ phase.
	Conditioned on that $|S_{i}^{L}|,|S_{i}^{R}|
	\ge {n(n-1)\over 1000}$ one can show from the clique result, that after running $S_{i}$, due to the very quick convergence 
	on the clique to the uniform vector: 
	\begin{equation}\label{eq:maindum}
		\bP\left[\underbrace{\max\{M_{i}^{L} - m_{i}^{L},\; M_{i}^{R} - m_{i}^{R}\}\ge 
			\epsilon/100 n^{2}}_{{\rm event}\; B}|  A^c\right]\le 0.01.
	\end{equation}
	When event $B$ does not happen, we have $\Delta_{i+1}< 2/n$. Further, still under $B$ does not happen,
	and also assuming $\Delta_{i} >{\epsilon\over n}$, during the equalization step following $S_{i}$, an amount $\ge {\epsilon\over 10n}$
	flows through $e$ from $L$ to $R$, if
	$M_{i}^{L}> M_{i}^{R}$, otherwise from 
	$R$ to $L$. Then applying
	Estimate (\ref{eq:maindum}), but at this time for $S_{i+1}$,
	we get the reduction $\Delta_{i+2}\le \Delta_{i} -
	\frac{\epsilon}{20n^2}$ with high probability.
\end{proof}
\paragraph{Proof for the cycle graph}\label{para:cycle}
Now we claim the following comparison lemma between the splitting process defined in Section \ref{para:comparison splitting}.

\begin{proof}[Proof of Lemma \ref{lem:comparison cycle}]
	We start with  the second inequality. It is clear that each splitted sequence $v^i(t)$ is still non-increasing by design of the process. Therefore we have $$\lVert v^i(t) - \bar{v^i}(t) \rVert_1 \leq  n(v^i_1(t) - v^i_n(t)).$$ Summing both sides from $i$ to $N_t$ gives the second inequality.
	
	We prove the first inequality by induction. When $t = 1$, the inequality follows immediately from the triangle inequality. Let $T(t)$ be the $L^1$ distance between the vector $v(t)$ and uniform vector  of the original process after step $t$, and $\tilde T(t)$ be the same quantity of the splitting process. Suppose the inequality $\bE[T(t)]\leq \bE[\tilde T(t)]$ is true for every $t\leq (s-1)$ and every non-increasing initialization. The $t = s$ case can be analyzed by a first-step analysis as follows.
	
	If the edge $(n, 1)$ is not chosen in the first step, then both processes will evolve in the same way, and $v(1)$ is still non-increasing, thus the inequality follows from the induction hypothesis. 
	
	If the edge $(n,1)$ is chosen, and the original process  has conditional expectation:
	% \begin{align*}
	%     \bE\left[T_v(t)|(n,1)~ \text{is chosen}\right] & = \bE\left[T_{((v_1 + v_n)/2,v_2, \cdots, v_{n-1},(v_1 + v_n)/2 )}(t-1)\right]\\
	%     &=  \bE\left[T_{((v_1 + v_n)/2,(v_1 + v_n)/2, v_2, v_3,\cdots, v_{n-1})}(t-1)\right],
	% \end{align*}
	\begin{align*}
		\bE\left[T_v(t)|(n,1)~ \text{is chosen}\right] & = \bE\left[T_{\left(\frac{v_1 + v_n}{2},v_2, \dots, v_{n-1},\frac{v_1 + v_n}{2} \right)}(t-1)\right]\\
		&=  \bE\left[T_{\left(\frac{v_1 + v_n}{2},\frac{v_1 + v_n}{2}, v_2, v_3,\dots, v_{n-1}\right)}(t-1)\right],
	\end{align*}
	where the last equality follows from the cycle structure, as the initializations are equivalent up to a circular shift.
	
	Suppose $(v_1 + v_n)/2\geq v_2$, then the sequence $\left((v_1 + v_n)/2,(v_1 + v_n)/2, v_2, \dots, v_{n-1}\right)$ is already non-increasing, which means $N_1 = 1$ and $\tilde v(1)$ equals $\left((v_1 + v_n)/2,(v_1 + v_n)/2, v_2, \dots, v_{n-1}\right)$, thus the first inequality also follows from the induction hypothesis. 
	
	Suppose $(v_1 + v_n)/2 < v_2$, then let $k$ denotes the largest index such that $v_k > (v_1 + v_n)/2$. For the original process we have
	% \begin{align*}
	%     \bE\left[T_v(t)|(n,1)~ \text{is chosen}\right]  =& \bE\left[T_{((v_1 + v_n)/2,(v_1 + v_n)/2, v_2, \cdots, v_{n-1})}(t-1)\right] \\
	%     \leq & \bE\left[ T_{((v_1 + v_n)/2,\cdots, (v_1 + v_n)/2, v_{k+1}, \cdots, v_{n-1})}(t-1)\right]  \\
	%     & + \bE\left[T_{(0,0, v_2 - (v_1 + v_n)/2\cdots, v_k - (v_1 + v_n)/2, 0, \cdots 0)}(t-1)\right] \\
	%     \leq & \bE\left[\tilde T_{\tilde{v}^1(1)}(t-1)\right]   + \bE\left[\tilde T_{\tilde{v}^2(1)}(t-1)\right] = \bE\left[\tilde T_v(t)|(n,1)~ \text{is chosen}\right],
	% \end{align*}
	\begin{align*}
		\bE\left[T_v(t)|(n,1)~ \text{is chosen}\right]  &= \bE\left[T_{\left(\frac{v_1 + v_n}{2},\frac{v_1 + v_n}{2}, v_2, v_3,\dots, v_{n-1}\right)}(t-1)\right] \\
		&\leq  \bE\left[ T_{\left(\frac{v_1 + v_n}{2},\dots, \frac{v_1 + v_n}{2}, v_{k+1}, \dots, v_{n-1}\right)}(t-1)\right]  \\
		& + \bE\left[T_{\left(0,0, v_2 - \frac{v_1 + v_n}{2},\dots, v_k - \frac{v_1 + v_n}{2}, 0, \cdots 0\right)}(t-1)\right] \\
		&\leq  \bE\left[\tilde T_{\tilde{v}^1(1)}(t-1)\right]   + \bE\left[\tilde T_{\tilde{v}^2(1)}(t-1)\right] = \bE\left[\tilde T_v(t)\,:\,(n,1)~ \text{is chosen}\right],
	\end{align*}
	where the first inequality follows from the fact that
	\begin{align*}
		\left(\frac{v_1 + v_n}{2},\frac{v_1 + v_n}{2}, v_2, \dots, v_{n-1}\right) = & \left(\frac{v_1 + v_n}{2}, \frac{v_1 + v_n}{2},\dots, \frac{v_1 + v_n}{2}, v_{k+1}, \dots, v_{n-1}\right)\\ 
		& + \left(0,0, v_2 - \frac{v_1 + v_n}{2}, \dots, v_k - \frac{v_1 + v_n}{2}, 0, \dots,0\right)
	\end{align*}
	and then use the triangle inequality. The second equality uses the fact that the initialization $(0,0, v_2 - \frac{v_1 + v_n}{2}, \cdots, v_k - \frac{v_1 + v_n}{2}, 0, 0, \cdots,0)$ is equivalent to the initialization \\ $(v_2 - \frac{v_1 + v_n}{2}, \cdots, v_k - \frac{v_1 + v_n}{2}, 0, 0, \cdots,0)$ because one is a circular shift of the other. The last inequality follows from the induction hypothesis, and the fact that
	\[
	\tilde v^1(1) = \left(\frac{v_1 + v_n}{2}, \frac{v_1 + v_n}{2},\dots, \frac{v_1 + v_n}{2}, v_{k+1}, \dots, v_{n-1}\right)^T
	\]
	and
	\[
	\tilde v^2(1) = \left(v_2 - \frac{v_1 + v_n}{2}, \dots, v_k - \frac{v_1 + v_n}{2}, 0,  \dots,0\right)^T
	\]
	are both non-increasing. Now we have exhausted all the possible cases for the first step, and therefore  we conclude that $\bE[T_v(t)] \leq  \bE[\tilde T_v(t)]$ for every $t$.
\end{proof}

The first inequality of the comparison lemma shows the splitting process `converges' slower than the averaging process, whereas the second inequality shows it suffices to bound the probability of  $\bP[\tilde v_1(t) - \tilde v_n(t)\geq \frac{\epsilon}{n}]$ in order to bound $\bE[\tilde T_v(t)]$.

Now we are ready to bound the $L^1$ convergence for the averaging process on the cycle $C_n$. Spectral arguments prove it takes $\Theta(n^3)$ to $\Theta(n^3\log n)$ steps for $L^1$ convergence, and we are trying to show $\Theta(n^3)$ is the correct magnitude. 

We assume $v = e_1 = (1,0, \cdots,0)\in \bR^n$ is the initial vector. Note that assuming $e_1$ is the initialization entails no loss of generality, as the slowest initialization happens on a corner, but every corner has the same convergence speed for a cycle graph. 

We will use the quantity $\tilde Q(t)$ to bound $\bE[\tilde T_v(t)]$ for the splitting process $PC_2$. Our next result shows $\tilde Q(t)$ is monotonically non-increasing with $t$.

\begin{proof}[Proof of Lemma \ref{lem: monotonicity Q(t)}]
	Since the splitting process happens independently for every splitted sequence, it suffices to show for any fixed non-increasing sequence $a = (a_1, a_2,\cdots a_n)$, after one step of the splitting process, the quantity $\tilde Q$ always decays, and in expectation decays no less than $(a_1 -  a_n)/2n$. Suppose edge $(1,n)$ is not chosen, it follows from straightforward calculation that the quantity $Q$ will decrease (since $Q(a_1, \cdots, a_n) = a_1 + (a_1 + a_2) + \cdots + (a_1 + a_2 + \cdots + a_n)$, and each term will not increase after averaging $a_i$ and $a_{i+1}$). If $(1,n)$ is chosen and $(a_1 + a_n)/2 \geq a_2$, then the new vector becomes $((a_1 + a_n)/2, (a_1 + a_n)/2, a_2, \dots, a_n)$ and we can directly check   
	% \begin{align*}
	%     Q\left((a_1 + a_n)/2, (a_1 + a_n)/2, a_2, \dots, a_{n-1}\right) &- Q\left(a_1, a_2,\dots, a_n\right)\\ &= -a_1/2  - a_2 - a_3 - \dots - a_{n-1} + (n - 3/2) a_n\\
	%     & = (a_n - a_1)/2 + (a_n - a_2) + (a_n - a_3) + \dots + (a_n - a_{n-1})\\
	%     &\leq 0.
	% \end{align*}
	\begin{align*}
		Q\left(\frac{a_1 + a_n}{2}, \frac{a_1 + a_n}{2}, a_2, \dots, a_{n-1}\right) &- Q\left(a_1, a_2,\dots, a_n\right)\\ &= -\frac{a_1}{2}  - a_2 - a_3 - \dots - a_{n-1} + \left(n - \frac{3}{2}\right) a_n\\
		& = \frac{a_n - a_1}{2} + (a_n - a_2) + (a_n - a_3) + \dots + (a_n - a_{n-1})\\
		&\leq 0\;.
	\end{align*}
	
	Otherwise, again let $k$ be the largest index such that $a_k > (a_1 + a_n)/2$, then the sequence will be splitted into two, and the two sequences add up to
	\[
	\tilde a = \left(a_2, a_3, \dots, a_k, \frac{a_1+ a_n}{2}, \frac{a_1+ a_n}{2}, a_{k+1}, \dots, a_{n-1}\right),
	\]
	and we have
	\begin{align*}
		Q(\tilde a) - Q(a) < 0
	\end{align*}
	for the same reason.  Since $\tilde v(t)$ is the summation of $N_t$ sequences, and $Q$ is a linear function, we have $\tilde Q(t) = \sum_{i=1}^{N_t} Q(v^i(t))$, as the $Q$ function decays on each individual sequence, we know $\tilde Q(t)$ decays with $t$. 
	
	For the second inequality, still we focus on one fixed sequence $(a_1, \cdots, a_n)$, after one step of the splitting process, it has $\frac{n-1}{n}$ probability of not choosing $(1,n)$, which contributes an average decay of at least (it is at least because we ignore the contribution when choosing $(1,n)$, which is non-negative)
	% \[
	% \frac{1}{n} \left( (a_1 - a_2)/2  + (a_2 - a_3)/2 + \cdots +(a_{n-1} - a_n)/2\right) = \frac{a_1 - a_n}{2n}.
	% \]
	\[
	\frac{1}{n} \left(\frac{a_1 - a_2}{2} + \frac{a_2 - a_3}{2}  + \cdots +\frac{a_{n-1} - a_n}{2} \right) = \frac{a_1 - a_n}{2n}.
	\]
	Again, given the previous state $\tilde v(t-1)$, summing up all the individual sequences $\tilde v^i(t-1)$ from $i =1$ to $N_{t-1}$ yields the second inequality.
\end{proof}
Finally, we are ready to show our main result.

\begin{proof}[Proof of Theorem \ref{thm:cycle}]
	It suffices to show $t_{\epsilon,1}(C_n)\leq C_\epsilon n^3$. Let $E(t)$ be the event that $\{\tilde v_1(t) - \tilde v_n(t) \geq \frac{\epsilon}{2n}\}$. By design we have $E(t) \subset E(t-1) \subset E(t-2)\cdots $, as $\tilde v_1(t)$ is non-increasing with $t$ and $\tilde v_n(t)$ is non-decreasing with $t$. Under the complement of $E(t)$, we know $\tilde T_v(t) \leq \frac{\epsilon}{2}$, under $E(t)$ we know $\tilde T_v(t)$ is no larger than its maximum possible value $2$, as $\tilde v(t)$ is a non-negative vector which sum up to $1$. 
	Therefore, by the comparison lemma we have
	\begin{align*}
		\bE[T_v(t)] \leq \bE[\tilde T_v(t)] \leq \frac{\epsilon}{2} + 2 \bP(E(t)).
	\end{align*}
	To bound $\bP(E(t))$, we have:
	\begin{align*}
		n &= \bE[\tilde Q(1)] \geq \sum_{i=1}^\infty \bE[\tilde Q(i) - \tilde Q(i+1)] \geq \sum_{i=1}^{t} \bE[\tilde Q(i) - \tilde Q(i+1)] \\
		&\geq t \frac{\epsilon}{4n^2}\bP[E(t)].
	\end{align*}
	Take $t \geq \frac{16n^3}{\epsilon^2}$ and we conclude $\bE[T_v(t)] \leq \epsilon$, thus $t_{\epsilon,1}(C_n)  \leq \frac{16n^3}{\epsilon^2}$.

\end{proof}

\subsection{Proofs for Section \ref{subsec:L2->1}}
\subsubsection{Proof of Proposition \ref{prop: comparison t_1 and t_12}}\label{subsubsec: proof comaprson t_1 and t_12}
\begin{proof}
	For every vector $v$ with $\lVert v \rVert_1 = 1$, we have $\lVert v \rVert _2\leq 1$. Therefore if we define
	\[
	t_{\epsilon}(v) \equiv \min\{t\in \ZZ: {\bE[\lVert v(t) - \bar v\rVert_1]} \leq \epsilon\}.
	\]
	Then by our previous observation and Proposition \ref{prop:translation and scaling} we have $
	t_{\epsilon}(v) \leq t_{\epsilon}(v/\lVert v \rVert_2) $,
	which implies $\bE[t_\epsilon(v)] \leq  t_{\epsilon,2\rightarrow 1}$
	for every $v$ with $\lVert v \rVert_1 = 1$. Then by definition of $t_{\epsilon,1}$ we conclude $
	t_{\epsilon,1}\leq t_{\epsilon,2\rightarrow 1}$.
\end{proof}

\subsubsection{Proof of Proposition \ref{prop: L21 convergence speed}}\label{subsubsec: proof of L21 speed}
\begin{proof}
	For the upper bound, Theorem \ref{Thm:L2_convg} and Cauchy-Schwarz give:
	% \begin{align*}
	%  \bE[\lVert v(t) - \bar v\rVert_1] &=   \bE \left[ \sum_{i=1}^n\lvert v_i(t) - \bar v_i\rvert\right] \;  \leq \; \sqrt n \; \bE \left[\sqrt{\sum_{i=1}^n \lvert v_i(t) - \bar v_i \rvert^2 }\; \right] \\&\leq \sqrt{n} \; \sqrt{\bE[\lVert v(t) - \bar v\rVert_2^2]}  \leq \sqrt n\;  \bigg(1- {\lambda_{2} \over 2 |E|}\bigg)^{\frac t2}  \lVert v(0) - \bar v\rVert_2.
	% \end{align*}
	\begin{align*}
		\bE[\lVert v(t) - \bar v\rVert_1] &=   \bE \left[ \sum_{i=1}^n\lvert v_i(t) - \bar v\rvert\right] \;  \leq \; \sqrt n \; \bE \left[\sqrt{\sum_{i=1}^n \lvert v_i(t) - \bar v \rvert^2 }\; \right] \\&\leq \sqrt{n} \; \sqrt{\bE[\lVert v(t) - \bar v\rVert_2^2]}  \leq \sqrt n\;  \bigg(1- {1 \over 2 \gamma(G)}\bigg)^{\frac t2}  \lVert v(0) - \bar v\rVert_2.
	\end{align*}
	
	For the lower-bound, again we prove it by taking the initial condition to be the
	second eigenvector $v(0)=u_{2}$. Recall that by Eq.~\eqref{eq:E_v(t)} we
	have 
	\[
	\mathbb{E}[v(t)]=\left(1-\frac{\lambda_{2}}{2|E|}\right)^{t}u_{2}.
	\]
	This along with the fact that $\left\Vert u_{2}\right\Vert _{2}^{2}=1$
	give
	\begin{eqnarray*}
		\mathbb{E}[\left\Vert v(t)\right\Vert _{1}] & = & \sum_{i=1}^{n}\mathbb{E}\left[\lvert v_{i}(t)\rvert \right]\geq \lVert\bE[v(t)]\rVert_1\geq \lVert\bE[v(t)]\rVert_2 \\
		& = & \sqrt{\left(1-\frac{\lambda_{2}}{2|E|}\right)^{2t} \sum_{i=1}^{n}u_{2,i}^{2}}=\left(1-\frac{\lambda_{2}}{2|E|}\right)^{t}.
	\end{eqnarray*}
	The desired lower-bound then follows
	\[
	\sup_{\left\Vert v(0)\right\Vert _{2}=1}\mathbb{E}[\left\Vert v(t)-\bar{v}\right\Vert _{1}]\geq\left(1-\frac{1}{2\gamma(G)}\right)^{t}.
	\]
	%The last inequality follows from the fact that $\lVert x\rVert_p \leq \lVert x\rVert_q$ for every $x\in \bR^n$ and $p \geq q$. 
\end{proof}

\subsubsection{Proof of Corollary \ref{cor: t_12}}\label{subsubsec: proof of t_12}
\begin{proof}
	According to Proposition \ref{prop: L21 convergence speed}, for starting vectors $v(0)$ 
	with $\lVert v(0) \rVert_{2}= 1$, it holds that
	\begin{equation}\label{start2}
		\bE[\lVert v(t) - \bar v\rVert_1] \; \leq \; \bigg(1- {1 \over 2 \gamma(G)}\bigg)^{\frac t2}  \sqrt{n}  \lVert v(0) - \bar v\rVert_2 \; \leq \; \exp\left(  \frac{-t}{4\gamma(G)}\right)\sqrt{n}
	\end{equation}
	where the second inequality uses the fact that $\lVert v(0) - \bar v\rVert_2^2 + \lVert \bar v\rVert_2^2 = \lVert v(0) \rVert_2^2 = 1$. 
	Solving the equation 
	$\exp\left(  \frac{-t}{4\gamma(G)}\right)\sqrt{n} \; = \; \epsilon$
	we get
	$
	t \; = \;4\gamma(G)\log (\sqrt n \epsilon^{-1}) .$
	This gives us the upper-bound in Eq.~\eqref{eqn: t_12, first formula}. The lower bound follows directly from the same calculation and the lower bound in Proposition \ref{prop: L21 convergence speed}.
\end{proof} 

\subsubsection{Proof of Theorem \ref{thm: t_12,delocalize}}\label{subsubsec:t12,delocalize}
\begin{proof}
	The upper bound  is precisely the same as Corollary \ref{cor: t_12}, thus it suffices to show the lower bound. Again, set 
	$v(0) = u_{2} $  where $u_2$ is the unit eigenvector of $L(G)$ corresponding to $\lambda_2$. Recall that
	$\bE[v(t)] = M^t v(0)$, which gives:
	\begin{align}\nonumber
		\bE[\lVert v(t)\rVert_1]  \; &\geq \; 
		\sum_{i=1}^n \bigg|\bE [ v_i(t)] \bigg| 
		= \;  \lVert M^{t} v(0) \rVert_1
		= \; 
		\bigg(1- {1 \over 2 \gamma(G)}\bigg)^t \lVert u_2 \rVert_1\\
		&\geq \delta \;\sqrt n\; \exp\left[-\frac {t}{2\gamma(G)}/ \left(1 - \frac {1}{2\gamma(G)}\right)\right],
	\end{align}
	where the last inequality uses the fact that $u_2$ is $\delta$-delocalized. 
	Solving the equation 
	\[
	\delta \;\sqrt n\; \exp\left[-\frac {t}{2\gamma(G)}/ \left(1 - \frac {1}{2\gamma(G)}\right)\right] = \epsilon
	\]
	gives the desired result.
\end{proof}

\subsubsection{Proof of the examples in Table \ref{tab:L21}}\label{subsubsec: example L21}
Now we apply Theorem \ref{thm: t_12,delocalize} on the concrete examples discussed before to get the correct magnitude of $t_{\epsilon,2\rightarrow 1}$.

\begin{eg}[Complete graph, continued]\label{eg: complete, t12}
	As shown in Example \ref{eg: complete, L2}, the complete graph $K_n$ has $\lambda_2 = n$ and thus $t_{\epsilon,2}(K_n) = \Theta(n)$. Applying Corollary \ref{cor: t_12} shows $t_{\epsilon,2\rightarrow 1}$ is between $\Theta(n)$ and $\Theta(n\log n)$. Meanwhile, every vector $u\in \bR^n$ with $\sum_{i = 1}^n u_i  = 0$ is an eigenvector of $L(K_n)$ with eigenvalue $n$. Therefore, we may choose $u_2 = (\frac 1{\sqrt n}, \cdots, \frac 1{\sqrt n}, -\frac 1{\sqrt n}, \cdots, -\frac 1{\sqrt n})$ when $n$ is even and $u_2 = (\frac 1{\sqrt {n-1}}, \cdots, \frac 1{\sqrt {n-1}}, -\frac 1{\sqrt {n-1}}, \cdots, -\frac 1{\sqrt {n-1}}, 0)$ when $n$ is odd. In either case $u_2$ is at least $\sqrt{\frac{n-1}{n}}$-delocalized. Therefore, applying Theorem \ref{thm: t_12,delocalize} yields
	
	\begin{align*}
		\frac{n-2}{2} \log((n-1)\epsilon^{-2})\leq t_{\epsilon,2\rightarrow 1}(K_n) \leq (n-1)\log(n \epsilon^{-2}),
	\end{align*}
	which shows $\Theta_\epsilon(n\log n)$ is the correct magnitude of $t_{\epsilon,2\rightarrow 1}(K_n)$. 
\end{eg}

\begin{eg}[Cycle graph, continued]\label{eg: cycle, t12}
	As shown in Example \ref{fig:C_5}, the cycle graph $C_n$ has\\ $\lambda_2 = 2 - 2\cos(\frac{2\pi}{n})$ and $t_{\epsilon,2} = \Theta_{\epsilon}(n^3)$. Applying Corollary \ref{cor: t_12} shows $t_{\epsilon,2\rightarrow 1}(C_n)$ is between $\Theta_{\epsilon}(n^3)$ and $\Theta_{\epsilon}(n^3\log n)$. 
	Meahwhile, the Laplacian $L(C_n)$ can be diagonalized explicitly. In particular, the normalized eigenvector of $\lambda_2$ is $u_2 = \sqrt{\frac{2}{ n}} (\cos(w), \cos(2w), \cdots, \cos(n w))^T$ where $w = \frac{2\pi}{n}$. Now we calculate the $L^1$-norm of $u_2$, we have: 
	\begin{align*}
		\sqrt{\frac{2}{ n}} \lVert u_2 \rVert_1 =  \sum_{k = 1}^n |\cos(kw)| \geq 4\sum_{k = 1}^{\lfloor n/4 \rfloor} \cos(kw)\;.
	\end{align*}
	Using the identity $
	\cos(w) + \cos(2w) + \cdots + \cos(lw) = \frac{\sin(\frac {lw}{2}) \cos(\frac{lw +w}{2})}{\sin(\frac{w}{2})}$
	we conclude \\$4\sum_{k = 1}^{\lfloor n/4 \rfloor} \cos(kw) = \Theta(n)$ since in the denominator $\sin(\frac w2) = \Theta(\frac 1n)$ and the numerator is $\Theta(1)$. Therefore, there exists some universal constant $\delta$ such that $\lVert u_2 \rVert_1 \geq \delta \sqrt n \lVert u_2 \rVert_2$. Then we  apply Theorem \ref{thm: t_12,delocalize} and conclude that $\Theta_\epsilon(n^3 \log n)$ is the correct magnitude of  $t_{\epsilon,2\rightarrow 1}(C_n)$.
\end{eg}

\begin{eg}[Star graph, continued]\label{eg: star, t12}
	As shown in Example \ref{eg: star, L2}, the star graph $S_{n-1}$ with $n$ nodes has $\lambda_2 = 1$ and $t_{\epsilon,2} = \Theta(n)$. Therefore Corollary \ref{cor: t_12} implies $t_{\epsilon,2\rightarrow 1}(S_{n-1})$ is between $\Theta(n)$ and $\Theta(n\log n)$. Meanwhile, if we label the root node by $1$ and the remaining leaves by $2, \cdots, n$. One can diagonalize $L(K_n)$, and observe that every vector $u \in \bR^n$ which satisfies $u_1 = 0$ and $\sum_{i=2}^n u_i = 0$ is  eigenvector of $L(K_n)$ with respect to $\lambda_2 = 1$. Therefore, we may choose $u_2 = \frac 1{\sqrt {n-1}}(0,1, \cdots, 1, -1, \cdots, -1)$  when $n$ is odd and $u_2 = \frac 1{\sqrt {n-2}}(0,0,1, \cdots, 1, -1, \cdots, -1)$ when $n$ is even. In both cases $u_2$ is at least $\sqrt{\frac{n-2}{n}}$-delocalized. Therefore, applying Theorem~\ref{thm: t_12,delocalize} yields
	\[
	\frac{2n-3}{2} \log((n-2)\epsilon^{-2})
	\leq   t_{\epsilon,2\rightarrow 1}  \leq  (2n-2) \log(n\epsilon^{-2}),
	\]
	which shows $t_{\epsilon,2\rightarrow 1}(S_{n-1}) = \Theta_\epsilon(n\log n)$.
\end{eg}

Though Theorem \ref{thm: t_12,delocalize} is successfully applied in the above examples to get the correct magnitude of $t_{\epsilon,2\rightarrow 1}$, it requires the precise knowledge of the second eigenvector $u_2$, which is often not available in practical settings. In the next example, we show an alternative way to derive the correct magnitude of $t_{\epsilon,2\rightarrow 1}$ without using Theorem \ref{thm: t_12,delocalize}. Our idea is partially motivated by the proof of the Cheeger's inequality \citep{cheeger2015lower}.

\begin{eg}[Binary tree, continued]\label{eg: binary tree, t12}
	As shown in Example \ref{eg: binary tree, L2}, a balanced full binary tree $B_n$ has $n-1$ nodes, $\lambda(B_n) = \Theta(\frac 1n),$ and $t_{\epsilon,2} = \Theta(n^2)$. Therefore, Corollary \ref{cor: t_12} implies $t_{\epsilon,2\rightarrow 1}(B_{n})$ is between $\Theta_\epsilon(n^2)$ and $\Theta_\epsilon(n^2\log n)$.  Unlike the previous examples, we do not know the explicit structure of $u_2$, therefore it require new tools. We claim $t_{\epsilon,2\rightarrow 1}(B_{n}) = \Theta_\epsilon(n^2\log n)$.
\end{eg}

\begin{proof}[Proof of the claim in Example \ref{eg: binary tree, t12}: $t_{\epsilon,2\rightarrow 1}(B_{n}) = \Theta_\epsilon(n^2\log n)$]
	Let us first label the binary tree as follows. We label the root by $1$, and label the nodes of the left subtree by $2, 3, \cdots, \frac{n+1}{2}$, and the nodes in the right subtree by $\frac{n+3}{2}, \cdots, n$. We further notice that a binary tree with $n$ nodes has $\log_2(n+1)$ levels, and we label the levels by $1, \cdots, k = \log_2(n+1)$, from bottom to top. 
	
	Now let us consider the following unit vector $v(0) = \frac{1}{\sqrt {n-1}}(0,1, \cdots, 1, -1,\cdots, -1)$. In other words, the value of the root node equals $0$, the value of every node in the left subtree equals $\frac {1}{\sqrt{n-1}}$, and the value of every node in the right subtree equals $-\frac{1}{\sqrt{n-1}}$. 
	
	Let $v(0)$ be the initial vector for the averaging process, and we are going to prove it takes at least $\Theta(n^2\log(n))$ steps for the process to mix in $L^1$ distance. We claim the following properties of this process starting with $v(0)$:
	
	\begin{itemize}
		\item For every $t$, $\bE[v_1(t)] = 0$.
		\item For every $t$, $\bE[v_j(t)] \geq 0$, if $2\leq j \leq \frac{n+1}{2}$, and $\bE[v_j(t)] \leq 0$ if $j > \frac{n+1}{2}$.
		\item For every $t$, let $j$ be any node in the left subtree (i.e.,~$j\leq \frac{n+1}{2}$), we have $\bE[v_2(t)] \leq \bE[v_j(t)]$. Similarly, let $j$ be any node in the right subtree (i.e.,~$j\geq \frac{n+3}{2}$), we have $\bE[v_{\frac{n+3}{2}}(t)] \geq \bE[v_j(t)]$.
	\end{itemize}
	Roughly, we claim that the root vertex always has expectation $0$, the left subtree always has positive expectation while the right subtree always has negative expectation. In particular, the top node (node $2$) in the left subtree has the smallest (but positive) expectation  among all the nodes in the left, while the top node (node $\frac{n+3}{2}$) on the right subtree has the largest (but negative) expectation among all the nodes in the right. 
	
	All of the above properties can be proved by induction, and we will only sketch the proof here. 
	
	It is clear that every node in the left subtree has the same expectation as long as they are at the same level. Let $l^t_s$ be the expectation of the nodes in the $s$-th level of the left subtree after $t$ steps, and likewise $r^t_s$ for expectation of the nodes in the $s$-th level of the right subtree. It is clear that we have $l^t_s = - r^t_s$. For $2\leq s \leq k-1$ (except for the bottom level), we have the recursion:
	\[
	l^t_s = \frac{2}{n-1} l^{t-1}_{s+1} + \frac{1}{n-1}l^{t-1}_{s-1} + \frac{n-4}{n-1}l^{t-1}_{s}
	\] 
	and for $s = k$, we have
	\[
	l^t_k = \frac{1}{n-1}l^{t-1}_{k-1} + \frac{n-2}{n-1}l^{t-1}_{k}.
	\]
	This recursion can be directly used to prove the first and second claim inductively.  The last claim can also be proved as we can show  $l_1^t \leq l_2^t \leq l_3^t \leq \cdots \leq l_k^t$ for every $t$ inductively. 
	
	Now we are ready to estimate the decay of $v(t)$ to $\bar v = \vec 0$ in $L^1$ metric. Let $M^t \equiv \sum_{i=2}^{\frac{n-1}{2}} v_i(t)$ be the sum of coordinates in the left subtree. Given $v(t-1)$, we know $M^t$ equals $M^{t-1}$ with probability $\frac{n-2}{n-1}$ (as long as the edge $(1,2)$ is not chosen), and if $(1,2)$ is chosen, $M^{t} = M^{t-1} - \frac{v_2(t-1) - v_1(t-1)}{2}$. In other words, we have
	$$
	\bE[M^t|v(t-1)] = M^{t-1} - \frac{1}{n-1} \frac{v_2(t-1) - v_1(t-1)}{2}\; .
	$$
	We can further take expectation with respect to $v(t-1)$ and get
	\begin{align}\label{eqn: binary tree, left sum}
		\bE[M^t] = \bE[M^{t-1}] - \frac{\bE[v_2(t-1)]}{2(n-1)} 
	\end{align}
	where the last equality uses $\bE[v_1(t)] = 0$. From our third claim, we further have:
	\[
	\bE[M^t] = \bE[M^{t-1}] - \frac{\bE[v_2(t-1)]}{2(n-1)}  \geq \bE[M^{t-1}] - \frac{\bE[M^{t-1}]}{(n-1)^2} = \left(1 - \frac{1}{(n-1)^2}\right) \bE[M^{t-1}]\; .
	\]
	
	Therefore, we know:
	\begin{align*}
		\bE\lVert v(t) - \bar v  \rVert_1 =  \bE\lVert v(t)  \rVert_1 \geq \bE[M^t] &\geq \left(1 - \frac{1}{(n-1)^2}\right)^t M^0\\ 
		&= \left(1 - \frac{1}{(n-1)^2}\right)^t \frac{\sqrt{n-1}}{2}\\
		&\geq \frac{\sqrt{n-1}}{2} e^{-\frac{t}{(n-1)^2 - 1}}\; .
	\end{align*}
	Therefore, solving the equation
	\[
	\frac{\sqrt{n-1}}{2} e^{-\frac{t}{(n-1)^2 - 1}}= \epsilon
	\]
	yields,
	\[
	t_{\epsilon,2\rightarrow 1} (B_n) \geq \frac{(n-1)^2 - 1}{2} \log\left(\frac{n-1}{4}\epsilon^{-2}\right),
	\]
	which shows $\Theta(n^2 \log n)$ is also the lower bound. Thus $t_{\epsilon,2\rightarrow 1} (B_n) = \Theta(n^2\log n)$. 
\end{proof}

\bibliography{arxiv_MSW.bib} 
% \newpage

\end{document}